\documentclass[11pt]{amsart}

\usepackage{mathrsfs} 
\usepackage{mathtools}
\usepackage{a4wide} 
\usepackage{graphicx} 
\usepackage{amssymb} 
\usepackage{epstopdf}
\usepackage{enumerate} 
\usepackage{titletoc} 
\usepackage[colorlinks=true, urlcolor=blue, linkcolor=blue, citecolor=black]{hyperref} %
\usepackage[nameinlink]{cleveref} 
\usepackage{esint} 
\usepackage{verbatim}
\usepackage{enumitem}
\usepackage{kotex}
\usepackage{tikz}
\usetikzlibrary{arrows.meta,calc}

\numberwithin{equation}{section}
\DeclareMathOperator*{\osc}{osc}

\theoremstyle{plain}
\newtheorem{theorem}{Theorem}[section]
\newtheorem{lemma}[theorem]{Lemma}
\newtheorem{corollary}[theorem]{Corollary}

\theoremstyle{definition}
\newtheorem{definition}[theorem]{Definition}
\newtheorem{remark}[theorem]{Remark}

\newtheorem*{notation}{Notation}

\makeatletter
\@namedef{subjclassname@2020}{\textup{2020} Mathematics Subject Classification}
\makeatother

\title[Fully nonlinear equations]{Boundary H\"older gradient estimates for fully nonlinear degenerate or singular parabolic equations}

\author{Hyungsung Yun}
\address{School of Mathematics, Korea Institute for Advanced Study, Seoul 02455, Republic of Korea}
\email{hyungsung@kias.re.kr}
  

\thanks{Hyungsung Yun has been supported by the KIAS Individual Grant (No. MG097801) at Korea Institute for Advanced Study.}

\begin{document}
\begin{abstract}
We study boundary regularity of viscosity solutions to fully nonlinear degenerate or singular parabolic equations. The gradient-dependent degeneracy or singularity, along with the time derivative, introduces significant challenges beyond the elliptic case. By combining compactness methods, barrier constructions, regularization techniques, and the boundary regularity of small perturbation solutions, we establish boundary H\"older gradient estimates that unify and extend previous results.	
\end{abstract}

\maketitle
%
%

\section{Introduction}
We study the boundary regularity of viscosity solutions to the Dirichlet problem for the fully nonlinear degenerate ($\gamma>0$) or singular ($-1<\gamma<0$) parabolic equation
\begin{equation} \label{eq:main}
\left\{\begin{aligned}
	u_t &= |Du|^\gamma F(D^2 u) + f  && \text{in } \Omega \\
	u&=g && \text{on } \partial_p \Omega,
\end{aligned}\right.
\end{equation}
where $\Omega \subset \mathbb{R}^{n+1}$ is a domain with $C^{1,\alpha}$ parabolic boundary, $f \in C(\Omega)\cap L^\infty(\Omega)$, and $g \in C^{1,\alpha}(\partial_p \Omega)$. The operator $F$ is assumed to be uniformly elliptic and continuous on the space of symmetric matrices.

The main objective of this paper is to establish the boundary $C^{1,\alpha}$-regularity of viscosity solutions to \eqref{eq:main} under the assumption that $F$ is convex (or concave). While the analysis of this equation presents various difficulties, two central challenges arise from the presence of $|Du|^\gamma$ and $u_t$. The degeneracy or singularity induced by $|Du|^\gamma$ prevents the direct application of classical techniques for uniformly parabolic equations to derive $C^{1,\alpha}$-regularity. Moreover, the time derivative $u_t$ obstructs the use of certain key tools developed for degenerate elliptic equations, such as the cutting lemma introduced by Imbert–Silvestre \cite{IS13}, thereby requiring fundamentally different approaches in the parabolic setting. To overcome this difficulty, we develop a strategy that combines barrier constructions, comparison principles, compactness methods, regularized equations with uniform $C^{1,\alpha}$-estimates, and the boundary regularity of small perturbation solutions to derive the desired boundary estimates.
\subsection{Main results}
The boundary regularity of solutions to \eqref{eq:main} can be established by first analyzing a model problem, specifically, the Dirichlet problem with zero boundary data on a flat boundary. This strategy is inspired by the approach developed by Lian--Zhang \cite{LZ20, LZ22}.
\begin{theorem}  [Boundary $C_\gamma^{1,\alpha}$-estimates on a flat boundary] \label{thm:c1a_flat}
Let $\gamma >-1$. Suppose that $F$ satisfies \textnormal{\ref{F1}}, \textnormal{\ref{F2}}, and $u$ is a viscosity solution to 
\begin{equation} \label{model}  
\left\{\begin{aligned}
	u_t &= |Du + p|^\gamma F(D^2 u)  && \text{in } Q_1^+ \coloneqq Q_1 \cap \{x_n>0\} \\
	u&=0 && \text{on } S_1 \coloneqq Q_1 \cap \{x_n=0\}
\end{aligned}\right.
\end{equation}
with $|p| \leq 1$ and $p_n=0$. Then $u \in C_\gamma^{1,\bar{\alpha}}(0,0)$ for some $\bar{\alpha}\in(0,1)$ with $\bar\alpha<\frac{1}{1+\gamma}$, that is, there exists a constant $a\in \mathbb{R}$ such that
\begin{equation*}
	|u(x,t)-ax_n| \leq C_\star(|x|^{1+\bar{\alpha}}+|t|^{\frac{1+\bar{\alpha}}{2-\bar{\alpha} \gamma}}) \quad \text{for all }(x,t) \in Q_{1/2}^+
\end{equation*}
and $|a| \leq C_\star$, where $C_\star>0$ is a constant depending only on $n$, $\lambda$, $\Lambda$, $\gamma$, and $\|u\|_{L^\infty(Q_1^+)}$.
\end{theorem}
We denote by $\bar\alpha$ the constant appearing in \Cref{thm:c1a_flat}. We are now in a position to state our main theorems.
\begin{theorem} [Boundary $C^{1,\alpha}$-estimates for the degenerate equations]\label{thm:main_deg}
Let $\alpha \in (0,\bar\alpha)$, $\gamma \ge 0$, and $(x_0,t_0) \in \partial_p \Omega$. Suppose that  $F$ satisfies \textnormal{\ref{F1}}, \textnormal{\ref{F2}},  
\begin{equation*}
	f \in C(\Omega) \cap L^\infty(\Omega), \quad 
	g \in C^{1,\alpha}_\gamma (x_0,t_0), \quad 
	\partial_p \Omega \in C^{1,\alpha}_\gamma (x_0,t_0),
\end{equation*}
and $u \in C(\overline{\Omega})$ is a viscosity solution to 
\begin{equation*}  
\left\{\begin{aligned}
	u_t &= |Du|^\gamma F(D^2 u)  + f&& \text{in } \Omega\cap Q_1(x_0,t_0)\\
	u&=g && \text{on } \partial_p \Omega\cap Q_1(x_0,t_0).
\end{aligned}\right.
\end{equation*}
Then $u \in C^{1,\alpha}(x_0,t_0)$, that is, there exists a linear function $L(x)$ such that
\begin{equation*}
	|u(x,t)-L(x)| \leq C(|x-x_0|^{1+\alpha}+|t-t_0|^{\frac{1+\alpha}{2}}) \quad \text{for all } (x,t) \in \Omega \cap Q_1(x_0,t_0)
\end{equation*}
and $|DL| \leq C$, where $C>0$ is a constant depending only on $n$, $\lambda$, $\Lambda$, $\gamma$, $\alpha$, $\|u\|_{L^\infty(\Omega)}$, $\|f\|_{L^\infty(\Omega)}$, $\|g\|_{C_\gamma^{1,\alpha}(x_0,t_0)}$, and $[\partial_p \Omega]_{C_\gamma^{1,\alpha}(x_0,t_0)}$.
\end{theorem}
\begin{theorem} [Boundary $C_\gamma^{1,\alpha}$-estimates for the singular equations]\label{thm:main_sing}
Let $\alpha \in (0,\bar\alpha)$, $\gamma \in (-1, 0)$, and $(x_0,t_0) \in \partial_p \Omega$. Suppose that  $F$ satisfies \textnormal{\ref{F1}}, \textnormal{\ref{F2}},  
\begin{equation*}
	f \in C(\Omega) \cap L^\infty(\Omega), \quad 
	g \in C^{1,\alpha} (x_0,t_0), \quad 
	\partial_p \Omega \in C^{1,\alpha} (x_0,t_0),
\end{equation*}
and $u \in C(\overline{\Omega})$ is a viscosity solution to 
\begin{equation*}  
\left\{\begin{aligned}
	u_t &= |Du|^\gamma F(D^2 u)  + f&& \text{in } \Omega\cap Q_1(x_0,t_0)\\
	u&=g && \text{on } \partial_p \Omega\cap Q_1(x_0,t_0).
\end{aligned}\right.
\end{equation*}
Then $u \in C_\gamma^{1,\alpha}(x_0,t_0)$, that is, there exists a linear function $L(x)$ such that
\begin{equation*}
	|u(x,t)-L(x)| \leq C(|x-x_0|^{1+\alpha}+|t-t_0|^{\frac{1+\alpha}{2-\alpha\gamma}}) \quad \text{for all } (x,t) \in \Omega \cap Q_1(x_0,t_0)
\end{equation*}
and $|DL| \leq C$, where $C>0$ is a constant depending only on $n$, $\lambda$, $\Lambda$, $\gamma$, $\alpha$, $\|u\|_{L^\infty(\Omega)}$, $\|f\|_{L^\infty(\Omega)}$, $\|g\|_{C^{1,\alpha}(x_0,t_0)}$, and $[\partial_p \Omega]_{C^{1,\alpha}(x_0,t_0)}$.
\end{theorem}
In contrast to the interior regularity theory developed in \cite{IJS19, LLY24}, where the space-time scaling remains fixed, the boundary regularity of solutions to \eqref{eq:main} depends not only on the regularity of the boundary data $g$, but also on the regularity of the parabolic boundary $\partial_p \Omega$. As a result, the appropriate space-time scaling near the boundary varies with the range of $\gamma$. The underlying reason is that, when $Du(x_0,t_0)\neq 0$, the equation behaves, roughly speaking, like a uniformly parabolic one. In such cases, the solution inherits regularity both from the structure of the equation and from the external date, namely, $g$ and $\partial_p \Omega$, and the actual boundary regularity of $u$ is governed by the lower of these two contributions. Consequently, depending on the value of $\gamma$, it is natural to impose the more favorable condition among $g, \partial_p \Omega \in C^{1,\alpha}(x_0,t_0)$ and $g, \partial_p \Omega \in C_\gamma^{1,\alpha}(x_0,t_0)$, but the resulting regularity of $u$ will still be limited to the lesser one. This phenomenon was first observed in \cite{LLYZ25b}.
\subsection{Related works} 
The study of fully nonlinear degenerate or singular parabolic equations of the form
\begin{equation}\label{eq:local}
	u_t =|Du|^\gamma F(D^2u) +f 
\end{equation}
has been actively pursued in recent years. Early foundational work by Demengel \cite{Dem11} established comparison principles and existence results for viscosity solutions. For the case $\gamma>0$, Argiolas--Charro--Peral \cite{ACP11} obtained ABP-type estimates. In the broader range $\gamma>-1$, H\"older continuity of the gradient was established by Lee–Lee–Yun \cite{LLY24}. More recently, Lee--Lian--Yun--Zhang \cite{LLYZ25a} established time derivative estimates and optimal regularity results for $\gamma>0$, while Arya--Julin \cite{AJ25} proved Harnack inequalities for the same range.

In the elliptic setting, i.e.,
\begin{equation*}
	|Du|^\gamma F(D^2u) =f 
\end{equation*}
the $C^{1,\alpha}$-regularity for $\gamma>0$ was established by Imbert--Silvestre \cite{IS13}. This was further extended to boundary regularity by Birindelli--Demengel \cite{BD14}, and to optimal regularity by Araújo--Ricarte--Teixeira \cite{ART15}. In addition, $W^{2,\delta}$ type estimates were studied in Byun--Kim--Oh \cite{BKO25}. The literature on this topic has since grown rapidly, with various contributions addressing diverse aspects of degenerate elliptic equations (see, e.g., \cite{APPT22, AS23, BBLL23, BBLL22b, BJdSRR23, BJdSR23, BPRT20, BKO25b, dSRRV23, dSR20, dSV21a, dSV21b, FRZ21}).

Another important class is the quasilinear model
\begin{equation*}
	u_t = |Du|^\gamma \Delta_p u + f.
\end{equation*}
For the case $\gamma=2-p$, Jin--Silvestre \cite{JS17} obtained H\"older gradient estimates in the homogeneous case, and Imbert--Jin--Silvestre \cite{IJS19} extended this to the range $\gamma>-1$. More recently, Lee--Lian--Yun--Zhang \cite{LLYZ25b} developed boundary H\"older gradient estimates in this framework. Further contributions to the theory of degenerate or singular quasilinear parabolic equations include \cite{Att20, AP18, AR20, FZ21, FZ23, FPS25, PV20}.

\subsection{Strategy of the proof}
Despite significant progress in both elliptic and parabolic settings, many aspects of boundary regularity for fully nonlinear degenerate or singular parabolic equations remain poorly understood. In particular, the interplay between the gradient-dependent degeneracy and the boundary data presents delicate challenges. The present work aims to address this issue by developing a unified boundary regularity theory for viscosity solutions to fully nonlinear parabolic equations with $|Du|^\gamma$-type degeneracy or singularity.

The proof of our main results builds upon several key analytical techniques developed in recent literature. We outline the strategy as follows:

\begin{enumerate}[label=\text{(\roman*)}]

\item \textbf{Improvement of the oscillation of \( Du \).}  
We follow the strategy of \cite{IJS19, JS17}, where the improvement of oscillation plays a central role in establishing the H\"older gradient estimates. In particular, we develop a mechanism in \Cref{sec:bdry_reg_flat} that allows one to improve the oscillation of $Du$ under the assumption that the set where the gradient stays away from size one has positive measure. The argument relies fundamentally on the weak Harnack inequality, which enables a quantitative control of the upper bound of $Du$ in small cylinders.

\item \textbf{Boundary regularity theory without flattening the boundary.}  
Inspired by the approach in \cite{LZ20, LZ22}, we adopt a compactness method to derive boundary regularity without applying any flattening transformations. Instead of modifying the geometry of the domain, we reduce the boundary behavior to that of a model problem with flat boundary and zero Dirichlet data. We establish boundary regularity for the model problem in \Cref{sec:model_problem}. These results serve as the foundation for the problem in general domains, which we develop in \Cref{sec:gen_bdry}.

\item \textbf{Techniques for fully nonlinear degenerate or singular parabolic equations.}  
In line with the framework developed in \cite{LLY24}, we treat the fully nonlinear operator $F$ via suitable regularization. This allows us to differentiate the equation and obtain uniform estimates for smooth approximating solutions. The uniform $C_\gamma^{1,\alpha}$-estimates for the regularized equations play a crucial role in establishing the regularity of solutions to the original degenerate or singular problem.

\item \textbf{Boundary regularity theory for parabolic $p$-Laplace type equations.}  
We also follow the strategy in \cite{LLYZ25b}, which provides a boundary regularity theory for parabolic $p$-Laplace type equations. This includes the boundary regularity of small perturbation solutions and the analysis of how the regularity of the boundary data propagates to the solution in the degenerate or singular regime. We tailor these methods to suit the fully nonlinear equation, which is essential for establishing boundary regularity in the degenerate or singular setting.
\end{enumerate}

\subsection{Organization of the paper}
The paper is organized as follows. In \Cref{sec:pre}, we introduce the necessary preliminaries, including notations, structural conditions on the operator, definitions of viscosity solutions, and known regularity results. \Cref{sec:small_pert} is devoted to establishing boundary regularity of small perturbation solutions to the regularized problem. In \Cref{sec:model_problem}, we analyze a model problem with zero boundary data on a flat boundary and prove \Cref{thm:c1a_flat}, which serves as the foundation for the general case. \Cref{sec:gen_bdry} presents the proof of the main theorems. The argument proceeds via contradiction, using compactness techniques to reduce to the model problem.
%
%
\section{Preliminaries} \label{sec:pre}
In this section, we summarize some basic notations, definitions and known regularity results that will be used throughout the paper. 
\begin{notation}
Let us display some basic notations as follows.
\begin{itemize}
	\item $\mathcal{S}^n = \{ M \in  \mathbb{R}^{n\times n} : M=M^T\}$.
	\item $B_r(x_0)=\{x\in \mathbb{R}^{n} : |x-x_0|<r\}$ and $B_r=B_r(0)$.
	\item $B_r^+(x_0)= \{x\in B_r(x_0)  : x_n >0\}$ and $B_r^+=B_r^+(0)$.
	\item $Q_r(x_0,t_0)=B_r(x_0)\times (t_0-r^{2},t_0]$ and $Q_r=Q_r(0,0)$.
	\item $Q_r^+(x_0,t_0)=\{(x,t) \in Q_r(x_0,t_0) : x_n>0 \}$ and $Q_r^+=Q^+_r(0,0)$.
	\item $Q^\rho_r=B_r\times (-\rho^{-\gamma}r^2,0]$ and $Q^{\rho+}_r=B_r^+\times (-\rho^{-\gamma}r^2,0]$.
	\item $S_r=\{x\in B_r : x_n=0\} \times (-r^2,0]$.
	\item We denote $\{e_1,e_2,\cdots, e_n\}$ as the standard basis of $\mathbb{R}^n$.
	\item We denotes $\|M\|$ as the Frobenius norm of $M=(M_{ij})_{1\leq i,j\leq n} \in \mathbb{R}^{n\times n}$, i.e.,
\begin{equation*}
    \|M\| = \left(\sum_{i,j=1}^n |M_{ij}|^2\right)^{1/2}.
\end{equation*}
	\item For a unit vector $e\in\mathbb{R}^n$, we denote by $D_e u$ the directional derivative of $u$ in the direction $e$.
\end{itemize}
\end{notation}
\subsection{Structural hypotheses on $F$}  
In the regularity theory of fully nonlinear equations, Pucci's extremal operators are employed to quantitatively characterize the uniform ellipticity condition in a precise and explicit manner, facilitating the analysis of the structure and behavior of general nonlinear operators. Given fixed ellipticity constants, they play a central role as canonical representatives of the maximal and minimal operators within the class of uniformly elliptic operators. For this reason, Pucci's extremal operators are regarded as indispensable tools throughout the theory of viscosity solutions to fully nonlinear equations. We now introduce the precise definition of these operators.
\begin{definition} [Pucci's extremal operators] 
Given ellipticity constants $0<\lambda \leq \Lambda$, the \textit{Pucci's extremal operators} $\mathcal{M}_{\lambda, \Lambda}^{+}$ and $\mathcal{M}_{\lambda, \Lambda}^{-}$ for a symmetric matrix $M\in\mathcal{S}^n$ are defined by
\begin{align*}
	\mathcal{M}_{\lambda, \Lambda}^{+}(M) \coloneqq \sup_{\lambda I \leq A \leq \Lambda I} \text{tr} \,(AM) \quad \text{and} \quad 
	\mathcal{M}^{-}_{\lambda, \Lambda}(M)  \coloneqq \inf_{\lambda I \leq A \leq \Lambda I} \text{tr} \,(AM).
\end{align*}
\end{definition}

In our main theorems, we assume the following structural conditions on the fully nonlinear operator $F:\mathcal{S}^n \to \mathbb{R}$.
\begin{enumerate} [label=\text{(F\arabic*)}]
	\item \label{F1} $F$ is uniformly elliptic, that is,  there exist constants $\lambda$ and $\Lambda$ with $0<\lambda \leq \Lambda$ such that \begin{equation*}
	\mathcal{M}^{-}_{\lambda, \Lambda}(M-N) \leq F(M) - F(N) \leq \mathcal{M}^{+}_{\lambda, \Lambda}(M-N)
	\quad \text{for all } M,N \in \mathcal{S}^n.
\end{equation*}
In addition, $F(O)=0$.
	\item \label{F2} $F$ is either concave or convex.
\end{enumerate}
\subsection{Regularization of $F$} \label{sec:regul_F}
We introduce a method to regularize $F$. By considering 
\begin{equation*}
	\widetilde{F}(M)=F\left(\frac{M+M^T}{2}\right)
\end{equation*}
instead of $F(M)$, we can understand $F$ as an operator defined in $\mathbb{R}^{n\times n}$. Let  $\psi\in C_c^{\infty}(\mathbb{R}^{n\times n})$  be a standard mollifier with 
\begin{equation*}
	\int_{\mathbb{R}^{n\times n}}\psi(M)\,dM=1 \quad \text{and} \quad \textnormal{supp}(\psi) \subset  \{M : \|M\| \leq 1 \}.
\end{equation*}

We define $F^\varepsilon$ as 
\begin{equation} \label{mollification}
	F^\varepsilon(M)\coloneqq  F*\psi_\varepsilon(M) = \int_{\mathbb{R}^{n\times n}} F(M-N) \psi_\varepsilon(N)\,dN \quad \text{for } \psi_{\varepsilon}(M)\coloneqq \varepsilon^{-n^2} \psi(M/\varepsilon). 
\end{equation}
Then $F^\varepsilon$ is smooth and $F^\varepsilon \rightarrow F$ uniformly in $\mathcal{S}^n$. Furthermore, for $M \in \mathbb{R}^{n\times n}$, we have
\begin{align*}
	|F^\varepsilon(M)-F(M)| &= \int_{\|N\|\leq \varepsilon} | F(M-N)-F(M) |\psi_\varepsilon(N)\,dN  \leq C\varepsilon.
\end{align*}
The structural hypotheses on $F$ are preserved under this mollification.
\begin{lemma} \label{lem-F-approx}
Let $F^\varepsilon$  denote the mollification of $F$ given by \eqref{mollification}. Then the following statements hold:
\begin{enumerate} [label=(\roman*)]
\item If $F$ satisfies \textnormal{\ref{F1}}, then $F^\varepsilon$ also satisfies \textnormal{\ref{F1}}.
\item If $F$ satisfies \textnormal{\ref{F2}}, then $F^\varepsilon$ also satisfies \textnormal{\ref{F2}}.
\end{enumerate}
\end{lemma}

\begin{proof}
(i) Since $F$ satisfies \textnormal{\ref{F1}}, for $M, N\in\mathbb{R}^{n\times n}$, we have
\begin{align*}
	F^\varepsilon(M)-F^\varepsilon(N)&=\int_{\mathbb{R}^{n\times n}} \big( F(M-P) - F(N-P) \big) \psi_\varepsilon(P)\,dP \\
	&\leq \int_{\mathbb{R}^{n\times n}} \mathcal{M}_{\lambda,\Lambda}^+(M-N)\psi_\varepsilon(P)\,dP =  \mathcal{M}_{\lambda,\Lambda}^+(M-N). 
\end{align*}
Similarly, we have $F^\varepsilon(M)-F^\varepsilon(N) \ge \mathcal{M}_{\lambda,\Lambda}^-(M-N)$, and hence $F^\varepsilon$ also satisfies \textnormal{\ref{F1}}. \\

\noindent (ii) It suffices to consider the case when $F$ is concave. Since $F$ is concave, for any $M_1, M_2 \in \mathbb{R}^{n\times n}$ and $s\in [0,1]$, we have
\begin{align*}
	F^\varepsilon(s M_1 + (1-s)M_2) &=  \int_{\mathbb{R}^{n\times n}} F(s M_1 + (1-s)M_2-N) \psi_\varepsilon(N)\,dN \\
	&=\int_{\mathbb{R}^{n\times n}} F\big(s (M_1-N) + (1-s)(M_2-N)\big) \psi_\varepsilon(N)\,dN \\
	&\ge\int_{\mathbb{R}^{n\times n}} \big(sF(M_1-N) + (1-s)F(M_2-N)\big) \psi_\varepsilon(N)\,dN \\
	&=sF^\varepsilon(M_1) + (1-s) F^\varepsilon(M_2), 
\end{align*}
and hence $F^\varepsilon$ is also concave. 
\end{proof}

The interior $C^{1,\alpha}$-estimates for \eqref{eq:local} adopt an approach that compensates for the lack of regularity of solutions by regularizaing the equation as follows 
\begin{equation*} 
	u_t = (|Du|^2+\varepsilon^2)^{\gamma/2} F^\varepsilon(D^2u) + f , 
\end{equation*}
see \cite{LLY24}.
However, unlike the interior case, the boundary $C^{1,\alpha}$-estimates seek regularity under prescribed boundary values, and such regularizaing is not appropriate.
Consequently, the regularized equation becomes slightly more involved and takes the following form:
\begin{equation} \label{eq:general} 
	u_t = (|\nu Du + p|^2+\varepsilon^2)^{\gamma/2} F^\varepsilon(D^2u) + f.
\end{equation}
\subsection{Viscosity solutions}
We now introduce the concept of solutions to fully nonlinear parabolic equations:
\begin{equation}\label{eq:general'}  
	u_t = (|\nu Du + p|^2+\varepsilon^2)^{\gamma/2} F(D^2u) + f.
\end{equation}

\begin{definition} [Test functions]
	Let $u:\Omega \to \mathbb{R}$ be a continuous function. The function $\varphi : \Omega \to \mathbb{R}$ is called \textit{test function} if it is a function of class $C^2$ in $x$ and $C^1$ in $t$.
\end{definition}

We say that the function $\varphi$ touches $u$ from above (resp. below) at $(x_0,t_0)$ if there exists an open neighborhood $U$ of $(x_0,t_0)$ such that 
\begin{equation*}
	u \leq (\text{resp.} \ge)~ \varphi  \quad \text{in } U \qquad  \text{and} \qquad u(x_0,t_0)= \varphi(x_0,t_0). 
\end{equation*}
\begin{definition}[Viscosity solutions]  
Let $F:\mathcal{S}^n \to \mathbb{R}$ be a fully nonlinear operator and let $u, f:\Omega \to \mathbb{R}$ be  continuous functions. We say that $u$ is a \textit{viscosity subsoution} (resp. \textit{viscosity supersoution}) to \eqref{eq:general'} in $\Omega$ if for any $(x_0,t_0)\in \Omega $, the following two statements hold: 
\begin{enumerate} [label=(\roman*)]
\item For any test function $\varphi$ touching $u$ from above (resp. below) at $(x_0,t_0)$, \\
if $\nu D\varphi(x_0,t_0) + p \neq 0$, then
	\begin{equation*}
		\varphi_t (x_0,t_0)\leq (\text{resp.} \ge)~  (|\nu D\varphi(x_0,t_0) + p|^2+\varepsilon^2)^{\gamma/2} F(D^2\varphi(x_0,t_0)) + f(x_0,t_0).
	\end{equation*}
\item Let $v(x,t)=\nu u(x,t) + p\cdot x$. For any function $\psi \in C^1(t_0-\delta, t_0 +\delta)$ for some $\delta>0$, if 
\begin{equation*}
	v(x_0,t_0) -\psi(t_0) \ge (\text{resp.} \leq)~  v(x_0,t) - \psi(t) \quad \text{for all } t \in  (t_0-\delta, t_0 +\delta)
\end{equation*}
and
\begin{equation*}
	\sup_{t\in (t_0-\delta, t_0 +\delta)} (\text{resp.} \inf)~(v(x,t) -\psi(t)) = v(x_0,t_0) -\psi(t_0) \quad \text{for all } x \in B_{\delta}(x_0),
\end{equation*}
then 
\begin{equation*}
	\psi'(t_0) \leq (\text{resp.} \ge)~ \nu f(x_0,t_0).
\end{equation*}
\end{enumerate}
\end{definition}
\begin{remark} \label{rmk:eq_tr}
If $\nu \neq 0$, then $u$ is a viscosity solution to \eqref{eq:general'} in $\Omega$ if and only if $v=\nu u + p\cdot x$ is a viscosity solution to 
\begin{equation*}
	v_t = (|Dv|^2+\varepsilon^2)^{\gamma/2}\widetilde{F}(D^2 v)+ \nu f \quad \text{in } \Omega ,
\end{equation*}
where $\widetilde{F}(M)= \nu F(\nu^{-1} M)$. 
\end{remark}
\subsection{H\"older spaces}
Degenerate or singular parabolic equations exhibit a space-time scaling that differs from that of uniformly parabolic equations. For detailed discussion on the relevant scaling, we refer the reader to \cite{IJS19, KLY23, KLY24, LLY24, LY25, LLYZ25a}. Accordingly, it is necessary to introduce H\"older spaces that are adapted to both the space-time scaling and the regularity structure, in a way that reflects both the regularity of solutions to \eqref{eq:main} as well as the geometry of $\partial_p \Omega$.
\begin{definition} \label{def:function}
	Let $A \in \mathbb{R}^{n+1}$ be a bounded set and $u: A \to \mathbb{R}$ be a function. For $\alpha \in (0, 1)$ with $\alpha <\frac{1}{1+\gamma}$, we say that $u$ is $C_\gamma^{1,\alpha}$ at $(x_0,t_0) \in A$ (denoted by $u \in C_\gamma^{1,\alpha}(x_0,t_0)$) if there exist constants $N>0$, $r>0$ and a linear function $L(x)$ such that 
	\begin{equation} \label{c1a-est}
		|u(x,t)-L(x)| \leq N(|x-x_0|^{1+\alpha}+|t-t_0|^{\frac{1+\alpha}{2-\alpha\gamma}}) \quad \text{for all } (x,t) \in A \cap Q_r^{r^\alpha}(x_0,t_0).
	\end{equation}
	We define 
	\begin{equation*}
		[u]_{C_\gamma^{1,\alpha}(x_0,t_0)} = \inf \{N \mid \eqref{c1a-est} \text{ holds with some } N \text{ and } L \}
	\end{equation*}
and
	\begin{equation*}
		\|u\|_{C_\gamma^{1,\alpha}(x_0,t_0)} =  [u]_{C_\gamma^{1,\alpha} (x_0,t_0)} + |L(x_0)| + |DL|.
	\end{equation*}
\end{definition}

\begin{definition} \label{def:domain}
	Let $\Omega \subset \mathbb{R}^{n+1}$ be a bounded domain. For $\alpha \in (0,1)$ with $\alpha<\frac{1}{1+\gamma}$, we say that $\partial_p \Omega$ is $C_\gamma^{1,\alpha}$ at $(x_0,t_0) \in \partial_p \Omega$ (denoted by $\partial_p \Omega \in C_\gamma^{1,\alpha}(x_0,t_0)$), if there exist constants $N>0$, $r>0$ and a new coordinate system $(x,t)$ such that $(x_0,t_0)$ is mapped to the origin and the plane $x_n=0$ is tangent to $\partial_p \Omega$ at the origin, in which the following conditions hold:
\begin{equation}\label{c1a_d}
		\begin{aligned}
			&\{ (x,t) \in Q_r^{r^\alpha} \mid x_n >   N(|x|^{1+\alpha}+|t|^{\frac{1+\alpha}{2-\alpha\gamma}}) \} \subset Q_r^{r^\alpha} \cap \Omega \quad \text{and} \\
			&\{ (x,t) \in Q_r^{r^\alpha} \mid x_n <  - N(|x|^{1+\alpha}+|t|^{\frac{1+\alpha}{2-\alpha\gamma}}) \} \subset Q_r^{r^\alpha} \cap \Omega^c.
		\end{aligned}
	\end{equation}
We define 
	\begin{equation*}
		[\partial_p \Omega]_{C_\gamma^{1,\alpha}(x_0,t_0)} = \inf \{N \mid \eqref{c1a_d} \text{ holds with some } N \}.
	\end{equation*}
\end{definition}

In the case $\gamma=0$, the space $C_\gamma^{1,\alpha}(x_0,t_0)$ coincides with the classical H\"older space, so we omit $\gamma$ and denote it simply by $C^{1,\alpha}(x_0,t_0)$.

\begin{definition}
	Let $\Omega \subset \mathbb{R}^{n+1}$ be a bounded domain. If there exist constants $N>0$, $r>0$ and $\rho$ such that 
\begin{equation} \label{osc_cond}
	\{ (x,t) \in Q_r^\rho \mid x_n >   N  \} \subset Q_r^\rho \cap \Omega \quad \text{and} \quad
	\{ (x,t) \in Q_r^\rho \mid x_n <  - N  \} \subset Q_r^\rho \cap \Omega^c,
\end{equation}
then we define 
\begin{equation*}
	\osc_{Q_r^\rho} \partial_p \Omega  = \inf \{N \mid \eqref{osc_cond} \text{ holds with some } N \}.
\end{equation*}
\end{definition}
\subsection{Known regularity results}
Before presenting the comparison principle, we emphasize that a key analytic tool for the boundary regularity is the construction of appropriate barrier functions. By comparing the solution with these barriers via the comparison principle, we are able to analyze the boundary behavior of viscosity solutions to \eqref{eq:general'}.

Such a comparison principle can be found in \cite[Theorem 8.3]{CIL92} when $\varepsilon \neq 0$, and in \cite[Theorem 1]{Del11} for the case $\varepsilon =0$.
\begin{theorem}  [Comparison principle] \label{thm:comp_prin}
Let $u$ and $v$ be a viscosity subsolution and a viscosity supersolution to \eqref{eq:general'} in $\Omega$,
respectively. If $u \leq v$ on $\partial_p \Omega$, then $u\leq v$ in $\Omega$.
\end{theorem}
We now present two types of stability results, each playing a distinct role in the analysis. The first stability theorem, which can be found in \cite[Theorem 2.10]{LLY24}, is used to identify the limiting equation satisfied by a locally uniformly convergent sequence of viscosity solutions, typically obtained via compactness arguments such as the Arzelà--Ascoli theorem. The second stability theorem ensures that the regularized approximations converge back to a viscosity solution of the original equation. This is particularly crucial for recovering solutions to degenerate or singular equations through regularization techniques.
\begin{theorem}  [Stability theorem 1] \label{thm:stability1}
Assume that $\mathcal{F}_k, \mathcal{F} : \mathcal{S}^n \times \mathbb{R}^n \times \mathbb{R} \times \Omega \to \mathbb{R}$ are continuous satisfying $\mathcal{F}_k \to \mathcal{F}$ locally uniformly in $\mathcal{S}^n \times \mathbb{R}^n \times \mathbb{R} \times \Omega$ and
\begin{equation*}
	\mathcal{F}_k(M,p,r, x,t) \leq \mathcal{F}_k(N,p,s, x,t) \quad \text{whenever } s\leq r \text{ and } M \leq N.
\end{equation*}
Let $u_k$ be a viscosity subsolution (resp. supersolution) to
\begin{equation*}
	\partial_t u_k = \mathcal{F}_k(D^2u_k,Du_k,u_k, x,t) \quad \text{in } \Omega.
\end{equation*}
If $u_k \to u$ locally uniformly in $\Omega$, then $u$ is a viscosity subsolution (resp. supersolution) to 
\begin{equation*}
	u_t = \mathcal{F}(D^2u,Du,u, x,t) \quad \text{in } \Omega.
\end{equation*}
\end{theorem}
\begin{theorem}  [Stability theorem 2] \label{thm:stability2}
Let $u^\varepsilon$ be a viscosity subsolution (resp. supersolution) to \eqref{eq:general} in $\Omega$. If $u^\varepsilon \to u$ locally uniformly in $\Omega$, then $u$ is a viscosity subsolution (resp. supersolution) to 
\begin{equation*}
	u_t = |\nu Du + p|^{\gamma} F(D^2u) + f \quad \text{in } \Omega.
\end{equation*}
\end{theorem}

\begin{proof}
The case $\gamma \ge 0$ follows from \Cref{thm:stability1}. For $\gamma < 0$, the desired conclusion can be obtained by substituting \eqref{eq:general} for the approximate equation in \cite[Section 6]{MK97}.
\end{proof}

We introduce the weak Harnack inequality which can be bound in \cite[Theorem 2.4.15]{IS13b} and \cite[Corollary 4.14]{Wan92}. This fundamental estimate provides a quantitative lower bound in terms of integral averages, and plays a crucial role in regularity theory. Once established, the weak Harnack inequality allows us to derive the improvement of oscillation, which is essential in proving H\"older continuity.
\begin{theorem}  [Weak Harnack inequality] \label{thm:WHI}
There exists a constants $\varepsilon_0> 0$ such that if $u$ be a nonnegative viscosity supersolution to
\begin{equation*}
	u_t =\mathcal{M}_{\lambda,\Lambda}^-(D^2 u) + f \quad \text{in } Q_1,
\end{equation*}
then
\begin{equation*}
	\left(\frac{1}{|Q_{1/2}(0,-3/4)|} \int_{Q_{1/2}(0,-3/4)} u^{\varepsilon_0} \,dx\,dt \right)^\frac{1}{\varepsilon_0} \leq C \left( \inf_{Q_{1/2}} u + \|f\|_{L^{n+1}(Q_1)}\right),
\end{equation*}
where $C>0$ is a constant depending only on $n$, $\lambda$ and $\Lambda$. 
\end{theorem}
\begin{lemma}  [Improvement of oscillation] \label{lem:IO}
Let $u$ be a nonnegative viscosity supersolution to 
\begin{equation*}
	u_t =  \mathcal{M}^-_{\lambda,\Lambda}(D^2 u) \quad \text{in } Q_1.
\end{equation*}
Then for each $\mu \in (0,1)$, there exist $r>0$ and $c>0$, where $r$ depends only on $n$ and $\mu$; and $c$ depends only on $n$, $\lambda$, $\Lambda$, $\gamma$ and $\mu$ such that if 
\begin{equation*}
	|\{(x,t) \in Q_1: u \ge 1\}| > \mu |Q_1|,
\end{equation*}
then
\begin{equation*}
	u \ge c \quad \text{in } Q_r.
\end{equation*}
\end{lemma}

\begin{proof}
Using \Cref{thm:WHI}, the proof is exactly the same as in \cite[Proposition 2.3]{JS17}.
\end{proof}
We require the following theorem in order to guarantee the smoothness of solutions to \eqref{eq:general}. This result was established in \cite[Theorem 1.2]{LLY24}. Although the assumption $F\in C^{1,\kappa}(\mathcal{S}^n)$ is imposed in that theorem, this condition is in fact not essential.
\begin{theorem}  [Solvability of the Cauchy-Dirichlet problmes] \label{thm:LLY24-12}
Assume that $F$ satisfies \textnormal{\ref{F1}}, \textnormal{\ref{F2}}, and let $u$ be a viscosity solution to 
\begin{equation} \label{eq:solvability}  
\left\{\begin{aligned}
	u_t &= (|Du|^2 + 1)^{\gamma/2} F(D^2 u) && \text{in } Q_1 \\
	u&=g && \text{on } \partial_p Q_1
\end{aligned}\right.
\end{equation}
with $\gamma >-2$. If $g\in C^{2,\alpha}(\overline{Q_1})$ for some $\alpha\in(1/2,1)$, and if $g$ satisfies 
\begin{equation*}
	g_t =(|Dg|^2 + 1)^{\gamma/2} F(D^2 g) \quad \text{on } \partial B_1 \times \{t=-1\},
\end{equation*}
then the Cauchy--Dirichlet problem \eqref{eq:solvability} admits a unique solution in $C^{2,\alpha}(\overline{Q_1})$.
\end{theorem}
\begin{proof}
Since the proof is exactly the same as that of \cite[Theorem 1.2]{LLY24}, we only explain how the assumption $F\in C^{1,\kappa}(\mathcal{S}^n)$ can be removed. Indeed, the only places where this assumption is used are \cite[Lemma 3.6]{LLY24} and \cite[Lemma 3.7]{LLY24}, but the final estimates in the proofs of these lemmas involve constants that are independent of $\|F\|_{C^{1,\kappa}(\mathcal{S}^n)}$. For instance, at the end of Step 2 in the proof of \cite[Lemma 3.6]{LLY24}, the constant $C$ in 
\begin{equation*}
	v_t \leq \mathcal{M}_{\lambda',\Lambda'}^+(D^2 v) + C |Dv|^2
\end{equation*}
does not depend on $F\in C^{1,\kappa}(\mathcal{S}^n)$.

Therefore, by considering a mollified approximation $F^\varepsilon$ in place of $F$, proving these lemmas for $F^\varepsilon$, and then passing to the limit as $\varepsilon \to 0$, we obtain the desired conclusion.
\end{proof}
By combining \Cref{thm:LLY24-12} with a bootstrap argument, we obtain the following regularity result which can be found in \cite[Theorem 4.12]{LLY24}.
\begin{theorem} [Interior $C^\infty$-regularity] \ \label{thm:LLY24-412}
Assume that $F$ satisfies \textnormal{\ref{F1}}, \textnormal{\ref{F2}}, and let $u$ be a viscosity solution to 
\begin{equation*}  
\left\{\begin{aligned}
	u_t &= (|Du|^2 + \varepsilon^2)^{\gamma/2} F^\varepsilon(D^2 u) && \text{in } Q_1 \\
	u&=g && \text{on } \partial_p Q_1
\end{aligned}\right.
\end{equation*}
with $\gamma >-2$. If $g\in C(\overline{Q_1})$, then the Cauchy--Dirichlet problem \eqref{eq:solvability} admits a unique solution in $C(\overline{Q_1})\cap C_{\textnormal{loc}}^{\infty}(Q_1)$.
\end{theorem}

We now state two boundary regularity results for fully nonlinear uniformly parabolic equations. The first is a boundary $C^{1,\alpha}$-regularity of functions in the class of solutions. The second theorem establishes boundary $C^{2,\alpha}$-regularity under structural assumptions on the nonlinearity. These results can be found in \cite[Theorem 2.8]{LZ22} and \cite[Theorem 1.16]{LZ22}, respectively.
\begin{theorem} [Boundary $C^{1,\alpha}$-regularity]  \label{thm:LZ22-28}
Let $u$ be a function satisfying $u=0$ on $S_1$ and
\begin{equation*}
	\mathcal{M}_{\lambda,\Lambda}^-(D^2 u) \leq u_t \leq \mathcal{M}_{\lambda,\Lambda}^+(D^2 u) \quad\text{in } Q_1^+ \text{ in the viscosity sense.}
\end{equation*}
Then $u\in C^{1,\alpha}(0,0)$ for some $\alpha\in(0,1)$, that is, there exists a constant $a\in \mathbb{R}$ such that
\begin{equation*}
	|u(x,t) -a x_n| \leq C\|u\|_{L^\infty(Q_1^+)} x_n (|x|^\alpha+|t|^{\alpha/2}) \quad \text{for all } (x,t) \in Q_{1/2}^+,
\end{equation*}
where $C>0$ is a constant depending only on $n$, $\lambda$, and $\Lambda$. 
\end{theorem}
\begin{theorem} [Boundary $C^{2,\alpha}$-regularity]  \label{thm:LZ22-116}
Assume that $\mathcal{F} : \mathcal{S}^n \times Q_1^+ \to \mathbb{R}$ satisfies \textnormal{\ref{F1}} for all $(x,t) \in Q_1^+$. Suppose that there exist constnats $\beta \in (0,1)$, $r>0$, and nonnegative functions $f_1,f_2 \in C^\beta(0,0)$ with $f_1(0,0)=f_2(0,0)$ such that $\mathcal{F}$ satisfies
\begin{equation*}
	|\mathcal{F}(M,x,t)-\mathcal{F}(M,0,0)| \leq f_1(x,t) \|M\| +f_2(x,t) \quad \text{for all } M\in\mathcal{S}^n \text{ and } (x,t)\in\overline{Q_r^+}.
\end{equation*}
Let $u$ be a viscosity solution to 
\begin{equation*}  
\left\{\begin{aligned}
	u_t &= \mathcal{F}(D^2 u,x,t) && \text{in } Q_1^+ \\
	u&=0 && \text{on } \partial_p Q_1^+.
\end{aligned}\right.
\end{equation*}
Then $u\in C^{2,\alpha}(0,0)$ for some $\alpha \in (0, \beta)$, that is, there exists a polynomial $P(x,t)$ with $\deg P \leq 2$ such that
\begin{equation*}
	|u(x,t) -P(x,t)| \leq C(\|u\|_{L^\infty(Q_1^+)} + \|f_2\|_{C^\alpha(0,0)}) (|x|^{2+\alpha}+|t|^{\frac{2+\alpha}{2}}) \quad \text{for all } (x,t) \in Q_{1/2}^+,
\end{equation*}
where $C>0$ is a constant depending only on $n$, $\lambda$, $\Lambda$, $\alpha$, and $\|f_1\|_{C^\alpha(0,0)}$. 
\end{theorem}
The concept of parabolic semijets provides a generalized notion of second-order contact in the viscosity framework and serves as a foundation for formulating key tools in the theory.
\begin{definition}[Parabolic semijets] 
	 Let $u: Q_1 \to \mathbb{R}$ be an upper (resp. lower) semicontinuous function and let $(x, t)\in Q_1$.
\begin{enumerate}[label=(\roman*)]
	\item A \textit{parabolic superjet} $\mathscr{P}^{2, +}u(x, t)$ (resp. \textit{parabolic subjet} $\mathscr{P}^{2, -}u(x, t)$) is the set of all triples $(\sigma, p, X) \in \mathbb{R} \times \mathbb{R}^n \times \mathcal{S}^n$ such that
\begin{align*}
	u(y, s) &\leq (\textnormal{resp.} \ge)~ u(x,t)+\sigma(s-t)+p \cdot(y-x) \\
	&\quad +\frac{1}{2} (y-x)^T X (y-x) +o(|y-x|^2+|s-t|) \quad \text{as $(y, s) \to (x, t)$}. 
\end{align*}
	\item A \textit{limiting superjet} $\overline{\mathscr{P}}^{2, +}u(x, t)$ (resp. \textit{limiting subjet} $\overline{\mathscr{P}}^{2, -}u(x, t)$) is the set of all triples $(\sigma, p, X) \in \mathbb{R} \times \mathbb{R}^n \times \mathcal{S}^n$ for which there exists a sequence $\{(x_n, t_n, \sigma_k, p_k, X_k)\}_{k=1}^{\infty}$ such that $(\sigma_k, p_k, X_k) \in \mathscr{P}^{2, +}u(x_k, t_k)$ (resp. $\mathscr{P}^{2, -}u(x_k, t_k)$) and 
	\begin{equation*}
		\text{$(x_k, t_k, u(x_k, t_k), \sigma_k, p_k, X_k) \to (x, t, u(x, t), \sigma, p, X)$} \quad \text{as } k \to \infty.
	\end{equation*}
\end{enumerate}	
\end{definition}

One of the most powerful and widely used tools in the analysis of viscosity solutions is Jensen--Ishii’s lemma. 
This lemma enables one to perform second-order differential tests at a local maximum of an auxiliary function involving several semicontinuous functions, even in the absence of classical differentiability. Although originally developed in the context of comparison principles, the lemma has found broad applications in establishing various regularity results, including H\"older and Lipschitz estimates. In this work, we employ the parabolic version of Jensen--Ishii’s lemma as a central component in the derivation of such estimates. The lemma can be found in \cite[Theorem 8.3]{CIL92}.
\begin{lemma}[Jensen--Ishii's lemma]\label{lem:Jen-Ish}
Let $u,v: Q_1 \to\mathbb{R}$ be upper semicontinuous functions and let $\varphi : B_1 \times B_1 \times (-1, 0) \to \mathbb{R}$ be a function that is $(x, y, t) \mapsto \varphi(x,y, t)$ is twice continuously differentiable in $(x,y)$ and once continuously differentiable in $t$. Suppose that
\begin{equation*}
	 w(x,y, t) = u(x, t) + v(y, t)-\varphi(x,y,t)
\end{equation*}
attains a local maximum at $(x_0, y_0, t_0) \in B_1 \times B_1 \times (-1, 0)$. Moreover, assume that there exists a constant $r>0$ such that for every $K>0$, one can find a constant $C>0$ satisfying the following property: if
\begin{equation*}
	 |x-x_0| + |t-t_0| \leq r , \quad
	 |y-y_0| + |t-t_0| \leq r, 
\end{equation*}
\begin{equation*}
	 (\upsilon_1, q_1, M) \in \mathscr{P}^{2, +}u(x, t), \quad 
	 (\upsilon_2, q_2, N) \in \mathscr{P}^{2, +}v(y, t), 
\end{equation*}
\begin{equation*}
	|u(x, t)| +  |q_1| + \|M\| \leq K, \quad \text{and} \quad
	|v(y, t)| + |q_2| + \|N\| \leq K,
\end{equation*}
then $\upsilon_1 \leq C$ and $\upsilon_2 \leq C$.
Then for each $\delta>0$, there exist $\sigma_1, \sigma_2 \in \mathbb{R}$ and $X, Y \in \mathcal{S}^n$ such that
\begin{enumerate}[label=(\roman*)]
	\item $(\sigma_1, p_1, X) \in \overline{\mathscr{P}}^{2, +}u(x_0, t_0)$ and $(\sigma_2, p_2, Y) \in \overline{\mathscr{P}}^{2, +}v(y_0, t_0)$, 
	\item $-(\delta^{-1}+\|A\|) \begin{pmatrix}
			I & O \\
			O & I
		\end{pmatrix} \leq 
		\begin{pmatrix}
			X & O \\
			O & Y
		\end{pmatrix}
		\leq A+ \delta A^2,$
	\item $\sigma_1 + \sigma_2=\varphi_t(x_0, y_0, t_0)$,
\end{enumerate}
where $p_1 = D_{x}\varphi(x_0, y_0, t_0)$, $p_2 =D_{y}\varphi(x_0, y_0, t_0)$, and $A=D_{(x,y)}^2\varphi (x_0, y_0, t_0)$.
\end{lemma}
%
%
\section{Boundary regularity of small perturbation solutions} \label{sec:small_pert}
In this section, we establish the boundary regularity of small perturbation solutions to the regularized problem
\begin{equation} \label{eq:ep-nu-p} 
\left\{\begin{aligned}
	u_t &= (|\nu Du + p|^2 + \varepsilon^2)^{\gamma/2} F(D^2 u) + f  && \text{in } \Omega \cap Q_1 \\
	u&=g && \text{on } \partial_p \Omega \cap Q_1.
\end{aligned}\right.
\end{equation}
To this end, we first derive interior Lipschitz estimates of viscosity solutions to \eqref{eq:general'} in $ Q_1$, and boundary Lipschitz estimates of viscosity solutions to \eqref{eq:ep-nu-p}.

Throughout this paper, we assume that  
\begin{equation*}
	\gamma>-1 \quad \text{and} \quad 
	0\leq \varepsilon\leq1.
\end{equation*}
\begin{theorem} [Boundary regularity of small perturbation solutions] \label{lem:small_bdry}
Assume that $F$ satisfies \textnormal{\ref{F1}}. Let $\alpha \in (0,1)$ and  let  $u$ be a viscosity solution to \eqref{eq:ep-nu-p} with  
\begin{equation*}
 	0\leq \nu \leq 1 ,\quad 
	1/2 \leq |p| \leq 2, \quad 
	\partial_p \Omega \in C^{1,\alpha}(0,0), 
\end{equation*}
 \begin{equation*}
	\|f\|_{L^\infty(\Omega\cap Q_1)} \leq 1, \quad
	g \in C^{1,\alpha}(0,0), \quad \text{and} \quad
	g(0,0)=0=|Dg(0,0)|.
 \end{equation*}
 Suppose that there exists a constant $\eta \in (0,1)$ depending only on $n$, $\lambda$, $\Lambda$, $\alpha$, $\gamma$, and $[\partial_p\Omega]_{C^{1,\alpha}(0,0)}$ such that
\begin{equation*}
	\|u\|_{L^\infty(\Omega\cap Q_1)} \leq \eta \quad \text{and} \quad \|g\|_{C^{1,\alpha}(0,0)} \leq \eta.
\end{equation*}
Then $u\in C^{1,\alpha}(0,0)$, that is, there exists a constant $a\in \mathbb{R}$ such that 
\begin{equation*}
	|u(x,t) - ax_n| \leq C( |x|^{1+\alpha} + |t|^{\frac{1+\alpha}{2}}) \quad \text{for all } (x,t) \in \Omega\cap Q_1
\end{equation*}
and $|a| \leq C\eta$, where $C>0$ is a constant depending only on $n$, $\lambda$, $\Lambda$, $\alpha$, and $\gamma$.
\end{theorem}
This theorem is fundamentally based on a compactness argument showing that the locally uniform limit of a sequence of viscosity solutions remains a solution to a uniformly parabolic equation, owing to the assumption $1/2 \leq |p| \leq 2$. This structural property enables us to carry out the boundary regularity analysis in a uniformly parabolic framework, despite the original degeneracy.
\subsection{Interior Lipschitz estimates in the spatial variables}
In this subsection, we establish interior Lipschitz estimates in the spatial variables for viscosity solutions to \eqref{eq:general'} in $ Q_1$. 
We first consider a simplified case with $\nu=1$ and $p=0$, where we prove log-Lipschitz estimates. These estimates are improved to Lipschitz estimates. Finally, we extend the argument to \eqref{eq:general'}, thereby obtaining the desired Lipschitz regularity.
\begin{lemma} \label{lem:log-Lip}
Assume that $F$ satisfies \textnormal{\ref{F1}} and let $u$ be a viscosity solution to \eqref{eq:general'} in $Q_1$ with $\nu=1$ and $p=0$. Then, for each $t\in(-1/4,0]$, $u(\cdot,t)$ is log-Lipschitz continuous in $B_{1/2}$, that is,
\begin{equation*}
	|u(x,t)-u(y,t)| \leq C(\|u\|_{L^\infty(Q_1)} + \|u\|_{L^\infty(Q_1)}^\frac{1}{1+\gamma}+ \|f\|_{L^\infty(Q_1)}^\frac{1}{1+\gamma}) |x-y| \big| \log |x-y| \big| 
\end{equation*}
for all $(x,t),(y,t) \in Q_{1/2}$, where $C>0$ is a constant depending only on $n$, $\lambda$, $\Lambda$, and $\gamma$. 
\end{lemma}

\begin{proof}
Without loss of generality, we may assume that $\|u\|_{L^\infty(Q_1)} \leq 1$ and $ \|f\|_{L^\infty(Q_1)}\leq1$. It then suffices to prove log-Lipschitz continuity at the origin, which is equivalent to verifying that
\begin{equation} \label{M<0}
	M\coloneqq \max_{\substack{x,y \in \overline{B_{1/2}}, \\ t \in [-1/4,0]}} \left\{ u(x,t)-u(y,t) - K_1 \phi(|x-y|) - \frac{K_2}{2} |x|^2- \frac{K_2}{2}|y|^2 - \frac{K_2}{2} t^2 \right\} \leq 0,
\end{equation}
where 
\begin{equation*}
\phi(r) =\begin{cases}
	-r \log r & \text{if } 0 \leq r \leq e^{-1} \\
	e^{-1} & \text{if } r \ge e^{-1}
\end{cases}
\end{equation*}

The proof of \eqref{M<0} proceeds by contradiction. Suppose, for contradiction, that there exist $(x_0,t_0),(y_0,t_0) \in \overline{Q_{1/2}}$ such that 
\begin{equation*}
	u(x_0,t_0)-u(y_0,t_0) - K_1 \phi(|x_0-y_0|) - \frac{K_2}{2} |x_0|^2- \frac{K_2}{2} |y_0|^2 - \frac{K_2}{2} t_0^2 > 0,
\end{equation*}
which implies that  
\begin{equation} \label{x0y0t0}
	x_0\neq y_0, \quad 
	\phi(\rho) \leq \frac{|u(x_0,t_0)-u(y_0,t_0)|}{K_1} \leq \frac{2}{K_1}, \quad \text{and} \quad 
	|x_0| + |y_0| + |t_0| \leq \frac{6}{\sqrt{K_2}} ,
\end{equation}
where $\rho=|z_0|$ and $z_0=x_0-y_0$. 

Choose $K_2>0$ sufficiently large so that $(x_0,t_0),(y_0,t_0) \in Q_{1/2}$. Furthermore, by chooing $K_1$ sufficiently large so that $\rho$ is small enough to safisty
\begin{equation} \label{pp'}
	\phi(\rho) \ge 2\rho  \quad\text{and} \quad
	\phi'(\rho) \ge 1.
\end{equation}
Then, by \Cref{lem:Jen-Ish}, for every $\delta>0$, there exist $\sigma_{x_0}, \sigma_{y_0}\in \mathbb{R}$ and $X,Y \in \mathcal{S}^n$
such that 
\begin{enumerate}  [label=(\roman*)]
	\item 
	$\Bigg\{\begin{aligned} 
		\sigma_{x_0} &\leq (|p_{x_0}|^2+\varepsilon^2)^{\gamma/2} F(X) + f(x_0, t_0) \\
		\sigma_{y_0} &\ge (|p_{y_0}|^2+\varepsilon^2)^{\gamma/2} F(Y) + f(y_0, t_0),
	\end{aligned} $
	\item $\begin{pmatrix}
			X & O \\
			O & -Y
		\end{pmatrix} \leq K_1
		\begin{pmatrix}
			Z & -Z \\
			-Z & Z
		\end{pmatrix} + (2K_2 +\delta)
		\begin{pmatrix}
			I & O \\
			O & I
		\end{pmatrix},$
	\item $\sigma_{x_0} - \sigma_{y_0}  = K_2 t_0$,
\end{enumerate}
where 
\begin{align*}
	p_{x_0} &= q + K_2 x_0 ,\quad
	p_{y_0}= q - K_2 y_0, \quad \text{and} \quad 
	Z= \phi''(\rho) \bar{z}_0 \otimes \bar{z}_0 + \rho^{-1}  \phi'(\rho)  (I - \bar{z}_0 \otimes \bar{z}_0)
\end{align*}
for $q = K_1 \phi'(\rho) \bar{z}_0$  and  $\bar{z}_0 = z_0 / \rho$. Combining (i) and (iii), we obtain
\begin{align}
	K_2 t_0 & \leq (|p_{x_0}|^2+\varepsilon^2)^{\gamma/2} F(X) -  (|p_{y_0}|^2+\varepsilon^2)^{\gamma/2} F(Y) + 2 \nonumber \\
	&\leq  \underbrace{C \|X\| \big| (|p_{x_0}|^2+\varepsilon^2)^{\gamma/2} -  (|p_{y_0}|^2+\varepsilon^2)^{\gamma/2} \big|}_{\eqqcolon T_1} +  \underbrace{(|p_{y_0}|^2+\varepsilon^2)^{\gamma/2} \mathcal{M}_{\lambda,\Lambda}^+(X-Y)}_{\eqqcolon T_2} + 2, \label{T1+T2}
\end{align}
where $C>0$ is a constant depending only on $n$ and $\Lambda$.

We now aim to find upper bound for $T_1$ and $T_2$. To this end, we first estimate the upper bound of $|q|$, $|p_{x_0}|$, $|p_{y_0}|$, $\|X\|$, and $\|Y\|$. By choosing $K_1$ sufficiently large, we ensure that $|q|$ is large enough, which yields 
\begin{equation} \label{px0py0}
	\frac{1}{2}|q| \leq |p_{x_0}| \leq 2|q| \quad \text{and} \quad
	\frac{1}{2}|q| \leq |p_{y_0}| \leq 2|q|.
\end{equation}
Since $\phi''  <0$, it follows by the matrix inequality (ii) that 
\begin{equation*}
	X, -Y \leq K_1\rho^{-1} \phi'(\rho) (I - \bar{z}_0 \otimes \bar{z}_0) + 3K_2 I,
\end{equation*}
hence, we have 
\begin{align}
	F(X) &\ge (|p_{x_0}|^2+\varepsilon^2)^{-\gamma/2}  (K_2 t - 2) + \frac{(|p_{y_0}|^2+\varepsilon^2)^{\gamma/2}}{(|p_{x_0}|^2+\varepsilon^2)^{\gamma/2}} F(Y) \nonumber \\
	& \ge - C\left(  |q|^{-\gamma} + K_1 \rho^{-1}  \phi'(\rho)  + 1\right), \label{FX_lower}
\end{align}
where $C>0$ is a constant depending only on $n$, $\lambda$, $\Lambda$, and $\gamma$. Thus, it follows from \ref{F1} that 
\begin{equation} \label{XY}
	\|X\|, \|Y\| \leq C\left(  |q|^{-\gamma} + K_1\rho^{-1} \phi'(\rho) + 1\right). 
\end{equation}
We now proceed to estimate $T_1$. By \eqref{x0y0t0}, \eqref{px0py0}, \eqref{XY}, and the mean value theorem, we obtain
\begin{equation} \label{T1}
	T_1 \leq C|q|^{\gamma-1} \|X\| |x_0+y_0| \leq C\left( |q|^{-1} + |q|^{\gamma-1} + \rho^{-1} |q|^\gamma \right).
\end{equation}

Next, we derive an estimate for $T_2$. Evaluating the matrix inequality (ii) at $(\xi,\xi)$ for any vector $\xi\in\mathbb{R}^n$, we have
\begin{equation*}
	(X-Y)\xi \cdot \xi \leq 6K_2 |\xi|^2,
\end{equation*}
which implies that all eigenvalues of $X-Y$ are less than or equal to $3K_2$. On the other hand, evaluating the matrix inequality (ii) at $(\bar{z}_0, -\bar{z}_0)$ yields
\begin{equation*}
	(X-Y)\bar{z}_0 \cdot \bar{z}_0 \leq 4K_1 \phi''(\rho)+ 6K_2,
\end{equation*}
which implies that at least one eigenvalue of $X-Y$ is less than or equal to $4K_1 \phi''(\rho)+ 6K_2$. Hence, we have 
\begin{equation} \label{T2}
	T_2 \leq C  |q|^\gamma(K_1 \phi''(\rho)+1) .
\end{equation}
Combining \eqref{T1+T2}, \eqref{T1}, and \eqref{T2} gives 
\begin{equation} \label{K1p''}
	-K_1 \phi''(\rho) \leq C (1 + \rho^{-1} + |q|^{-1}  + |q|^{-\gamma} +  |q|^{-\gamma-1}).
\end{equation}
We now show that, by choosing $K_1$ sufficiently large, \eqref{K1p''} does not hold. From \eqref{x0y0t0} and \eqref{pp'}, we have 
\begin{equation*} 
	\rho\leq \frac{1}{2}\phi(\rho) \leq \frac{1}{K_1}, \quad 
	K_1 \leq |q|= K_1 \phi'(\rho) \leq-K_1 \log \rho, \quad \text{and} \quad
	-\phi''(\rho)= \frac{1}{\rho} \ge  K_1,
\end{equation*}
By taking $K_1$ sufficiently large, we have
\begin{equation*} 
	C (1 + |q|^{-1}  + |q|^{-\gamma} +  |q|^{-\gamma-1}) 
	\leq \frac{1}{2}K_1^2 \leq -\frac{1}{2} K_1 \phi''(\rho), \quad \text{and} \quad
	C\rho^{-1}< \frac{1}{2}K_1 \rho^{-1}=-\frac{1}{2}K_1 \phi''(\rho) ,
\end{equation*}
which contradicts \eqref{K1p''}.
\end{proof}
\begin{lemma} \label{lem:int-lip}
Assume that $F$ satisfies \textnormal{\ref{F1}} and let $u$ be a viscosity solution to \eqref{eq:general'} in $Q_1$ with $\nu=1$ and $p=0$. Then $u(\cdot,t) \in C^{0,1}(B_{1/2})$ for all $t\in (-1/4,0]$, that is,
\begin{equation*}
	|u(x,t)-u(y,t)|  \leq C(\|u\|_{L^\infty(Q_1)} + \|u\|_{L^\infty(Q_1)}^\frac{1}{1+\gamma}+ \|f\|_{L^\infty(Q_1)}^\frac{1}{1+\gamma})  |x-y|
\end{equation*}
for all  $(x,t),(y,t) \in Q_{1/2}$, where $C>0$ is a constant depending only on $n$, $\lambda$, $\Lambda$, and $\gamma$. 
\end{lemma}

\begin{proof}
The proof proceeds in essentially the same manner as that of \Cref{lem:log-Lip} with differences arising only in certain estimates. We therefore employ the same notation as in \Cref{lem:log-Lip} and highlight only the points where modifications are required.

In this case, $\phi$ is given by
\begin{equation*}
\phi(r) =\begin{dcases}
	r- \frac{1}{2-\bar\gamma} r^{2-\bar\gamma}& \text{if } 0 \leq r \leq 1 \\
	1-\frac{1}{2-\bar\gamma} & \text{if } r \ge 1
\end{dcases}
\end{equation*}
for some $\bar\gamma \in (0,1)$. Accordingly, in view of the modified $\phi$, \eqref{pp'} is revised as follows
\begin{equation} \label{new_p_p'}  
	\phi(\rho) \ge \frac{1}{2}\rho  \quad\text{and} \quad
	\phi'(\rho) \ge \frac{1}{2}.
\end{equation}

By \Cref{lem:log-Lip}, we know that 
\begin{equation*}
	u(x_0,t_0)-u(y_0,t_0) \leq C\rho  |\log \rho|.
\end{equation*}
Consequently, the last estimate in \eqref{x0y0t0} can be sharpened to 
\begin{equation*}
	|x_0| + |y_0| + |t_0| \leq \frac{C}{\sqrt{K_2}} \sqrt{\rho |\log \rho|},
\end{equation*}
which refines \eqref{T1} as follows
\begin{equation*}
	T_1 \leq C|q|^{\gamma-1} \|X\| |x_0+y_0| \leq C\left( |q|^{-1} + |q|^{\gamma-1} + \rho^{-1} |q|^\gamma \right) \sqrt{
	\rho|\log\rho|}.
\end{equation*}
Therefore, \eqref{K1p''} is rewritten as
\begin{equation}\label{K1p''2}
	-K_1 \phi''(\rho) \leq C \left(1 + |q|^{-\gamma} + (\rho^{-1} + |q|^{-1} +  |q|^{-\gamma} + |q|^{-\gamma-1})\sqrt{
	\rho|\log\rho|} \right).
\end{equation}
By choosing $K_1$  sufficiently large so that  $\rho$ is small, we obtain
\begin{equation} \label{for_1/p}
	C\rho^{-1}\sqrt{\rho |\log \rho|} \leq \rho^{-\frac{1+2\bar\gamma}{4}} \sqrt{|\log \rho|} \leq \frac{1-\bar\gamma}{2} K_1 \rho^{-\bar\gamma}.
\end{equation}
Since $\phi''(\rho) = -(1-\bar\gamma)\rho^{-\bar\gamma}$, combining \eqref{K1p''2} and \eqref{for_1/p} yields
\begin{equation*}
	\frac{1-\bar\gamma}{2} K_1 \rho^{-\bar\gamma}\leq C \left(1 + |q|^{-\gamma} + ( |q|^{-1} +  |q|^{-\gamma} + |q|^{-\gamma-1})\sqrt{\rho|\log\rho|} \right).
\end{equation*}
From \eqref{new_p_p'}, we have
\begin{equation*}
	\rho \leq 2\phi(\rho) \leq \frac{4}{K_1} \quad\text{and}\quad
	|q|= K_1 \phi'(\rho) \ge \frac{1}{2} K_1,
\end{equation*}
and hence we obtain
\begin{equation} \label{final_est}
	K_1^{1+\bar\gamma} \leq C \left(1 + K_1^{-\gamma } + (1 + K_1^{-1} + K_1^{-\gamma} + K_1^{-\gamma-1})K_1^{-1/2}\sqrt{\log K_1} \right),
\end{equation}
Here, the term ``$+1$''  in the innermost parentheses ensures that this inequality holds regardless of the sign of $\gamma$. 
Moreover, because $\gamma>-1$ and $\bar\gamma \in (0,1)$, taking $K_1$ sufficiently large leads to a  contradiction with \eqref{final_est}, since  $1+\bar\gamma >-\gamma-1/2$.
\end{proof}
\begin{corollary} \label{cor:temp}
Assume that $F$ satisfies \textnormal{\ref{F1}} and $0<\nu \leq1$. Let $u$ be a viscosity solution to \eqref{eq:general'} in $Q_1$. If $\|u\|_{L^\infty(Q_1)} \leq 1$ and $ \|f\|_{L^\infty(Q_1)}\leq1$,  then 
 \begin{equation*}
	|u(x, t)-u(y, t)| \leq C\left(\nu^{-\frac{\gamma}{1+\gamma}}+\nu^{-1}\big(|p|+|p|^{\frac{1}{1+\gamma}}\big)\right)|x-y| \quad \text{for all } (x,t),(y,t) \in Q_{1/2},
\end{equation*}
where $C>0$ is a constant depending only on $n$, $\lambda$, $\Lambda$, and $\gamma$. 
\end{corollary}

\begin{proof}
By \Cref{lem:int-lip} and \Cref{rmk:eq_tr}, we have
	\begin{equation} \label{temp1}
		|v(x, t)-v(y, t)| \leq C (\|v\|_{L^{\infty}(Q_1)} +\|v\|_{L^{\infty}(Q_1)}^{\frac{1}{1+\gamma}}+\|\nu f\|_{L^{\infty}(Q_1)}^{\frac{1}{1+\gamma}}  )|x-y|
	\end{equation}
for all  $(x,t),(y,t) \in Q_{1/2}$.

Because $\|v\|_{L^\infty(Q_1)} \leq \nu \|u\|_{L^\infty(Q_1)} + |p|$, rewriting \eqref{temp1} for $u$ gives
\begin{align*}
	&|u(x,t) - u(y,t)| \\
	&\quad \leq C\Big( \|u\|_{L^{\infty}(Q_1)} + \nu^{-1}(|p|+ |p|^{\frac{1}{1+\gamma}} ) + \nu^{-\frac{\gamma}{1+\gamma}}  \|u\|^{\frac{1}{1+\gamma}} _{L^{\infty}(Q_1)}  +\nu^{-\frac{\gamma}{1+\gamma}}  \|f\|^{\frac{1}{1+\gamma}} _{L^{\infty}(Q_1)} \Big)|x-y| 
\end{align*}
for all  $(x,t),(y,t) \in Q_{1/2}$, which gives the desired result, since $\|u\|_{L^\infty(Q_1)} \leq 1$ and $ \|f\|_{L^\infty(Q_1)}\leq1$.
\end{proof}
\begin{lemma} \label{lem:uni_hol}
Assume that $F$ satisfies \textnormal{\ref{F1}} and  
\begin{equation*} 
	1/2 \leq |p| \leq 1, \quad 
	\|u\|_{L^\infty(Q_1)} \leq 1, \quad\text{and} \quad 
	\|f\|_{L^\infty(Q_1)}\leq1.
\end{equation*}
Then there exist constants $\alpha_1, \nu_1 \in (0,1)$ depending only on $n$, $\lambda$, $\Lambda$, and $\gamma$ such that if $u$ is a viscosity solution to \eqref{eq:general'} in $Q_1$ with $0<\nu \leq \nu_1$, then $u(\cdot,t) \in C^{\alpha_1}(B_{1/2})$ for all $t\in (-1/4,0]$, that is,
\begin{equation*}
	|u(x,t)-u(y,t)|  \leq C |x-y|^{\alpha_1} \quad \text{for all  }(x,t),(y,t) \in Q_{1/2},
\end{equation*}
where $C>0$ is a constant depending only on $n$, $\lambda$, $\Lambda$, and $\gamma$. 
\end{lemma}

\begin{proof}
The proof proceeds in essentially the same manner as that of \Cref{lem:log-Lip} with differences arising only in certain estimates. We therefore employ the same notation as in \Cref{lem:log-Lip} and highlight only the points where modifications are required.

In this case, we set $\phi(r)=r^{\alpha_1}$,  where $\alpha_1 \in (0,1)$ is a constant to be determined later. Since $u$ is a viscosity solution to \eqref{eq:general'}, (i) of \Cref{lem:log-Lip} is modified to
\begin{enumerate}  [label=(\roman*)]
	\item 
	$\Bigg\{\begin{aligned} 
		\sigma_{x_0} &\leq (|\nu p_{x_0} + p|^2+\varepsilon^2)^{\gamma/2} F(X) + f(x_0, t_0)  \\
		\sigma_{y_0} &\ge (|\nu p_{y_0} + p|^2+\varepsilon^2)^{\gamma/2} F(Y) + f(y_0, t_0),
	\end{aligned} $
\end{enumerate}
consequently, \eqref{T1+T2} becomes
\begin{equation}\label{K2t}
\begin{aligned}
	K_2 t_0  &\leq  \underbrace{C \|X\| \big| (|\nu p_{x_0}+p|^2+\varepsilon^2)^{\gamma/2} -  (|\nu p_{y_0}+p|^2+\varepsilon^2)^{\gamma/2} \big|}_{\eqqcolon T_1} \\
	&\quad +  \underbrace{(|\nu p_{y_0}+p|^2+\varepsilon^2)^{\gamma/2} \mathcal{M}_{\lambda,\Lambda}^+(X-Y)}_{\eqqcolon T_2} + 2.
\end{aligned}
\end{equation}

By choosing $\alpha_1\in (0,1)$ sufficiently small, \Cref{cor:temp} together with \eqref{x0y0t0} yields 
\begin{equation} \label{vuq}
	\nu |q|=  \frac{\alpha_1 \nu K_1 }{|x_0-y_0|} \phi(\rho) \leq \alpha_1 \nu\frac{|u(x_0,t_0)-u(y_0,t_0)|}{|x_0-y_0|} \leq  \alpha_1 C  \leq \frac{1}{8}\leq\frac{1}{4} |p|.
\end{equation}
Combining \eqref{px0py0} and \eqref{vuq} gives
\begin{equation*}  
	\frac{1}{2}|p| \leq |\nu p_{x_0}+p| \leq 2|p| \quad \text{and} \quad
	\frac{1}{2}|p| \leq |\nu p_{y_0}+p| \leq 2|p|.
\end{equation*}
These inequalities modify \eqref{FX_lower} into
\begin{equation*}
	F(X) 
	\ge - C\left(  |p|^{-\gamma} + K_1 \rho^{-1}  \phi'(\rho)  + 1\right) \\
	\ge -C (K_1 \rho^{-1}\phi'(\rho)+1).
\end{equation*}
Consequently, \eqref{T1} and \eqref{T2} take the form
\begin{align}  
	T_1 &\leq \nu C|p|^{\gamma-1} \|X\| |x_0+y_0| \leq \nu C(K_1 \rho^{-1}\phi'(\rho)+1), \label{T1T2-2}\\
	T_2 &\leq C  |p|^\gamma(K_1 \phi''(\rho)+1) \leq C(K_1 \phi''(\rho)+1) .\nonumber
\end{align}
Hence, \eqref{K1p''} is rewritten as 
\begin{equation*}
	-K_1 \phi''(\rho) \leq C (\nu K_1 \rho^{-1}\phi'(\rho) + 1).
\end{equation*}
Since $\phi''(\rho)= -\alpha_1(1-\alpha_1)\rho^{\alpha_1-2}=  -(1-\alpha_1)\rho^{-1}\phi'(\rho)$, choosing $\nu_1$​ sufficiently small yields
\begin{equation*}
	-K_1 \phi''(\rho) \leq C.
\end{equation*}
However, if $K_1$ is taken sufficiently large, then $\phi''(\rho)$ also becomes large, leading to a contradiction.
\end{proof}
\begin{lemma} \label{lem:uni_lip}
Assume that the hypotheses of  \Cref{lem:uni_hol} hold, and let $\nu_1$ be the constant given in \Cref{lem:uni_hol}. Then there exists a constant $\nu_2 \in (0,\nu_1)$ depending only on $n$, $\lambda$, $\Lambda$, and $\gamma$ such that if $u$ is a viscosity solution to \eqref{eq:general'} in $Q_1$ with  $0<\nu \leq \nu_2$, then $u(\cdot,t) \in C^{0,1}(B_{1/2})$ for all $t\in (-1/4,0]$, that is,
\begin{equation*}
	|u(x,t)-u(y,t)|  \leq C |x-y| \quad \text{for all  }(x,t),(y,t) \in Q_{1/2},
\end{equation*}
where $C>0$ is a constant depending only on $n$, $\lambda$, $\Lambda$, and $\gamma$. 
\end{lemma}

\begin{proof}
The proof proceeds in essentially the same manner as that of \Cref{lem:int-lip} with differences arising only in certain estimates. We therefore employ the same notation as in \Cref{lem:log-Lip} and highlight only the points where modifications are required.

By \Cref{lem:uni_hol}, we know that 
\begin{equation*}
	u(x_0,t_0)-u(y_0,t_0) \leq C\rho^{\alpha_1} .
\end{equation*}
Consequently, the last estimate in \eqref{x0y0t0} can be sharpened to 
\begin{equation*}
	|x_0| + |y_0| + |t_0| \leq \frac{C}{\sqrt{K_2}} \rho^{\alpha_1/2}.
\end{equation*}
Noting that $\phi$ coincides with that in \Cref{lem:int-lip}, we choose $\nu_2 \in (0,1)$ sufficiently small so that $\nu|q| \leq \nu K_1 \leq |p|/4$ for all $\nu \leq \nu_2$. Then, \eqref{T1T2-2} takes the form
\begin{equation*} 
	T_1 \leq \nu C(K_1 \rho^{-1}\phi'(\rho)+1) \rho^{\alpha_1/2}.
\end{equation*}
Thus, if we choose $\bar\gamma=1-\alpha_1/4$, then \eqref{K2t} becomes
\begin{equation*}
	K_2 t_0 \leq C(1+ K_1 \rho^{\alpha_1/2-1} -K_1 \rho^{\alpha_1/4-1}),
\end{equation*}
which implies
\begin{equation}\label{cont3}
	K_1\rho^{\alpha_1/4-1} \leq C (1+K_1\rho^{\alpha_1/2-1}).
\end{equation}

Finally, taking $K_1$ sufficiently large makes $\rho$ sufficiently small, thereby leading to a contradiction with  \eqref{cont3}.
\end{proof}
The range of $\nu$ in \Cref{lem:uni_lip} is extended as follows
\begin{corollary} \label{lem:uni_lip_full}
Assume that the hypotheses of  \Cref{lem:uni_hol} hold, and let $u$ be a viscosity solution to \eqref{eq:general'} in $Q_1$ with  $0\leq \nu \leq 1$, then $u(\cdot,t) \in C^{0,1}(B_{1/2})$ for all $t\in (-1/4,0]$, that is,
\begin{equation*}
	|u(x,t)-u(y,t)|  \leq C |x-y| \quad \text{for all  }(x,t),(y,t) \in Q_{1/2},
\end{equation*}
where $C>0$ is a constant depending only on $n$, $\lambda$, $\Lambda$, and $\gamma$. 
\end{corollary}
\begin{proof}
If $\nu=0$, then \eqref{eq:general'} is uniformly parabolic and hence trivial. Let $\nu_2$ be the constant given in \Cref{lem:uni_lip}. By \Cref{lem:uni_lip}, it suffices to consider only the range $\nu > \nu_2$.

Finally, combining \Cref{lem:int-lip} with \Cref{rmk:eq_tr}, we obtain the desired conclusion.
\end{proof}
\subsection{Interior H\"older estimates in the time variable}
After establishing interior Lipschitz regularity in the spatial variables, we now aim to derive H\"older continuity in time. The core idea is that once Lipschitz estimates in the spatial variables are available, and the parameter $\nu$ is sufficiently small relative to $|p|$, the equation \eqref{eq:general'} behaves effectively like a uniformly parabolic equation. 
\begin{lemma} \label{lem:int_hol_time}
Assume that the hypotheses of  \Cref{lem:uni_hol} hold, and let $\nu_2$ be the constant given in \Cref{lem:uni_lip}. Then there exist constants $\alpha_2 \in(0,1)$ and $\nu_3 \in (0,\nu_2)$ depending only on $n$, $\lambda$, $\Lambda$, and $\gamma$ such that if $u$ is a viscosity solution to \eqref{eq:general'} in $Q_1$ with  $0<\nu \leq \nu_3$, then $u(x,\cdot) \in C^{\alpha_2/2}([-1/4,0])$ uniformly in $x\in B_{1/2}$, that is,
\begin{equation*}
	|u(x,t)-u(x,s)|  \leq C |t-s|^{\alpha_2/2} \quad \text{for all  }(x,t),(y,t) \in Q_{1/2},
\end{equation*}
where $C>0$ is a constant depending only on $n$, $\lambda$, $\Lambda$, and $\gamma$. 
\end{lemma}

\begin{proof}
By \Cref{lem:uni_lip}, we have $\|Du\|_{L^\infty(Q_{3/4})} \leq C$. Hence, we choose $\nu_3>0$ sufficiently small to ensure 
\begin{equation*}
	\frac{1}{2}|p| \leq |\nu Du + p| \leq 2|p| \quad \text{for all } \nu \leq \nu_3.
\end{equation*}
Since $1/2 \leq |p| \leq 1$,  we may regard $u$ as a subsolution to
\begin{equation*}
	u_t = \mathcal{M}_{\tilde{\lambda},\tilde{\Lambda}}^- (D^2u) + f \quad \text{in } Q_{3/4},
\end{equation*}
and as a supersolution to 
\begin{equation*}
	u_t = \mathcal{M}_{\tilde{\lambda},\tilde{\Lambda}}^+(D^2u)  + f  \quad \text{in } Q_{3/4},
\end{equation*}
where the new ellipticity constants $\tilde{\lambda}$, $\tilde{\Lambda}$ depend only on $\lambda$, $\Lambda$, and $\gamma$. 

Therefore, we obtain the desired conclusion from the Krylov--Safonov theory for uniformly parabolic equations.
\end{proof}
The range of $\nu$ in \Cref{lem:int_hol_time} is extended as follows, and its proof is exactly the same as that of \Cref{lem:uni_lip_full}
\begin{corollary} \label{lem:int_hol_time_full}
Assume that the hypotheses of  \Cref{lem:uni_hol} hold, and let $u$ be a viscosity solution to \eqref{eq:general'} in $Q_1$ with  $0\leq\nu \leq 1$, then $u(x,\cdot) \in C^{\alpha_2/2}([-1/4,0])$ uniformly in $x\in B_{1/2}$, that is,
\begin{equation*}
	|u(x,t)-u(x,s)|  \leq C |t-s|^{\alpha_2/2} \quad \text{for all  }(x,t),(y,t) \in Q_{1/2},
\end{equation*}
where $\alpha_2 \in (0,1)$ and $C>0$ are constants depending only on $n$, $\lambda$, $\Lambda$, and $\gamma$. 
\end{corollary}
\subsection{Boundary regularity of small perturbation solutions}
We assume that $(0,0) \in \partial_p \Omega$, and for $C^{1,\alpha}$-domains, we further assume that $(e_n,0)$ is the inward normal vector to $\partial_p \Omega$ at the origin.
\begin{lemma}[Boundary Lipschitz estimates] \label{lem:bdry_lip1}
Let $0\leq \delta < 1/4$. Assume that $F$ satisfies \textnormal{\ref{F1}} and
\begin{equation*}
 	0\leq \nu \leq 1 ,\quad 
	1/2 \leq |p| \leq 1 , \quad 
	\|u\|_{L^\infty(\Omega\cap Q_1)} \leq 1, 
\end{equation*}
\begin{equation*}
	\|f\|_{L^\infty(\Omega\cap Q_1)} \leq \delta, \quad 
	\|g\|_{L^\infty(\partial_p\Omega \cap Q_1)} \leq \delta, \quad \text{and} \quad
	\underset{Q_1}{\mathrm{osc}} \, \partial_p\Omega \leq \delta.
\end{equation*}
Let $u$ be a viscosity solution to \eqref{eq:ep-nu-p}. Then we have
\begin{equation} \label{est:u<xn+d}
	|u(x,t)| \leq C(x_n +\delta) \quad \text{for all } (x,t) \in \Omega \cap Q_{1/2},
\end{equation}
 where $C>0$ is a constant depending only on $n$, $\lambda$, $\Lambda$, and $\gamma$.
 
 Moreover, if there exists a constant $c>0$ depending only on $n$, $\lambda$, $\Lambda$, and $\gamma$ such that $\nu \ge c$, then the estimate \eqref{est:u<xn+d} also holds when $|p|\leq 1$.
 
\end{lemma}

\begin{proof}
Let us consider the function  
\begin{equation*}
 	\varphi(x,t)=2e^A\left(e^{-A} - e^{-\frac{A|x+(1+\delta)e_n|^2}{t+1}}\right) + (2+t)\delta.
\end{equation*}
Choosing $A>0$ sufficiently large ensures that $|D\varphi| \leq 4A$. Moreover, since  $1/2\leq |p|\leq 1$, for $\nu \leq \frac{1}{16A}$, we have 
\begin{equation*}
	 (|\nu D\varphi + p|^2 + \varepsilon^2)^{\gamma/2} \ge c_1,
\end{equation*}
where $c_1>0$ is a constant depending only on $\gamma$. Furthermore, $\varphi$ satisfies 
\begin{equation*}  
\left\{\begin{aligned}
	\varphi_t &> (|\nu D\varphi + p|^2 + \varepsilon^2)^{\gamma/2} F(D^2 \varphi) + \|f\|_{L^\infty(\Omega\cap Q_1)} && \text{in } \Omega\cap Q_1^+(-\delta e_n,0)  \\
	\varphi&\ge \|g\|_{L^\infty(\partial_p\Omega\cap Q_1)}  && \text{on } \partial_p \Omega \cap Q_1^+(-\delta e_n,0)  \\
	\varphi&\ge\|u\|_{L^\infty(\Omega\cap Q_1)}  && \text{on }  \Omega \cap \partial_p Q_1^+(-\delta e_n,0) ,
\end{aligned}\right.
\end{equation*}
where $Q_1^+(-\delta e_n,0) \coloneqq Q_1(-\delta e_n,0) \cap \{x_n >-\delta\}$.

Therefore, by \Cref{thm:comp_prin}, we obtain
\begin{equation*}
	-C(x_n+\delta) \leq -\varphi(0',x_n,0) \leq u(0',x_n,0) \leq \varphi(0',x_n,0) \leq C(x_n +\delta)  
\end{equation*}
for all $(0',x_n,0) \in \Omega \cap Q_{1/2}^+(-\delta e_n,0)$. Now, considering the translation $\varphi(x'+y',x_n, t+s)$ for $(y',0,s) \in S_{1/2}$, we obtain
\begin{equation*}
	|u(x,t)| \leq C (x_n+\delta) \quad \text{for all } (x,t) \in \Omega \cap Q_{1/2}^+(-\delta e_n,0).
\end{equation*}

On the other hand, in the case $\nu > \frac{1}{16A}$, we modify the argument by replacing the barrier function $\varphi$ with 
\begin{equation*}
 	\psi(x,t)= A\cdot 3^B (1-|x+(1+\delta)e_n|^{-B})-t+\delta.
\end{equation*}
Then, for sufficiently large constant $B>0$, we have
\begin{align*}
	\psi_t -(|\nu D\psi + p|^2 + \varepsilon^2)^{\gamma/2} F(D^2 \psi) &\ge-1 + ABc_2 \cdot 3^B( (B+2) \lambda - n \Lambda) |x+(1+\delta)e_n|^{-B-2} \\
	& \ge \|f\|_{L^\infty(\Omega\cap Q_1)} \quad  \text{in } \Omega\cap Q_1^+(-\delta e_n,0)
\end{align*}
whenever $|p|\leq 1$. Here, $c_2>0$ is defined by
\begin{equation*}
c_2=\begin{dcases}
	(2A^2B^2 \cdot 3^{2B}+3)^{\gamma/2} & \text{if } -1<\gamma<0\\
	\left(\frac{B^2}{2^5 \cdot3^4} - 2 \right)^{\gamma/2} & \text{if } \gamma\ge0.
\end{dcases}
\end{equation*}
Furthermore, $\psi$ satisfies
\begin{equation*}  
\left\{\begin{aligned}
	\psi&\ge \|g\|_{L^\infty(\partial_p\Omega\cap Q_1)}  && \text{on } \partial_p \Omega \cap Q_1^+(-\delta e_n,0)  \\
	\psi&\ge\|u\|_{L^\infty(\Omega\cap Q_1)}  && \text{on }  \Omega \cap \partial_p Q_1^+(-\delta e_n,0) .
\end{aligned}\right.
\end{equation*}
The rest of the proof proceeds in the same way as for $\varphi$.
\end{proof}
\begin{lemma} \label{lem:small_1st}
Assume that $F$ satisfies \textnormal{\ref{F1}}. Then, for any $\alpha \in (0,1)$, there exist constants $C_0>0$ and  $\delta_0 \in (0,1)$ depending only on $n$, $\lambda$, $\Lambda$, $\alpha$, $\gamma$, and $[\partial_p\Omega]_{C^{1,\alpha}(0,0)}$ such that if $u$ is a viscosity solution to \eqref{eq:ep-nu-p} with 
\begin{equation*}
 	0\leq \nu  \leq 1 ,\quad 
	1/2 \leq |p| \leq 2, \quad  
	\partial_p \Omega \in C^{1,\alpha}(0,0), \quad
	\|f\|_{L^\infty(\Omega\cap Q_1)} \leq 1,
\end{equation*}
\begin{equation*}
	\|u- a_0 x_n \|_{L^\infty(\Omega\cap Q_r)} \leq r^{1+\alpha}, \quad \text{and} \quad 
	\|g\|_{L^{\infty}(\partial_p \Omega\cap Q_r)} \leq \delta_0 r^{1+\alpha}
\end{equation*}
for some $r \leq \delta_0$ and $a_0 \in \mathbb{R}$ with $|a_0| \leq C_0 \delta_0^\alpha$, then there exists a constant $a \in \mathbb{R}$ such that 
\begin{equation*}
	\|u- a x_n \|_{L^\infty(\Omega\cap Q_{\tau r})} \leq (\tau r)^{1+\alpha} \quad \text{and} \quad 
	|a-a_0| \leq C_0 (\tau r)^{\alpha},
\end{equation*}
where $\tau \in (0,1)$ is a constant depending only on $n$, $\lambda$, $\Lambda$, $\alpha$, and $\gamma$.
\end{lemma}

\begin{proof}
Let us assume that the conclusion is false, that is, there exist constants $\alpha\in(0,1)$, $A>0$ and sequences $F_k$, $u_k$, $f_k$, $g_k$, $\Omega_k$, $\varepsilon_k$, $\nu_k$, $p_k$, $r_k$, and $a_k$ such that $u_k$ is a viscosity solution to 
\begin{equation*} 
\left\{\begin{aligned}
	\partial_t u_k &= (|\nu_k Du_k + p_k| + \varepsilon_k^2)^{\gamma/2} F_k(D^2 u_k) + f_k  && \text{in } \Omega_k\cap Q_1 \\
	u_k&=g_k && \text{on } \partial_p \Omega_k\cap Q_1,
\end{aligned}\right.
\end{equation*}	
where $F_k$ satisfies \textnormal{\ref{F1}}, and 
\begin{equation*}
	0\leq \varepsilon_k, \nu_k \leq 1, \quad 
	1/2 \leq |p_k| \leq 2, \quad 
	[\partial_p\Omega_k]_{C^{1,\alpha}(0,0)} \leq A, \quad
	\|f_k\|_{L^\infty(\Omega_k\cap Q_1)} \leq 1,
\end{equation*}
\begin{equation*}
	\|u_k- a_k x_n \|_{L^\infty(\Omega_k\cap Q_{r_k})} \leq r_k^{1+\alpha}, \quad
	\|g_k\|_{L^{\infty}(\partial_p \Omega_k\cap Q_{r_k})} \leq \frac{r_k^{1+\alpha}}{k}, \quad 
	r_k \leq \frac{1}{k}, \quad \text{and} \quad
	|a_k| \leq \frac{C_0}{k^\alpha}.
\end{equation*}
Moreover, the following estimate holds:
\begin{equation} \label{est:uk-ax}
	\|u_k- a x_n \|_{L^\infty(\Omega_k\cap Q_{\tau r_k})} > (\tau r_k)^{1+\alpha} \quad \text{for all } a \in \mathbb{R}^n \text{ with } |a-a_k| \leq C_0 (\tau r_k)^{\alpha},
\end{equation}
where $C_0>0$ and $\tau \in (0,1)$ will be determined later. 

Then the function
\begin{equation*}
	\tilde{u}_k(x,t) = \frac{u_k(r_kx,r_k^2t) - a_k r_k x_n}{r_k^{1+\alpha}} 
\end{equation*}
is a viscosity solution to 
\begin{equation*} 
\left\{\begin{aligned}
	\partial_t \tilde{u}_k &= (| \tilde{\nu}_k D\tilde{u}+ \tilde{p}_k| + \varepsilon_k^2)^{\gamma/2} \widetilde{F}_k (D^2 \tilde{u}) + \tilde{f}_k  && \text{in } \widetilde{\Omega}_k\cap Q_1 \\
	\tilde{u}_k&=\tilde{g}_k && \text{on } \partial_p \widetilde{\Omega}_k\cap Q_1,
\end{aligned}\right.
\end{equation*}	
where 
\begin{equation*}
	\|\tilde{u}_k\|_{L^\infty(\widetilde{\Omega}_k\cap Q_1)} \leq 1,\quad
	\tilde{\nu}_k=\nu_k r_k^\alpha, \quad 
	\tilde{p}_k= \nu_k a_k e_n  + p_k, \quad
	\widetilde{\Omega}_k= \{(x,t) \mid (r_k x, r_k^2 t) \in \Omega_k\},
\end{equation*}
\begin{equation*}
	\widetilde{F}_k(M) = \frac{F_k(r_k^{\alpha-1} M)}{r_k^{\alpha-1}} , \quad 
	\tilde{f}_k(x,t)= \frac{f_k(r_kx,r_k^2t)}{r_k^{\alpha-1}}, \quad \text{and} \quad
	\tilde{g}_k(x,t)= \frac{g_k(r_kx,r_k^2t)-a_kr_k x_n}{r_k^{1+\alpha}}.
\end{equation*}

Since $r_k \to 0^+$, the Bolzano–Weierstrass theorem ensures that there exist $\tilde{\varepsilon} \in [0,1]$ and $\tilde{p} \in \mathbb{R}^n$ with $1/2 \leq |\tilde{p}| \leq 2$ such that 
\begin{equation*}
	\varepsilon_k \to \tilde{\varepsilon} ,\quad 
	\tilde{\nu}_k \to 0^+,  \quad \text{and} \quad
	\tilde{p}_k \to \tilde{p},
\end{equation*}
where, for notational convenience, we shall use the same index $k$ when referring to convergence up to a subsequence. Furthermore, since each $F_k$ is Lipschitz continuous, the sequence $\{F_k\}$ is equicontinuous. Hence, by the Arzelà--Ascoli theorem, there exists a nonlinear operator $\widetilde{F}$ such that $\widetilde{F}_k \to \widetilde{F}$ uniformly on any compact set of $\mathcal{S}^n$.

By \Cref{lem:uni_lip_full} and \Cref{lem:int_hol_time_full}  the sequence $\{\tilde{u}_k\}$ is equicontinuous. By the Arzelà--Ascoli theorem, there exists a function $\tilde{u}$ such that $\tilde{u}_k \to \tilde{u}$ uniformly on any compact set of $Q_1^+$. Moreover, for any test function $\varphi$ that touches $\tilde{u}$ from above, we have
\begin{equation*}
	\frac{1}{4} \leq | \tilde{\nu}_k D\varphi+ \tilde{p}_k|  \leq 4 \quad \text{for sufficiently large } k.
\end{equation*}
Hence, by \Cref{thm:stability1}, $u$ is a viscosity solution to
\begin{equation*} 
	\tilde{u}_t = (|\tilde{p}| + \tilde{\varepsilon}^2)^{\gamma/2} \widetilde{F} (D^2 \tilde{u})   \quad  \text{in } Q_1^+.
\end{equation*}	

On the other hand, noting that
\begin{equation*}
	\|\tilde{f}_k\|_{L^\infty(\widetilde{\Omega}_k\cap Q_1)} \leq r_k^{1-\alpha}, \quad
	\|\tilde{g}_k\|_{L^\infty(\partial_p\widetilde{\Omega}_k\cap Q_1)} \leq \frac{1}{k} + \frac{AC_0 }{k^\alpha}, \quad \text{and} \quad 
	\underset{Q_1}{\mathrm{osc}} \, \partial_p\widetilde{\Omega}_k \leq 2Ar_k^\alpha,
\end{equation*}
we may apply \Cref{lem:bdry_lip1} to obtain
\begin{equation*}
	|\tilde{u}_k(x,t)| \leq C(x_n + \delta_k) \quad \text{for all } (x,t) \in \widetilde{\Omega}_k \cap Q_{3/4},
\end{equation*}
where $\delta_k \to 0^+$ as $k\to \infty$. Thus, $\tilde{u}$ is a viscosity solution to the uniformly parabolic equation
\begin{equation*} 
\left\{\begin{aligned}
	\tilde{u}_t &= (|\tilde{p}| + \tilde{\varepsilon}^2)^{\gamma/2} \widetilde{F} (D^2 \tilde{u})  &&  \text{in } Q_{3/4}^+ \\
	\tilde{u}&=0 && \text{on } S_{3/4}.
\end{aligned}\right.
\end{equation*}	

By \Cref{thm:LZ22-116}, we have $\tilde{u} \in C^{2,\alpha}(0,0)$ for some $\alpha \in (0,1)$. In particular, there exists a constant $\tilde{a} \in \mathbb{R}$ such that 
\begin{equation} \label{est:order2}
	\|\tilde{u} - \tilde{a} x_n \|_{L^\infty(Q_\tau^+)} \leq C \tau^2  \quad \text{for all } \tau \in (0,1/2)
\end{equation}
and moreover $|\tilde{a}|\leq C$. Choose $\tau>0$ sufficiently small and $C_0>0$ sufficiently large so that
\begin{equation*}
	\|\tilde{u} - \tilde{a} x_n \|_{L^\infty(Q_\tau^+)} \leq \frac{1}{2} \tau^{1+\alpha} \quad \text{and}\quad
	|\tilde{a}| \leq C_0 \tau^\alpha.
\end{equation*}

Now, setting $b_k=a_k + r_k^\alpha \tilde{a}$, we have 
\begin{equation*}
	|b_k - a_k| \leq C_0(\tau r_k)^\alpha.
\end{equation*}
Hence, by \eqref{est:uk-ax}, we obtain
\begin{equation*}
	\|u_k- b_k x_n \|_{L^\infty(\Omega_k \cap Q_{\tau r_k})} > (\tau r_k)^{1+\alpha} ,
\end{equation*}
which, when expressed in terms of $\tilde{u}_k$, becomes 
\begin{equation} \label{contra2}
	\|\tilde{u}_k- \tilde{a} x_n \|_{L^\infty(\widetilde{\Omega}_k \cap Q_{\tau})} > \tau^{1+\alpha}.
\end{equation}
Finally, letting $k \to \infty$ in \eqref{contra2} yields a contradiction to \eqref{est:order2}.
\end{proof}

We are now ready to prove the main theorem of this section.
\begin{proof} [Proof of \Cref{lem:small_bdry}]
Let $C_0$, $\delta_0$, and $\tau$ be the constants from \Cref{lem:small_1st}. Then it suffices to show that there exists a sequence $\{a_k\}_{k=-1}^\infty$ such that for each $k\ge 0$,  
\begin{equation} \label{small_ind}
	\|u-a_kx_n\|_{L^\infty(\Omega \cap Q_{\delta_0 \tau^k})} \leq (\delta_0 \tau^k)^{1+\alpha} \quad \text{and} \quad
	|a_k-a_{k-1}| \leq C_0 (\delta_0 \tau^k)^{\alpha}.
\end{equation}
The proof proceeds by induction. First, set $a_{-1}=0=a_0$. Without loss of generality, we may apply  \Cref{lem:bdry_lip1}. For sufficiently small $\eta \in (0,1)$, we have
\begin{equation*}
	\|u\|_{L^\infty(\Omega \cap Q_{\delta_0})} \leq C\eta ([\partial_p\Omega]_{C^{1,\alpha}(0,0)} + 1) \leq \delta_0^{1+\alpha},
\end{equation*}
and thus \eqref{small_ind} holds for $k=0$.

Suppose that \eqref{small_ind} holds for each $k=0,1,\cdots,i$. Since $\tau <1/2$, we have
\begin{equation*}
	|a_i| \leq \sum_{k=1}^i |a_k-a_{k-1}| \leq C_0\delta_0^\alpha  \frac{\tau^\alpha}{1-\tau^\alpha} \leq C_0\delta_0^\alpha.
\end{equation*}
Furthermore, since $g(0,0)=0=|Dg(0,0)|$, choose $\eta>0$ sufficiently small so that
\begin{equation*}
	\|g\|_{L^{\infty}(\partial_p \Omega \cap Q_r)} 
	\leq C \|g\|_{C^{1,\alpha}(0,0)} r^{1+\alpha} 
	\leq C\eta r^{1+\alpha}
	\leq \delta_0 r^{1+\alpha},
\end{equation*}
where $r=\delta_0 \tau^i \leq \delta_0$. Then, by \Cref{lem:small_1st}, we deduce the existence of a constant $a_{i+1}\in\mathbb{R}$ satisfying 
\begin{equation*}
	\|u- a_{i+1} x_n \|_{L^\infty(\Omega\cap Q_{\tau r})} \leq (\tau r)^{1+\alpha}   \quad \text{and} \quad 
	|a_{i+1}-a_i| \leq C_0 (\tau r)^{\alpha} .
\end{equation*}
Consequently, \eqref{small_ind} holds for $k=i+1$ as well, completing the proof by induction.
\end{proof}

\section{The Dirichlet problem with zero boundary data on a flat boundary} \label{sec:model_problem}
While most earlier contributions to boundary regularity employed flattening transformations, Lian--Zhang \cite{LZ20, LZ22} proposed a novel approach that establishes boundary regularity directly without flattening the boundary. We shall adopt their method to establish boundary regularity. To this end, we first consider the following model problem with zero boundary data on a flat boundary.
\begin{equation} \label{eq:model} 
\left\{\begin{aligned}
	u_t &= (|\nu Du + p|^2 + \varepsilon^2)^{\gamma/2} F(D^2 u)  && \text{in } Q_1^+ \\
	u&=0 && \text{on } S_1.
\end{aligned}\right.
\end{equation}

The Boundary Lipschitz estimates can be regarded as the most basic ones for establishing boundary $C^{1,\alpha}$-regularity, and we have already derived them for \eqref{eq:model}.
\begin{remark} \label{rmk:b_lip}
In \Cref{lem:bdry_lip1}, when $\delta=0$, the problem \eqref{eq:ep-nu-p} reduces to problem \eqref{eq:model}, and therefore under the assumptions of \Cref{lem:bdry_lip1}, we have
\begin{equation*}
	|u(x,t)| \leq Cx_n \quad \text{for all } (x,t) \in Q_{1/2}^+.
\end{equation*}
\end{remark}
\begin{remark} \label{rmk:normal}
Let $u$ be a viscosity solution to \eqref{eq:model}. Then the function
\begin{equation*}
	\tilde{u}(x,t) = r\rho u(r^{-1}x, r^{-2} \rho^{\gamma}t)
\end{equation*}
is a viscosity solution to 
\begin{equation*}
\left\{\begin{aligned}
	\tilde{u}_t &= (|\nu D\tilde{u} + \rho p|^2 + \rho^{2} \varepsilon^2)^{\gamma/2} \widetilde{F}(D^2 \tilde{u})  && \text{in } Q_r^{\rho+} \\
	\tilde{u}&=0 && \text{on } S_r,
\end{aligned}\right.
\end{equation*}
where $\widetilde{F}(M) =r^{-1}\rho F(r\rho^{-1} M)$. Hence, by choosing $\rho \leq 1/(1+ \|Du\|_{L^\infty(Q_1)})$, we may assume
\begin{equation*}
	\|Du\|_{L^\infty(Q_1)} \leq 1.
\end{equation*}
\end{remark} 

To obtain global Lipschitz estimates, it is essential to first establish H\"older continuity in time. For this purpose, we employ a preliminary result which corresponds to \cite[Lemma 4.3]{LLY24}.
\begin{lemma} \label{lem43-LLY24}
Assume that $F$ satisfies \textnormal{\ref{F1}} and let $u$ be a viscosity solution to 
\begin{equation*} 
	u_t = (|Du|^2+\varepsilon^2)^{\gamma/2} F(D^2u) \quad \text{in } Q_1 .
\end{equation*}
Then
\begin{equation*}
	|u(x,t)-u(x,s)| \leq C|t-s|^{1/2} \quad \text{for all } (x,t), (x,s) \in Q_{1/2},
\end{equation*}
where $C>0$ is a constant depending only on $n$, $\lambda$, $\Lambda$, $\gamma$, and $\|u\|_{L^\infty(Q_1)}$.
\end{lemma}
\begin{lemma} \label{lem:G_hol_t}
Assume that $F$ satisfies \textnormal{\ref{F1}} and
\begin{equation*}
	|p| \leq 1 \quad \text{and} \quad 
	\|u\|_{L^\infty(Q_1^+)} \leq 1.
\end{equation*}
Let $u$ be a viscosity solution to \eqref{eq:model} with $\nu=1$. Then
\begin{equation*}
	|u(x,t) -u(y,s)| \leq C(|x-y| + |t-s|^{1/2}) \quad \text{for all } (x,t), (y,s) \in Q_{1/2}^+,
\end{equation*}
where $C>0$ is a constant depending only on $n$, $\lambda$, $\Lambda$, and $\gamma$.
\end{lemma}

\begin{proof}
Let $(x,t),(y,s) \in Q_{1/2}^+$ with $t \leq s $ and let $r=y_n$. Without loss of generality, assume $y_n \ge x_n$. Then by \Cref{rmk:eq_tr}, \Cref{lem:int-lip}, and \Cref{lem43-LLY24}, we have
\begin{equation}\label{est1}
	|u(x,t)-u(y,s)| \leq \widetilde{C}(|x-y| + |t-s|^{1/2}) \quad \text{for all } (x,t) \in Q_{r/2}(y,s),
\end{equation}
where $\widetilde{C}>0$ is a constant depending only on $n$, $\lambda$, $\Lambda$, $\gamma$, and $r^{-1}\|u\|_{L^\infty(Q_r(y,s))}$. In particular, by \Cref{lem:bdry_lip1}, we have 
\begin{equation*}
	|u(x,t)| \leq \widehat{C}x_n \leq  \widehat{C}r \quad \text{for all } (x,t) \in Q_r(y,s),
\end{equation*}
where $\widehat{C}>0$ is determined in \Cref{lem:bdry_lip1}. Hence, $\widetilde{C}$ does not depend on $r^{-1}\|u\|_{L^\infty(Q_r(y,s))}$.

Moreover, by \Cref{lem:bdry_lip1},  we have
\begin{equation} \label{est2}
	|u(x,t)-u(y,s)| 
	\leq 2\widehat{C} r
	\leq C(|x-y| + |t-s|^{1/2}) \quad \text{for all } (x,t) \in Q_{1/2}^+ \setminus Q_{r/2}(y,s).
\end{equation}
Finally, combining \eqref{est1} and \eqref{est2}, we obtain the desired conclusion.
\end{proof}
\subsection{Boundary $C^{1,\alpha}$-estimates for large $|p|$}
When boundary regularity results are obtained via the construction of a barrier function, if the gradient vector of the barrier function is chosen sufficiently small, then in the range 
\begin{equation*}
	1/2 \leq |p| \leq1,
\end{equation*}
the equation in \eqref{eq:model} can be regarded as uniformly parabolic. Accordingly, we aim to extend the main regularity results valid for uniformly parabolic equations to problem \eqref{eq:model} in this range.
\begin{lemma} [Strong maximum principle] \label{lem:SMP}
Assume that $F$ satisfies \textnormal{\ref{F1}} and 
\begin{equation*}
	0\leq \nu \leq 1  \quad \text{and} \quad
	1/2 \leq |p| \leq 1.
\end{equation*}
Let $u$ be a nonnegative viscosity solution to 
\begin{equation} \label{eq:gen_f=0} 
	u_t = (|\nu Du + p|^2 + \varepsilon^2)^{\gamma/2} F(D^2 u)  \quad \text{in } Q_1 .
\end{equation}
If $u(0,0) =0$, then 
\begin{equation*}
	u(x,t)=0 \quad \text{for all }(x,t) \in Q_1.
\end{equation*}
\end{lemma}
\begin{proof}
The proof is by contradiction, that is, we assume that the conclusion does not hold. Then the nonempty set $S= \{(x,t) \in Q_1 : u(x,t)=0\}$ is a closed subset of $Q_1$, so there exist $r,\rho>0$ and $(x_0,t_0) \in Q_1 \setminus S$ satisfying
\begin{equation*}
	Q_{r}(x_0,t_0) \subset Q_1 \setminus S, \quad
	B_{\rho}(x_0) \times (t_0-r^2, t_0] \subset Q_1, \quad \text{and} \quad
	(\partial B_\rho (x_0) \times \{t_0\}) \cap S \neq \varnothing.
\end{equation*}

Consider the function
\begin{equation*}
	\phi(x,t)= \frac{1}{A} \left(e^{-\frac{A|x-x_0|^2}{t-t_0+r^2}} - e^{-\frac{A\rho^2}{r^2}} \right).
\end{equation*}
Choosing $A>0$ sufficiently large ensures that $|D\phi|$ becomes small, in fact, 
\begin{equation*}
	|D\phi| \leq \frac{\sqrt{2}}{r\sqrt{A}}e^{-1/2} \quad \text{in } \Omega = \{x\in B_1: r < |x-x_0| < \rho\}  \times (t_0-r^2, t_0] .
\end{equation*}
Consequently, $\phi$ satisfies
\begin{equation*} 
\left\{\begin{aligned}
	\phi_t &< (|\nu D\phi + p|^2 + \varepsilon^2)^{\gamma/2} F(D^2\phi) && \text{in } \Omega \\
	\phi & \leq u && \text{on }  \partial B_{r}(x_0) \times [t_0-r^2, t_0] \\
	\phi & \leq 0 && \text{on }  \partial B_{\rho}(x_0) \times [t_0-r^2, t_0] \\
	\phi & \leq 0 && \text{on }  \{t=t_0-r^2\}.
\end{aligned}\right.
\end{equation*}
Therefore, \Cref{thm:comp_prin} yields
\begin{equation} \label{area_in}
	\phi (x,t) \leq u(x,t) \quad \text{for all } (x,t) \in \Omega.
\end{equation}

Since $(\partial B_\rho (x_0) \times \{t_0\}) \cap S \neq \varnothing$, take some $(y,t_0) \in (\partial B_\rho (x_0) \times \{t_0\}) \cap S$. Then 
\begin{equation}\label{touch_pt}
	\phi (y,t_0)=0=u(y,t_0).
\end{equation}
Moreover,
\begin{equation*} 
	\phi (x,t) \leq 0 \leq u(x,t) \quad \text{for all } (x,t) \in \{x\in B_1: |x-x_0|\ge \rho\} \times [t_0-r^2, t_0],
\end{equation*}
and combining this with \eqref{area_in} and \eqref{touch_pt} show that $\phi$ touches $u$ from below at $(y,t_0)$, which contradicts the fact that
\begin{equation*}
	\phi_t < (|\nu D\phi + p|^2 + \varepsilon^2)^{\gamma/2} F(D^2\phi) \quad \text{in a neighborhood of } (y,t_0).
\end{equation*}
\end{proof}
\begin{lemma} [Harnack inequality] \label{harnack}
Under the assumptions of \Cref{lem:SMP}, let us further assume $\|u\|_{L^\infty(Q_1^+)} \leq 2$ and $u(0, -1/2) \ge 1$. Let $u$ be a nonnegative viscosity solution to \eqref{eq:gen_f=0}. Then 
\begin{equation*}
	u(x,t) \ge \mu  \quad \text{for all } (x,t) \in Q_{1/2},
\end{equation*}
where $\mu>0$ is a constant depending only on $n$, $\lambda$, $\Lambda$, and $\gamma$.
\end{lemma}
\begin{proof}
	The proof is by contradiction, that is, we assume that the conclusion does not hold. Then there exist sequences $F_k$, $u_k$, $\varepsilon_k$, $\nu_k$, $p_k$, and $(x_k, t_k) \in Q_{1/2}$ such that $u_k$ is a nonnegative viscosity solution to 
\begin{equation*}
	\partial_t u_k = (|\nu_k Du_k + p_k|^2 + \varepsilon_k^2)^{\gamma/2} F_k(D^2 u_k)  \quad \text{in } Q_1 ,
\end{equation*}
where $F_k$ satisfies \textnormal{\ref{F1}}, and 
\begin{equation*}
	0\leq \varepsilon_k, \nu_k \leq 1, \quad
	1/2 \leq |p_k| \leq 1, \quad 
	\|u_k\|_{L^\infty(Q_1^+)} \leq 2,
\end{equation*}
\begin{equation*}
	u_k(0, -1/2) \ge 1,  \quad \text{and} \quad
	\lim_{k \to \infty} u_k(x_k,t_k) = 0.
\end{equation*}

As in the proof of \Cref{lem:small_1st}, there exist a point $(\tilde{x},\tilde{t}) \in \overline{Q_{1/2}}$ and a function $\tilde{u}$ such that $(x_k,t_k) \to (\tilde{x},\tilde{t})$ as $k \to \infty$ and $\tilde{u}$ is a nonnegative solution to
\begin{equation*}
	\tilde{u}_t = (|\tilde{\nu} D\tilde{u} + \tilde{p}|^2 + \tilde{\varepsilon}^2)^{\gamma/2} \widetilde{F}(D^2 \tilde{u})  \quad \text{in } Q_{3/4} ,
\end{equation*}
where  $0\leq \tilde{\varepsilon}, \tilde{\nu} \leq 1$ and $1/2 \leq |\tilde{p}| \leq 1$.

However, since $\tilde{u}(\tilde{x},\tilde{t})=0$ and $-1/4 \leq \tilde{t} \leq 0$, by \Cref{lem:SMP}, we have 
\begin{equation*}
	\tilde{u}(x,t) =0 \quad \text{for all } (x,t) \in Q_{3/4} \text{ with } t < \tilde{t},
\end{equation*}
which contradicts the fact that $u(0, -1/2) \ge 1$.
\end{proof}
\begin{lemma} [Hopf principle] \label{hopf}
Under the assumptions of \Cref{lem:SMP}, let us further assume $\|u\|_{L^\infty(Q_1^+)} \leq 2$ and $u(e_n/2, -3/4) \ge 1$. Let $u$ be a nonnegative viscosity solution to \eqref{eq:model}. Then 
\begin{equation*}
	u(x,t) \ge c_0 x_n \quad \text{for all } (x,t) \in Q_{1/2}^+,
\end{equation*}
where $c_0 \in (0,1)$ is a constant depending only on $n$, $\lambda$, $\Lambda$, and $\gamma$.
\end{lemma}
\begin{proof}
By \Cref{harnack}, there exists a constant $\mu>0$ such that 
\begin{equation*}
	u(x,t) \ge \mu \quad \text{for all } (x,t) \in B_{1/4}(e_n/2) \times (-1/4, 0].
\end{equation*}

If we replace $\phi$ in the proof of \Cref{lem:SMP} with
\begin{equation*}
	\phi(x,t)= \frac{1}{A} \left(e^{-\frac{A|x-e_n/2|^2}{t+1/4}} - e^{-A} \right),
\end{equation*}
then $\phi$ satisfies
\begin{equation*} 
\left\{\begin{aligned}
	\phi_t &< (|\nu D\phi + p|^2 + \varepsilon^2)^{\gamma/2} F(D^2\phi) && \text{in } \Omega \\
	\phi & \leq \mu && \text{on }  \partial B_{1/4}(e_n/2) \times [-1/4, 0] \\
	\phi & \leq 0 && \text{on }  \partial B_{1/2}(e_n/2) \times [-1/4, 0] \\
	\phi & \leq 0 && \text{on }  \{t=-1/4\}, 
\end{aligned}\right.
\end{equation*}
where $\Omega = \{x\in B_1: 1/4 < |x-e_n/2| < 1/2\}  \times (-1/4, 0]$. Therefore, \Cref{thm:comp_prin} yields
\begin{equation*}
	u(0',x_n,0) \ge \phi(0',x_n,0) \ge 4e^{-A}x_n  \quad \text{for all } x_n  \in (0, 1/4).
\end{equation*}
The remaining part of the proof is the same as that of \Cref{lem:bdry_lip1}.
\end{proof}
\Cref{lem:bdry_lip1} and \Cref{hopf} play a crucial role in the derivation of the following $C^{1,\alpha}$-estimates.
\begin{lemma} [Boundary $C^{1,\alpha}$-estimates for large $|p|$] \label{lem:c1a_large_p}
Under the assumptions of \Cref{lem:SMP}, let us further assume $p_n=0$ and $\|Du\|_{L^\infty(Q_1^+)} \leq 1$. Let $u$ be a viscosity solution to \eqref{eq:model}. Then $u \in C^{1,\alpha}(0,0)$ for some $\alpha\in(0,1)$, that is, there exists a constant $a \in \mathbb{R}$ such that 
\begin{equation*}
	|u(x,t) - ax_n| \leq C x_n (|x|^\alpha + |t|^{\alpha/2}) \quad \text{for all } (x,t) \in Q_{1/2}^+,
\end{equation*}
where $C>0$ is a constant depending only on $n$, $\lambda$, $\Lambda$, and $\gamma$.
\end{lemma}

\begin{proof}
We first construct an increasing sequence $\{a_k\}_{k=0}^\infty$ and a decreasing sequence $\{b_k\}_{k=0}^\infty$ satisfying $a_0=-2$, $b_0=2$ and for each $k\ge 1$,
\begin{equation} \label{ind_c1a}
	a_k x_n \leq u \leq b_k x_n \quad \text{in }  Q_{2^{-k}}^+ \quad \text{and} \quad
	0\leq b_k - a_k \leq 2^{-\alpha}(b_{k-1}-a_{k-1}),
\end{equation}
where $a \in (0,1)$ is a constant depending only on $n$, $\lambda$, $\Lambda$, and $\gamma$. 

The proof will proceed by induction. First, in the case $k=1$, since $\|Du\|_{L^\infty(Q_1^+)} \leq 1$, we may choose $a_1=-1$ and $b_1=1$, and then \eqref{ind_c1a} holds. Next, assume that \eqref{ind_c1a} holds for $k=i$. If $a_i=b_i$, then by setting $a_{i+1}=a_i$ and $b_{i+1}=b_i$, we see that \eqref{ind_c1a} holds for $i+1$. Hence, without loss of generality, we may assume that $a_i > b_i$.

Let $r_i=2^{-i}$. By the induction hypothesis, there are two possible cases:
\begin{equation*}
	\text{either} \quad u\left(\frac{1}{2} r_i e_n, -\frac{3}{4} r_i^2 \right) \ge \frac{1}{4} r_i(a_i + b_i ) \quad
	\text{or} \quad  u\left(\frac{1}{2} r_ie_n, -\frac{3}{4} r_i^2 \right) < \frac{1}{4} r_i(a_i + b_i ).
\end{equation*}
Since the proofs of the two cases are similar, we only consider the first case. Then the function 
\begin{equation*}
	v(x,t) = \frac{4}{(b_i-a_i)r_i} (u(r_ix,r_i^2t) - a_i r_i x_n) 
\end{equation*}
satisfies 
\begin{equation*}
	\|v\|_{L^\infty( Q_1^+ )} \leq 4, \quad 
	v\left(\frac{1}{2} e_n, -\frac{3}{4} \right) \ge 1, \quad \text{and} \quad
	v_t = (|\tilde{\nu} Dv + \tilde{p}|^2 + \varepsilon^2)^{\gamma/2} \widetilde{F}(D^2 v)  \quad \text{in } Q_1^+,
\end{equation*}
where
\begin{equation*}
	\widetilde{F}(M) =\frac{4r_i}{b_i-a_i} F\left(\frac{b_i-a_i}{4r_i} M\right), \quad
	\tilde{\nu}=  \frac{b_i-a_i}{4} \nu, \quad\text{and}\quad
	\tilde{p}= p+ \nu a_i e_n.
\end{equation*}

Since $|a_i|\leq 1$ and $|b_i| \leq1$, it follows that 
\begin{equation*}
	0\leq \tilde{\nu} \leq \frac{1}{2} \quad \text{and} \quad 
	\frac{1}{2} \leq |p| \leq |\tilde{p}| \leq 2.
\end{equation*}
Hence, by applying \Cref{hopf} to $v$, it follows that
\begin{equation*}
	v(x,t) \ge c_0 x_n \quad \text{for all }  (x,t) \in Q_{1/2}^+,
\end{equation*}
where $c_0 \in (0,1)$ is determined in \Cref{hopf}. Now, rescaling back to $u$, we obtain
\begin{equation*}
	u(x,t) \ge \left( a_i  + \frac{b_i-a_i}{4}  c_0 \right)x_n \quad \text{for all }  (x,t) \in Q_{r_{i+1}}^+.
\end{equation*}
By setting 
\begin{equation*}
	a_{i+1}= a_i  + \frac{b_i-a_i}{4}  c_0, \quad
	b_{i+1} = b_i \quad \text{and}\quad
	\alpha = -\log_2 \left(1-\frac{c_0}{4}\right),
\end{equation*}
we obtain that \eqref{ind_c1a} also holds for $k=i+1$.

Finally, this construction ensures that there exists a constant $a \in \mathbb{R}$ such that 
\begin{equation*}
	\lim_{k\to \infty} a_k =a = \lim_{k\to \infty} b_k  \quad \text{and} \quad
	|a| \leq 1.
\end{equation*}
Moreover, for any $(x,t) \in Q_{1/2}^+$, by choosing $k \in \mathbb{N}$ satisfying $r_{k+1} < \max\{|x|, \sqrt{|t|}\} \leq r_k$, we obtain
\begin{equation*}
	u(x,t) - ax_n \leq (b_k-a_k)x_n \leq 4r_k^{\alpha} x_n \leq Cx_n (|x|^\alpha + |t|^{\alpha/2}).
\end{equation*}
Similarly, we obtain the lower bound.
 \end{proof}
 The $C^{1,\alpha}$-estimates established in \Cref{lem:c1a_large_p} can be improved as follows for any $
 \alpha \in (0,1)$.
\begin{lemma} [Boundary $C^{1,1^-}$-estimates for large $|p|$]   \label{lem:c11_large_p}
Under the assumptions of \Cref{lem:c1a_large_p}, let $u$ be a viscosity solution to \eqref{eq:model}. Then $u \in C^{1,\alpha}(0,0)$ for all $\alpha\in(0,1)$, that is, there exists a constant $a \in \mathbb{R}$ such that 
\begin{equation*}
	|u(x,t) - ax_n| \leq C(|x|^{1+\alpha} + |t|^{\frac{1+\alpha}{2}}) \quad \text{for all } (x,t) \in Q_{1/2}^+,
\end{equation*}
where $C>0$ is a constant depending only on $n$, $\lambda$, $\Lambda$, $\alpha$, and $\gamma$.
\end{lemma}
\begin{proof}
By \Cref{lem:c1a_large_p}, there exists a constant $\tilde{a} \in \mathbb{R}$ such that 
\begin{equation*}
	|u(x,t) - \tilde{a} x_n| \leq \widetilde{C} x_n (|x|^{\tilde{\alpha}} + |t|^{\tilde{\alpha}/2}) \quad \text{for all } (x,t) \in Q_{1/2}^+,
\end{equation*}
where $\tilde{\alpha}\in(0,1)$ and $\widetilde{C}>0$ are determined in \Cref{lem:c1a_large_p}. Then for $r \in (0,1/2)$, the function
\begin{equation*}
	v(x,t) = \frac{u(rx,r^2t) - ar x_n}{r} 
\end{equation*}
is a viscosity solution to 
\begin{equation*} 
\left\{\begin{aligned}
	v_t &=(|\nu Dv + \tilde{p}|^2 + \varepsilon^2)^{\gamma/2} \widetilde{F}(D^2 v)  && \text{in }Q_1^+ \\
	v&=0 && \text{on }S_1,
\end{aligned}\right.
\end{equation*}	
where 
\begin{equation*}
	\widetilde{F}(M) =rF(r^{-1} M ) \quad\text{and}\quad
	\tilde{p}= p+ \nu a e_n.
\end{equation*}
Moreover, since $1/2 \leq |\tilde{p}| \leq 2$, if we choose $r$ sufficiently small with respect to $\eta$ in \Cref{lem:small_bdry} so that
\begin{equation*}
	\|v\|_{L^\infty(Q_1^+)} \leq 2C r^{\tilde{\alpha}} \leq \eta,
\end{equation*}
then by \Cref{lem:small_bdry}, for any $\alpha \in (0,1)$, there exists a constant $a \in \mathbb{R}$ such that 
\begin{equation*}
	|v(x,t) - ax_n| \leq C (|x|^{1+\alpha} + |t|^{\frac{1+\alpha}{2}}) \quad \text{for all } (x,t) \in Q_{1/2}^+,
\end{equation*}
where $C>0$ is a constant depending only on $n$, $\lambda$, $\Lambda$, $\alpha$, and $\gamma$. 

Finally, after rescaling back to $u$, we arrive at the desired conclusion.
\end{proof}
\subsection{Boundary $C_\gamma^{1,\alpha}$-estimates on a flat boundary} \label{sec:bdry_reg_flat}
In the previous subsection, we established boundary $C^{1,\alpha}$-regularity under the assumption that $|p|$ is bounded away from zero. We now aim to extend these results to the full range $|p| \leq 1$. To this end, we differentiate the equation with respect to spatial variables in order to derive an equation for $Du$, which requires sufficient regularity of both the solution $u$ and the operator $F$. To ensure this level of regularity, we impose the structural condition \ref{F2} and consider the regularized operator $F^\varepsilon$ introduced in \Cref{sec:regul_F}. In addition, to ensure the well-posedness of the differentiated equation and control degeneracy or singularity, we also consider the regularized problem with $\varepsilon>0$.
\begin{equation} \label{eq:model2} 
\left\{\begin{aligned}
	u_t &= (|Du + p|^2 + \varepsilon^2)^{\gamma/2} F^\varepsilon(D^2 u)  && \text{in } Q_1^+ \\
	u&=0 && \text{on } S_1.
\end{aligned}\right.
\end{equation}

The following lemma can be regarded as an extension of \cite[Lemma 4.4]{LLY24} to the case where $p\neq 0$. Under the assumptions $|p|\leq 1/2$ and $l \in (3/4,1)$, the conclusion follows directly by applying \cite[Lemma 4.4]{LLY24} to the function $v=u+p\cdot x$, together with \Cref{rmk:eq_tr}.
\begin{lemma}\label{lem44-LLY24}
Assume that $F$ satisfies \textnormal{\ref{F1}}, \textnormal{\ref{F2}},
\begin{equation*}
	|p| \leq 1/2 \quad \text{and} \quad
	\|Du\|_{L^\infty(Q_1^+)} \leq 1.
\end{equation*}
Let $u$ be a viscosity solution to 
\begin{equation} \label{eq:v=1}
	u_t = (|Du+p|^2+\varepsilon^2)^{\gamma/2} F^\varepsilon(D^2u) \quad \text{in } Q_1.
\end{equation}
Then for any $l \in (3/4,1)$ and $\mu>0$, there exist constants $\tau, \delta>0$ depending only on $n$, $\lambda$, $\Lambda$, $\gamma$, $l$, and $\mu$ such that if 
\begin{equation*}
	|\{(x,t) \in Q_1: D_q u (x,t) \leq l \}| > \mu |Q_1|  \quad  \text{for all unit vector } q\in \mathbb{R}^n,
\end{equation*}
then
\begin{equation*}
	D_q u < 1 -\delta \quad \text{in } Q_\tau^{1-\delta}.
\end{equation*}
\end{lemma}
Suppose that the subset of $Q_1^+$ where $\pm u_n$ stays strictly below 1 has positive measure. 
In this case, the oscillation of $Du$ can be improved in a smaller cylinder. To achieve this, we construct an auxiliary function $w$ based on the gradient vector and show that it is a subsolution to a uniformly parabolic equation in $Q_{1/2}^+$. We then extend $w$ by zero to $Q_{1/2}$. To ensure that this extension is well-defined, particularly that $w$ vanishes along $S_{1/2}$, we estimate $\pm u_n$ on $S_{1/2}$. The details of this step are provided in the following argument.
\begin{lemma}  \label{lem:decay}
Under the assumptions of \Cref{lem44-LLY24}, let $u$ be a viscosity solution to  \eqref{eq:model} with $\nu=1$. Then for any $l \in (3/4,1)$ and $\mu \in (0,1)$, there exists a constant $\theta \in (0, 1-l)$ depending only on $n$, $\lambda$, $\Lambda$, $\gamma$, $l$, and $\mu$ such that if 
\begin{equation*}
	|\{(x,t) \in Q_1^+: D_e u (x,t) \leq l \}| > \mu |Q_1^+|  \quad \text{for all } e \in \{e_n, -e_n\},
\end{equation*}
then
\begin{equation*}
	D_e u < l+ \theta \quad \text{on } S_{1/2}.
\end{equation*}
\end{lemma}

\begin{proof}
Since the arguments are analogous, it suffices to consider the case $e=e_n$. 
\begin{equation} \label{u_upper1}
	u(x,t) = \int_0^{x_n} u_n(x',z,t)\,dz \leq x_n \quad \text{for all }  (x,t) \in Q_1^+.
\end{equation}
By \Cref{lem44-LLY24}, there exist $\tau, \delta \in (0,1)$ depending only on $n$, $\lambda$, $\Lambda$, $\gamma$, $l$, and $\mu$ such that
\begin{equation*}
	u_n(x,t) < 1 - \delta \quad \text{for all } (x,t) \in \mathcal{Q}_1,
\end{equation*}
where $\mathcal{Q}_r= \{ (x,t) \in Q_1^+: |x'| < \tau,~ |x_n-1/2| < r \tau,~ -(1-\delta)^{-\gamma}\tau^2 < t \leq 0\}$. Then for any $(x,t) \in \mathcal{Q}_{1/2}$, we have 
\begin{align}
	u(x,t) &= \int_{0}^{1/2-\tau} u_n (x',z,t)\,dz +  \int_{1/2-\tau}^{x_n} u_n (x',z,t)\,dz \nonumber \\
	&< \frac{1}{2} -\tau + (1-\delta) \left(x_n -\frac{1}{2} + \tau\right) \nonumber \\
	&\leq x_n -\frac{1}{2} \tau \delta. \label{u_upper2}
\end{align}

If we replace $\phi$ in the proof of \Cref{lem:SMP} with
\begin{equation*}
	\phi(x,t)= \frac{1}{A} \left(e^{-\frac{A|x-e_n/2|^2}{t+(1-\delta)^{-\gamma}\tau^2}} - e^{-\frac{A}{4(1-\delta)^{-\gamma}\tau^2}} \right),
\end{equation*}
then $v= x_n - \phi$ satisfies
\begin{equation*} 
\left\{\begin{aligned}
	v_t &> (|Dv + p|^2 + \varepsilon^2)^{\gamma/2} F(D^2v) && \text{in } Q \setminus \overline{\mathcal{Q}_{1/2}} \\
	v & \ge u && \text{on }  \partial_p \mathcal{Q}_{1/2} \\
	v & \ge u  && \text{on }   \partial_p Q, 
\end{aligned}\right.
\end{equation*}
where $Q = B_{1/2}(e_n/2) \times (-(1-\delta)^{-\gamma}\tau^2 , 0]$. Here, the boundary comparison follows from \eqref{u_upper1} and \eqref{u_upper2}. Therefore, \Cref{thm:comp_prin} yields
\begin{equation*}
	u \leq v    \quad  \text{in } Q \setminus \overline{\mathcal{Q}_{1/2}}.
\end{equation*}
Since $u(0,0) = 0 = v(0,0)$, we have 
\begin{equation*}
	u_n(0,0) \leq v_n (0,0) = 1 - \phi_n(0,0).
\end{equation*}
If $A>0$ is taken sufficiently large, then $|D\phi|$ becomes sufficiently small. Hence, by employing the translation argument from the proof of \Cref{lem:bdry_lip1}, there exists a constnat $\theta \in (0,1-l)$ such that
\begin{equation*}
	u_n < l+\theta \quad \text{on } S_{1/2}.
\end{equation*} 
\end{proof}
\begin{lemma} \label{lem:Du<1-d}
Under the assumptions of \Cref{lem44-LLY24}, let $u$ be a viscosity solution to \eqref{eq:model2}. Then for any $l \in (3/4,1)$ and $\mu \in (0,1)$, there exist constants $\delta,\tau\in (0, 1/4)$ depending only on $n$, $\lambda$, $\Lambda$, $\gamma$, $l$, and $\mu$ such that if 
\begin{equation*}
	|\{(x,t) \in Q_1^+: D_e u (x,t) \leq l \}| > \mu |Q_1^+|  \quad \text{for all } e \in \{e_n, -e_n\},
\end{equation*}
then
\begin{equation*}
	|D u| < 1-\delta \quad \text{in } Q_\tau^{(1-\delta)+}.
\end{equation*}
\end{lemma}

\begin{proof}
We aim to show that 
\begin{equation*}
	D_q u < 1-\delta \quad \text{in } Q_\tau^{(1-\delta)+} \quad \text{for all unit vector } q \in \mathbb{R}^n.
\end{equation*}

By \Cref{lem:decay}, there exists $\theta \in (0,1-l)$ such that
\begin{equation} \label{Dq_decay}
	D_e u < l+ \theta \quad \text{on } S_{1/2} \quad \text{for all } e \in \{e_n, -e_n\}.
\end{equation}
\Cref{thm:LLY24-12} and \Cref{thm:LLY24-412}, combined with \Cref{rmk:eq_tr}, implies that $u$ is a smooth solution to 
\begin{equation*}
	u_t = (|Du + p|^2 + \varepsilon^2)^{\gamma/2} F^\varepsilon(D^2 u) \quad \text{in } Q_{3/4}^+.
\end{equation*}
Let $v=|Du|^2$ and consider the function
\begin{equation*}
	w = \max\{D_q u - \tilde{l} + \rho v, 0\} \quad \text{for } 
	\tilde{l} = \frac{1}{2}(1+l+\theta) \text{ and } 
	\rho = \frac{1}{4} (1-l-\theta).
\end{equation*}
Then in the region $\Omega_+ =\{(x,t) \in Q_{3/4}^+: w(x,t)>0\}$, we have
\begin{equation*}
	|Du+p| \ge |Du|- |p|> \frac{3}{4} -\frac{1}{2} = \frac{1}{4}.
\end{equation*}

Now, let us show that $w$ is a subsolution to a fully nonlinear uniformly parabolic equation. Differentiating the equation in \eqref{eq:model2} with respect to $x_k$, we obtain
\begin{equation} \label{eq:uk}
	\partial_ t u_k = (|Du + p|^2 + \varepsilon^2)^{\gamma/2} F_{ij}^\varepsilon(D^2 u)D_{ij}u_k + \gamma(|Du + p|^2 + \varepsilon^2)^{\gamma/2-1}  F^\varepsilon(D^2 u) (Du + p) \cdot Du_k,
\end{equation}
where $F_{ij}= \partial  F/ \partial M_{ij}$. Then we have
\begin{equation} \label{eq:Dqu-l}
\begin{aligned}
	\partial_t (D_q u - \tilde{l}) &=(|Du + p|^2 + \varepsilon^2)^{\gamma/2} F_{ij}^\varepsilon(D^2 u)D_{ij} (D_q u - \tilde{l}) \\
	&\quad +  \gamma(|Du + p|^2 + \varepsilon^2)^{\gamma/2-1}F^\varepsilon(D^2 u)  (Du + p) \cdot D(D_q u - \tilde{l})  .
\end{aligned}
\end{equation}
Moreover, multiplying both sides of \eqref{eq:uk} by $2u_k$ and summing over $k=1$ to $n$, we obtain
\begin{equation} \label{eq:vt}
\begin{aligned}
	v_ t  &= (|Du + p|^2 + \varepsilon^2)^{\gamma/2} F_{ij}^\varepsilon(D^2 u) D_{ij}v \\
	&\quad + \gamma(|Du + p|^2 + \varepsilon^2)^{\gamma/2-1}F^\varepsilon(D^2 u) (Du + p) \cdot Dv \\
	&\quad - 2(|Du + p|^2 + \varepsilon^2)^{\gamma/2} F_{ij}^\varepsilon(D^2 u) Du_i \cdot  Du_j.
\end{aligned}
\end{equation}

Since $1/4 < |Du+p| \leq 3/2 $ in $\Omega_+$ and 
\begin{equation*}
	\lambda \|D^2u\|^2  \leq F_{ij}^\varepsilon (D^2 u) Du_i \cdot  Du_j \leq \Lambda \|D^2u\|^2 \quad \text{for all } \xi \in \mathbb{R}^n,
\end{equation*}
it follows from \eqref{eq:Dqu-l}, \eqref{eq:vt}, and Young’s inequality that 
\begin{align*}
	w_ t  &\leq (|Du + p|^2 + \varepsilon^2)^{\gamma/2} F_{ij}^\varepsilon(D^2 u) D_{ij} w \\
	&\quad + C_1(|Du + p|^2 + \varepsilon^2)^{\gamma/2-1} \|D^2 u\| |Dw| - 2\lambda(|Du + p|^2 + \varepsilon^2)^{\gamma/2} \|D^2u\| \\
	&\leq (|Du + p|^2 + \varepsilon^2)^{\gamma/2} F_{ij}^\varepsilon(D^2 u) D_{ij} w  + \frac{C_1^2}{8\lambda}(|Du + p|^2 + \varepsilon^2)^{\gamma/2-2}|Dw|^2  \\
	&\leq \mathcal{M}^+_{\lambda^*,\Lambda^*}(D^2 w) + C_2 |Dw|^2 \quad \text{in } \Omega^+, 
\end{align*}
where $\lambda^*$ and $\Lambda^*$ are constants depending only on $\lambda$, $\Lambda$, and $\gamma$; $C_1>0$ is a constant depending only on $n$, $\Lambda$, and $\gamma$; and $C_2>0$  is a constant depending only on $n$, $\lambda$, $\Lambda$, and $\gamma$.

From \eqref{Dq_decay}, we have
\begin{equation*}
	q_n u_n - \tilde{l} + \rho u_n^2 
	\leq \frac{1}{4}(l+\theta -1) <0 \quad \text{on } S_{1/2},
\end{equation*}
which implies that
\begin{align*}
	w= \max\{q_n u_n - \tilde{l} + \rho u_n^2, 0\} =0 \quad \text{on } S_{1/2}.
\end{align*}
Then the function
\begin{equation*}
\tilde{w}(x,t) = 
	\begin{cases}
		w(x,t) & \text{if } (x,t) \in Q_{1/2}^+\\
		0 & \text{if } (x,t) \in Q_{1/2} \cap \{x_n\leq0\}
	\end{cases}
\end{equation*}
satisfies
\begin{equation*}
	\tilde{w}_t \leq \mathcal{M}^+_{\lambda^*,\Lambda^*}(D^2 \tilde{w}) + C |D\tilde{w}|^2 \quad \text{in } Q_{1/2} \text{ in the viscosity sense.}
\end{equation*}

Consider the nonnegative function
\begin{equation*}
	h= \frac{\lambda^*}{C_2}(1-e^{\frac{C_2}{\lambda^*}(\tilde{w}-1+\tilde{l}-\rho)}).
\end{equation*}
Then $h$ satisfies
\begin{align*}
	h_t & \ge \mathcal{M}^-_{\lambda^*,\Lambda^*}\left(D^2h+ \frac{C_2}{\lambda^*} e^{-\frac{C_2}{\lambda^*}(\tilde{w}-1+\tilde{l}-\rho)} Dh\otimes Dh\right) - C_2 e^{-\frac{C_2}{\lambda^*}(\tilde{w}-1+\tilde{l}-\rho)} |Dh|^2  \\
	&\ge  \mathcal{M}^-_{\lambda^*,\Lambda^*}(D^2 h) \quad \text{in } Q_{1/2} \text{ in the viscosity sense.}
\end{align*}
Since 
\begin{equation*}
	Q_{1/2}\cap \{x_n\leq0\}\subset \{(x,t) \in Q_{1/2} : \tilde{w}\leq 0 \}\subset \{(x,t) \in Q_{1/2} : h\ge \lambda^*(1-e^{\frac{C_2}{\lambda^*}(\tilde{l}-1-\rho)})/C_2 \} ,
\end{equation*}
we have 
\begin{equation*}
	|\{(x,t) \in Q_{1/2} : h \ge \lambda^*(1-e^{\frac{C_2}{\lambda^*}(\tilde{l}-1-\rho)})/C_2 \} | > \tilde{\mu} |Q_{1/2}| \quad \text{for all } \tilde{\mu} \in (0,1/2).
\end{equation*}
By \Cref{lem:IO}, there exist constants $r>0$ and $c>0$ such that 
\begin{equation*}
	1 - \tilde{l} + \rho -\tilde{w} \ge h\ge c \quad \text{in } Q_r,
\end{equation*}
which implies that
\begin{equation*}
	D_q u + \rho |D_qu|^2\leq D_q u + \rho |Du|^2 \leq 1 +\rho - c \quad \text{in } Q_r.
\end{equation*}
Therefore, we conclude that
\begin{equation*}
	D_q u \leq \frac{-1+\sqrt{1+4(1+\rho-c)}}{2\rho} < 1\quad \text{in } Q_r.
\end{equation*}
By setting 
\begin{equation*}
	\delta=1- \frac{-1+\sqrt{1+4(1+\rho-c)}}{2\rho} \quad \text{and} \quad
	\tau =
	\begin{cases}
		r & \text{if } \gamma <0 \\
		r(1-\delta)^{\gamma/2} & \text{if } \gamma \ge 0	,
	\end{cases}
\end{equation*}
we obtain the desired conclusion.
\end{proof}
The oscillation of $Du$ continues to improve under iteration and scaling, which ultimately guarantees the H\"older continuity of $Du$ at the origin.
\begin{lemma} \label{cor:Du<1-d}
Under the assumptions of \Cref{lem44-LLY24}, let $u$ be a viscosity solution to \eqref{eq:model2}. Then for any $l \in (3/4,1)$ and $\mu \in (0,1)$, there exist constants $\delta,\tau\in (0, 1/4)$ depending only on $n$, $\lambda$, $\Lambda$, $\gamma$, $l$, and $\mu$ such that for any nonnegative integer
\begin{equation*}
	k \leq \min\left\{ \frac{\log(2\varepsilon)}{\log (1-\delta)}, \frac{\log(2|p|)}{\log(1-\delta)} \right\},
\end{equation*}
and for any $e \in \{e_n, -e_n\}$, if
\begin{equation*}
	|\{(x,t) \in Q_{\tau^i}^{(1-\delta)^i+}: D_e u (x,t) \leq l (1-\delta)^i\}| > \mu |Q_{\tau^i}^{(1-\delta)^i+}| \quad \text{for all } i=0,1,\cdots, k,
\end{equation*}
then
\begin{equation*} 
	|D u| < (1-\delta)^{i+1} \quad \text{in } Q_{\tau^{i+1}}^{(1-\delta)^{i+1}+} \quad \text{for all } i=0,1,\cdots, k.
\end{equation*}
\end{lemma}

\begin{proof}
The proof proceeds by induction. The case $i = 0$ follows from \Cref{lem:Du<1-d}. Now assume that this corollary holds for some $i$. Applying  \Cref{lem:Du<1-d} to the function 
\begin{equation*}
	v(x,t)= \frac{1}{\tau^i(1-\delta)^i} u(\tau^ix, \tau^{2i}(1-\delta)^{-i \gamma}t)
\end{equation*}
and scaling back to $u$ yields this corollary holds for $i + 1$.
\end{proof}
If $|Du(0,0)| \neq 0$, then the iteration in \Cref{cor:Du<1-d} must terminate after finitely many steps. Intuitively, this implies that $Du$ is nearly aligned with the direction $e \in \{e_n,-e_n\}$ in a large portion of $Q_{\tau^k}^{(1-\delta)^k+}$. Consequently, the solution $u$ can be well approximated by a linear function $L(x)=e\cdot x$.
\begin{lemma} \label{lem:Du=e}
Assume that $F$ satisfies \textnormal{\ref{F1}} and  
\begin{equation*}
	|p| \leq 1 \quad \text{and} \quad
	\|Du\|_{L^\infty(Q_1^+)} \leq 1.
\end{equation*}
Let $u$ be a viscosity solution to  \eqref{eq:model} with $\nu=1$. Then for any $\eta \in (0,1)$, there exist two sufficiently small constants $\delta_1>0$ and $\delta_2 >0$ depending only on $n$, $\lambda$, $\Lambda$, $\gamma$, $\eta$, and $\|u\|_{L^\infty(Q_1^+)}$ such that 
if 
\begin{equation*}
	|\{(x,t) \in Q_1^+ : |Du(x,t)-e| > \delta_1\} | \leq \delta_2 \quad \text{for all } e \in \{e_n, -e_n\},
\end{equation*}
then 
\begin{equation*}
	|u(x,t) - e\cdot x| \leq \eta \quad \text{for all }(x,t) \in Q_{1/2}^+.
\end{equation*}
\end{lemma}

\begin{proof}
Since the arguments are analogous, it suffices to consider the case $e=e_n$.  Let 
\begin{equation*}
	A(t)=\{x \in B_1^+ : |Du(x,t)-e_n| > \delta_1\}  \quad  \text{for } t \in (-1,0].
\end{equation*}
Then, by Fubini’s theorem, we obtain
\begin{equation*}
	\int_{-1}^0 |A(t)| \, dt =|\{(x,t) \in Q_1^+ : |Du(x,t)-e| > \delta_1\} | \leq \delta_2.
\end{equation*}
Let $E=\{t \in (-1,0] : |A(t)| \ge \sqrt{\delta_2}\}$. \\

\noindent(\textbf{Case 1}: $t \in (-1/4,0] \setminus E$.) In this case, Morrey's inequality yields 
\begin{equation*}
	\|u(\cdot,t) -x_n \|_{L^\infty(B_{1/2}^+)} \leq C_1 \|Du(\cdot,t) - e_n \|_{L^{2n}(B_{1/2}^+)},
\end{equation*}
where $C_1>0$ is a constant depending only on $n$. Moreover, since
\begin{align*}
     \|Du(\cdot,t)-e_n\|_{L^{2n}(B_{1/2}^+)} 
    &= \left( \int_{B_{1/2}^+\setminus A(t)} |Du(x,t)-e_n|^{2n}\,dx + \int_{B_{1/2}^+\cap A(t)} |Du(x,t)-e_n|^{2n}\,dx  \right)^{\frac{1}{2n}}\\
     & \leq \left( \int_{B_{1/2}^+\setminus A(t)} \delta_1^{2n}\,dx + \int_{B_{1/2}^+\cap A(t)} 2^{2n}\,dx \right)^{\frac{1}{2n}} \\
     & \leq \left( \delta_1^{2n}|B_{1/2}^+| + 2^{2n} \sqrt{\delta_2} \right)^{\frac{1}{2n}}  
     \leq  C_2(\delta_1+\delta_2^{\frac{1}{4n}}),
 \end{align*}
we obtain
\begin{equation}\label{out_E}
	|u(x,t) -x_n | \leq C_2(\delta_1+\delta_2^{\frac{1}{4n}}) \quad \text{for all } x \in B_{1/2}^+.
\end{equation}

\noindent(\textbf{Case 2}: $t \in (-1/4,0] \cap E$.) Since 
\begin{equation*}
	|E| \leq \frac{1}{\sqrt{\delta_2}} \int_E f(t) \,dt \leq \frac{1}{\sqrt{\delta_2}} \int_{-1}^0 f(t) \,dt =\sqrt{\delta_2},
\end{equation*}
for each $t \in E$, there exists $s \in (-1/4,0] \setminus E$ such that $|t-s| \leq \sqrt{\delta_2}$. Then, for each $x \in B_{1/2}^+$,  combining \Cref{lem:G_hol_t} with \eqref{out_E} yields
\begin{align*}
	|u(x,t) -x_n| & \leq |u(x,t) - u(x,s)| + |u(x,s) -x_n|  \\
	&\leq \widetilde{C}|t-s|^{1/2} + C_2 (\delta_1 + \delta_2^{\frac{1}{4n}}) \\
	&\leq \widetilde{C}\delta_2^{1/4} + C_2 (\delta_1 + \delta_2^{\frac{1}{4n}}),
\end{align*}
where $\widetilde{C}>0$ is a constant depending only on $n$, $\lambda$, $\Lambda$, $\gamma$, and $\|u\|_{L^\infty(Q_1^+)}$.

Finally, choosing $\delta_1$ and $\delta_2$ sufficiently small yields the desired conclusion.
\end{proof}

We now state a uniform $C_\gamma^{1,\alpha}$-estimate for the regularized problem \eqref{eq:model2}, which is independent of $\varepsilon$.
\begin{lemma} \label{lem:C1a_small_p}
Assume that $F$ satisfies \textnormal{\ref{F1}}, \textnormal{\ref{F2}},
\begin{equation*}
	|p| \leq 1, \quad 
	p_n=0, \quad\text{and} \quad
	\|Du\|_{L^\infty(Q_1^+)} \leq 1.
\end{equation*}
Let $u$ be a viscosity solution to \eqref{eq:model2}. Then $u \in C_\gamma^{1,\bar{\alpha}}(0,0)$ for some $\bar{\alpha}\in(0,1)$ with $\bar\alpha<\frac{1}{1+\gamma}$, that is, there exists a constant $a\in \mathbb{R}$ such that
\begin{equation*}
	|u(x,t)-ax_n| \leq C(|x|^{1+\bar{\alpha}}+|t|^{\frac{1+\bar{\alpha}}{2-\bar{\alpha} \gamma}}) \quad \text{for all }(x,t) \in Q_{1/2}^+,
\end{equation*}
where $C>0$ is a constant depending only on $n$, $\lambda$, $\Lambda$, $\gamma$, and $\|u\|_{L^\infty(Q_1^+)}$.
\end{lemma}
\begin{proof}
Let $\eta \in (0,1)$ be the constant given in \Cref{lem:small_bdry}. For this choice of $\eta$, let $\delta_1>0$ and $\delta_2>0$ be the constants given in \Cref{lem:Du=e}. Since $\delta_1$ and $\delta_2$ are sufficiently small, we may assume that
\begin{equation*}
	l=1 - \frac{1}{2} \delta_1^2 \in (3/4,1)\quad \text{and} \quad
	\mu= \frac{\delta_2}{|Q_1^+|} \in (0,1).
\end{equation*}

For $(x,t) \in \{(x,t) \in Q_1^+ : |Du(x,t)-e| > \delta_1\}$, we have
\begin{equation*}
	2-2 D_e u(x,t) \ge |Du(x,t)|^2 -2 D_e u(x,t) + 1 = |Du(x,t)-e|^2 > \delta_1^2,
\end{equation*}
which implies that
\begin{equation*}
	\{(x,t) \in Q_1^+ : |Du(x,t)-e| > \delta_1\} \subset \{(x,t) \in Q_1^+: D_e u (x,t) \leq l \}.
\end{equation*}
Hence, if 
\begin{equation*}
	|\{(x,t) \in Q_1^+: D_e u (x,t) \leq l \}| \leq \mu |Q_1^+|  \quad \text{for all } e \in \{e_n, -e_n\},
\end{equation*}
then
\begin{equation} \label{rmk_in_proof}
	\{(x,t)\in Q_1^+:|Du(x,t)-e| > \delta_1 \} \leq \delta_2.
\end{equation}

Let $\delta, \tau \in (0,1/4)$ be the constants given in \Cref{cor:Du<1-d}. Here, $\tau$ can be chosen sufficiently small so that it satisfies
\begin{equation*}
	\tau < (1-\delta)^{1+\gamma} \quad \text{for all } \gamma >-1.
\end{equation*}
Let $\bar\alpha= \log_\tau(1-\delta) \in (0,\frac{1}{1+\gamma})$. By \Cref{cor:Du<1-d},  for any nonnegative integer
\begin{equation} \label{k_upper}
	k \leq \min\left\{ \frac{\log(2\varepsilon)}{\bar\alpha \log \tau}, \frac{\log(2|p|)}{\bar\alpha \log \tau} \right\},
\end{equation}
and for any $e \in \{e_n, -e_n\}$, if
\begin{equation}\label{i-th:Deu<1-d}
	|\{(x,t) \in Q_{\tau^i}^{\tau^{i\bar\alpha}+}: D_e u (x,t) \leq l \tau^{i\bar\alpha}\}| > \mu |Q_{\tau^i}^{\tau^{i\bar\alpha}+}| \quad \text{for all } i=0,1,\cdots, k,
\end{equation}
then
\begin{equation} \label{Du<tau}
	|D u| <\tau^{(i+1)\bar\alpha} \quad \text{in } Q_{\tau^{i+1}}^{\tau^{(i+1)\bar\alpha}+} \quad \text{for all } i=0,1,\cdots, k.
\end{equation}
Let $k_* \ge 1$ be the smallest integer such that both \eqref{k_upper} and \eqref{i-th:Deu<1-d} hold for all $k \leq k_*-1$, but at least one of them fails for $k=k_*$. From \eqref{Du<tau} for $k=k_*-1$, we obtain
\begin{equation} \label{est:Du_out}
	|Du(x,t)-q| \leq C(|x|^{\bar\alpha}+|t|^{\frac{\bar\alpha}{2-\bar\alpha\gamma}}) \quad \text{for all } (x,t) \in Q_1^+ \setminus Q_{\tau^{k_*+1}}^{\tau^{(k_*+1)\bar\alpha}+}
\end{equation}
for all $q \in \mathbb{R}^n$ with $|q| \leq \tau^{k_*\bar\alpha}$.

Let us now consider the cases where either \eqref{k_upper} or \eqref{i-th:Deu<1-d} fails. We shall show that the desired conclusion holds in each case.\\

\noindent(\textbf{Case 1}: \eqref{k_upper} fails for $k=k_*$.)  First, we consider the case $\varepsilon \ge |p|$.
Then the function
\begin{equation*}
	\tilde{u}(x,t) = \frac{1}{\tau^{k_*(1+\bar\alpha)}} u(\tau^{k_*}x, \tau^{k_*(2-\bar\alpha\gamma)}t)
\end{equation*}
satisfies $\|D\tilde{u}\|_{L^\infty(Q_1^+)}\leq 1$ and
\begin{equation*} 
\left\{\begin{aligned}
	\tilde{u}_t &= (|D\tilde{u} + \tilde{p}|^2 + \tilde{\varepsilon}^2)^{\gamma/2} \widetilde{F}(D^2 \tilde{u})  && \text{in } Q_1^+ \\
	\tilde{u}&=0 && \text{on } S_1,
\end{aligned}\right.
\end{equation*}
where $\tilde{p}=\tau^{- k_*\bar\alpha}p$, $\tilde{\varepsilon}= \tau^{- k_*\bar\alpha} \varepsilon$ and $\widetilde{F}(M)=\tau^{k_*(1-\bar\alpha)} F^\varepsilon(\tau^{k_*(\bar\alpha-1)}M)$.

Since $\varepsilon \ge |p|$ and  \eqref{k_upper} fails for $k=k_*$, it follows that $k_*-1 \leq \frac{\log(2\varepsilon)}{\bar\alpha \log \tau} < k_*$, and hence we obatin
\begin{equation*}
	|\tilde{p}| \leq \tau^{- k_*\bar\alpha} \varepsilon <1 \quad \text{and} \quad
	\frac{1}{2} < \tilde{\varepsilon} \leq \frac{1}{2\tau^{\bar\alpha}} <1.
\end{equation*}
Therefore, there exist constants $\tilde\lambda$ and $\tilde{\Lambda}$ depending only on $\lambda$, $\Lambda$, and $\gamma$ such that $\tilde{u}$ satisfies
\begin{equation*}
	\mathcal{M}_{\tilde\lambda,\tilde\Lambda}^-(D^2 \tilde{u}) \leq \tilde{u}_t \leq \mathcal{M}_{\tilde\lambda,\tilde\Lambda}^+(D^2 \tilde{u}) \quad  \text{in } Q_1^+ \text{ in the viscosity sense.}
\end{equation*}
By \Cref{thm:LZ22-28}, we have $D\tilde{u} \in C^{\alpha}(0,0)$ for some $\alpha \in (0,1)$, and hence, by \Cref{thm:LZ22-116}, it follows that $\tilde{u} \in C^{2,\alpha}(0,0)$, that is, there exists a constant $\tilde{a} \in [-1,1]$ such that
\begin{equation*}
	|D\tilde{u}(x,t)- \tilde{a}e_n| \leq C(|x|+|t|^{1/2}) \leq C(|x|^{\bar\alpha}+|t|^{\frac{\bar\alpha}{2-\bar\alpha\gamma}}) \quad \text{for all } (x,t) \in Q_\tau^{\tau^{\bar\alpha}+}.
\end{equation*}
Scaling back to $u$, we obtain
\begin{equation*}
	|Du(x,t) - \tau^{k_* \bar\alpha}\tilde{a}e_n|  \leq C(|x|^{\bar\alpha}+|t|^{\frac{\bar\alpha}{2-\bar\alpha\gamma}}) \quad \text{for all } (x,t) \in Q_{\tau^{k_*+1}}^{\tau^{(k_*+1)\bar\alpha}+}.
\end{equation*}
Combining this with \eqref{est:Du_out}, we obtain 
\begin{equation*}
	|Du(x,t)-\tau^{k_* \bar\alpha}\tilde{a}e_n| \leq C(|x|^{\bar\alpha}+|t|^{\frac{\bar\alpha}{2-\bar\alpha\gamma}}) \quad \text{for all } (x,t) \in Q_1^+.
\end{equation*}

Next, we consider the case $\varepsilon < |p|$. Then $k_*-1 \leq \frac{\log(2|p|)}{\bar\alpha \log \tau} < k_*$ and hence we obatin
\begin{equation*}
	\frac{1}{2}< |\tilde{p}| = \tau^{- k_*\bar\alpha} |p| \leq \frac{1}{2\tau^{\bar\alpha}} <1 \quad \text{and} \quad
	0<\tilde{\varepsilon} < \tau^{- k_*\bar\alpha} |p| <1 .
\end{equation*}
Hence, by \Cref{lem:c11_large_p}, we have $\tilde{u} \in C^{1,\alpha}(0,0)$ for all $\alpha \in (0,1)$, scaling back to $u$, we obtain
\begin{equation*}
	|u(x,t)-\check{a}x_n| \leq C(|x|^{1+\bar\alpha}+|t|^{\frac{1+\bar\alpha}{2-\bar\alpha\gamma}}) \quad \text{for all } (x,t) \in Q_{1/2}^+,
\end{equation*}
for some $\check{a} \in [-1,1]$. \\

\noindent(\textbf{Case 2}: For $k=k_*$, \eqref{i-th:Deu<1-d} fails while \eqref{k_upper} holds.) Since the arguments are analogous, it suffices to consider the case $e=e_n$, that is, 
\begin{equation*} 
	|\{(x,t) \in Q_{\tau^{k_*}}^{\tau^{k_*\bar\alpha}+}: u_n(x,t) \leq l \tau^{k_*\bar\alpha}\}| \leq \mu |Q_{\tau^{k_*}}^{\tau^{k_*\bar\alpha}+}|,
\end{equation*}
which is equivalent to
\begin{equation*} 
	|\{(x,t) \in Q_1^+: \tilde{u}_n(x,t) \leq l| \leq \mu |Q_1^+|.
\end{equation*}
Moreover, since \eqref{k_upper} holds for $k=k_*$, we have
\begin{equation*}
	|\tilde{p}| = \tau^{- k_*\bar\alpha}|p| \leq \frac{1}{2} \quad \text{and} \quad
	0<\tilde{\varepsilon}= \tau^{- k_*\bar\alpha} \varepsilon \leq \frac{1}{2}.
\end{equation*}
It follows from \eqref{rmk_in_proof} that
\begin{equation*}
	\{(x,t)\in Q_1^+:|D\tilde{u}(x,t)-e_n| > \delta_1 \} \leq \delta_2,
\end{equation*}
and hence, by \Cref{lem:Du=e}, we obtain
\begin{equation*}
	|\tilde{u}(x,t) -x_n| \leq \eta \quad \text{for all } (x,t) \in Q_{1/2}^+.
\end{equation*}
Then $v=\tilde{u}-x_n$ is a solution to 
\begin{equation*} 
\left\{\begin{aligned}
	v_t &= (|Dv + \hat{p}|^2 + \tilde{\varepsilon}^2)^{\gamma/2} \widetilde{F}(D^2 \tilde{u})  && \text{in } Q_1^+ \\
	v&=0 && \text{on } S_1,
\end{aligned}\right.
\end{equation*}
where $\hat{p} = \tilde{p} + e_n$. Since $\tilde{p}_n=0$ and $|\tilde{p}|\leq 1/2$, we have $1\leq |\hat{p} | \leq 3/2$. Therefore, by \Cref{lem:small_bdry}, we have $v \in C^{1,\alpha}(0,0)$ for all $\alpha \in (0,1)$, scaling back to $u$, we obtain
\begin{equation*}
	|u(x,t)-\hat{a}x_n| \leq C(|x|^{1+\bar\alpha}+|t|^{\frac{1+\bar\alpha}{2-\bar\alpha\gamma}}) \quad \text{for all } (x,t) \in Q_{1/2}^+,
\end{equation*}
for some $\hat{a} \in [-1,1]$. 
\end{proof}
From now on, we denote by $\bar\alpha$ the constant appearing in \Cref{lem:C1a_small_p}.
\begin{proof} [Proof of \Cref{thm:c1a_flat}]
Without loss of generality, we may assume that $u \in C(\overline{Q_1^+})$. Let $\Omega \subset \mathbb{R}^{n+1}$ be a smooth domain satisfying $Q_{3/4}^+ \subset \Omega \subset Q_1^+$. By \Cref{thm:LLY24-12}, \Cref{thm:LLY24-412}, and \Cref{rmk:eq_tr}, there exists a unique solution $u^\varepsilon \in C(\overline{\Omega})\cap C^\infty(\Omega)$ to
\begin{equation*} 
\left\{\begin{aligned}
	\partial_t u^\varepsilon &= (|Du^\varepsilon + p|^2 + \varepsilon^2)^{\gamma/2} F^\varepsilon(D^2 u^\varepsilon)  && \text{in } \Omega \\
	u^\varepsilon&=u && \text{on } \partial_p \Omega.
\end{aligned}\right.
\end{equation*}

By \Cref{lem:G_hol_t}, there exists a constant $C_1>0$ depending only on  $n$, $\lambda$, $\Lambda$, $\gamma$, and $\|u^\varepsilon\|_{L^\infty(Q_{3/4}^+)}$ such that $Du^\varepsilon \in L^\infty(Q_{3/4}^+) $. Setting $\rho=1/(1+\|Du^\varepsilon\|_{L^\infty(Q_{3/4}^+)} )$ in \Cref{rmk:normal}, we see that the assumptions of \Cref{lem:C1a_small_p} are satisfied.

Moreover, by the Arzelà--Ascoli theorem, there exists a function $\bar{u}\in C(\overline{\Omega})$ such that $u^\varepsilon \to \bar{u}$, and by \Cref{thm:stability2}, $\bar{u}$ is a solution to
\begin{equation*} 
\left\{\begin{aligned}
	\bar{u}_t &= |D\bar{u}+ p|^\gamma F(D^2 \bar{u})  && \text{in } \Omega  \\
	\bar{u}&=u&& \text{on } \partial_p \Omega .
\end{aligned}\right.
\end{equation*}
By \Cref{thm:comp_prin}, we conclude that $\bar{u}=u$. Finally, by \Cref{lem:C1a_small_p}, there exists a constant $a\in \mathbb{R}$ such that
\begin{equation*}
	|u^\varepsilon(x,t)-ax_n| \leq C(|x|^{1+\bar{\alpha}}+|t|^{\frac{1+\bar{\alpha}}{2-\bar{\alpha} \gamma}}) \quad \text{for all }(x,t) \in Q_{1/2}^+.
\end{equation*}
Letting $\varepsilon \to 0$, we arrive at the desired conclusion.
\end{proof}
From now on, we denote by $C_\star$ the constant appearing in \Cref{thm:c1a_flat}.
%
%
\section{Boundary $C^{1,\alpha}_\gamma$-estimates on general boundaries} \label{sec:gen_bdry}
In this section, we establish the boundary $C^{1,\alpha}_\gamma$-regularity of viscosity solutions to \eqref{eq:main} leading to the proof of our main theorems.

The key idea behind the following lemma is to argue by contradiction. Assuming that the desired estimate fails, we construct a sequence of viscosity solutions with normalized data that violate the estimate. By employing a compactness argument, we extract a locally uniformly convergent subsequence whose limit solves a model problem \eqref{model}. However, the regularity of this limiting solution has already been established in the model setting by \Cref{thm:c1a_flat}. This leads to a contradiction.
\begin{lemma} \label{lem:gen_1st}
Assume that $F$ satisfies \textnormal{\ref{F1}} and \textnormal{\ref{F2}}. Then, for any $\alpha\in(0,\bar\alpha)$ and $\eta \in (0,1)$, there exists $\delta \in (0,\eta]$ depending only on $n$, $\lambda$, $\Lambda$, $\gamma$, $\alpha$, and $\eta$ such that if $u$ is a viscosity solution to 
\begin{equation*}  
\left\{\begin{aligned}
	u_t &= |Du + p|^\gamma F(D^2 u)  + f&& \text{in } \Omega\cap Q_1 \\
	u&=g && \text{on } \partial_p \Omega\cap Q_1
\end{aligned}\right.
\end{equation*}
with
\begin{equation*}
	 |p| \leq 1, \quad 
	 p_n=0, \quad
	 \|u\|_{L^\infty(\Omega\cap Q_1)} \leq 2,
\end{equation*}
\begin{equation*}
	 \|f\|_{L^\infty(\Omega\cap Q_1)} \leq \delta, \quad
	 \|g\|_{L^\infty(\partial_p\Omega\cap Q_1)} \leq \delta, \quad\text{and}\quad
	 \osc_{Q_1} \partial_p\Omega \leq \delta,
\end{equation*}
then there exists a constant $a\in \mathbb{R}$ such that
\begin{equation*}
	\|u-ax_n\|_{L^\infty(\Omega\cap Q_\tau^{\tau^\alpha})} \leq \eta \tau^{1+\alpha} \quad\text{and}\quad
	|a| \leq C_\star,
\end{equation*}
where  $\tau \in (0,1)$ is a constant depending only on $n$, $\lambda$, $\Lambda$, $\gamma$, $\alpha$, and $\eta$.
\end{lemma}
\begin{proof}
Let us assume that the conclusion is false, that is, there exist constants $\alpha \in (0,\bar\alpha)$, $\eta\in(0,1)$ and sequences $F_k$, $u_k$, $f_k$, $g_k$, $\Omega_k$, and $p_k$ such that $u_k$ is a viscosity solution to 
\begin{equation*}  
\left\{\begin{aligned}
	\partial_t u_k &= |Du_k + p_k|^\gamma F_k(D^2 u_k)  + f_k&& \text{in } \Omega_k\cap Q_1 \\
	u_k&=g_k && \text{on } \partial_p \Omega_k\cap Q_1,
\end{aligned}\right.
\end{equation*}
where $F_k$ satisfies \textnormal{\ref{F1}}, \textnormal{\ref{F2}}, and
\begin{equation*}
	 |p_k| \leq 1, \quad 
	 p_k\cdot e_n=0, \quad
	 \|u_k\|_{L^\infty(\Omega_k\cap Q_1)} \leq 2,
\end{equation*}
\begin{equation*}
	 \|f_k\|_{L^\infty(\Omega_k\cap Q_1)} \leq \frac{1}{k}, \quad
	 \|g_k\|_{L^\infty(\partial_p\Omega_k\cap Q_1)} \leq  \frac{1}{k}, \quad\text{and}\quad
	 \osc_{Q_1} \partial_p\Omega_k \leq  \frac{1}{k}.
\end{equation*}
Moreover, the following estimate holds:
\begin{equation} \label{contra1}
	\|u_k-ax_n\| _{L^\infty(\Omega_k\cap Q_\tau^{\tau^\alpha})}> \eta \tau^{1+\alpha} \quad \text{for all } a \in \mathbb{R} \text{ with } |a| \leq C_\star,
\end{equation}
where $\tau \in (0,1)$ will be determined later.

Then, as in the proof of \Cref{lem:small_1st}, there exist a vector $\tilde{p}$, a function $\tilde{u}$ and an operator $\widetilde{F}$ such that  
\begin{equation*}
	p_k \to p, \quad
	u_k \to \tilde{u} \text{ locally uniformly in } \Omega_k \quad \text{and} \quad
	F_k \to \widetilde{F} \text{ locally uniformly in } \mathcal{S}^n. 
\end{equation*}

By \Cref{lem:bdry_lip1}, we obtain
\begin{equation*}
	|u_k(x,t)| \leq C\left(x_n + \frac{1}{k}\right) \quad \text{for all } (x,t) \in \Omega_k\cap Q_{3/4},
\end{equation*}
and letting $k\to\infty$, we deduce
\begin{equation*}
	|\tilde{u}(x,t)| \leq Cx_n \quad \text{for all } (x,t) \in Q_{3/4}^+.
\end{equation*}
Thus, $\tilde{u}$ is a viscosity solution to
\begin{equation*}  
\left\{\begin{aligned}
	\tilde{u}_t &= |D\tilde{u} + \tilde{p}|^\gamma \widetilde{F}(D^2 \tilde{u}) && \text{in } Q_{3/4}^+\\
	\tilde{u}&=0 && \text{on } S_{3/4},
\end{aligned}\right.
\end{equation*}
and by \Cref{thm:c1a_flat}, there exists a constant $\tilde{a}\in \mathbb{R}$ such that
\begin{equation}\label{est:from416}
	|\tilde{u}(x,t)-\tilde{a}x_n| \leq C_\star(|x|^{1+\bar{\alpha}}+|t|^{\frac{1+\bar{\alpha}}{2-\bar{\alpha} \gamma}}) \quad \text{for all }(x,t) \in Q_{1/2}^+
\end{equation}
and $|\tilde{a}| \leq C_\star$.

If we now choose $\tau\in(0,1)$ sufficiently small so that
\begin{equation*}
	C_\star(\tau^{\bar\alpha-\alpha}+\tau^{\frac{2-\alpha\gamma}{2-\bar\alpha\gamma}(1+\bar\alpha)-1-\alpha}) \leq \frac{1}{2}\eta,
\end{equation*}
then by \eqref{est:from416} we obtain
\begin{equation*}
	\|\tilde{u}-\tilde{a}x_n\|_{L^\infty(Q_\tau^{\tau^\alpha+})} \leq C_\star(\tau^{\bar\alpha-\alpha} + \tau^{\frac{2-\alpha\gamma}{2-\bar\alpha\gamma}(1+\bar\alpha)-1-\alpha}) \tau^{1+\alpha} \leq \frac{1}{2}\eta\tau^{1+\alpha}.
\end{equation*}
However, letting $k\to\infty$, \eqref{contra1} yields
\begin{equation*}  
	\|\tilde{u}-\tilde{a}x_n\| _{L^\infty(Q_\tau^{\tau^\alpha+})}\ge \eta \tau^{1+\alpha},
\end{equation*}
which leads to a contradiction.
\end{proof}
\begin{lemma}  \label{lem:gen_next}
Assume that $F$ satisfies \textnormal{\ref{F1}} and \textnormal{\ref{F2}}. Let $\alpha$, $\eta$, $\delta$, and $\tau$ be the constants given in \Cref{lem:gen_1st} and let $r \in (0,1]$. If $u$ is a viscosity solution to 
\begin{equation*}  
\left\{\begin{aligned}
	u_t &= |Du + p|^\gamma F(D^2 u)  + f&& \text{in } \Omega\cap Q_r^{r^\alpha} \\
	u&=g && \text{on } \partial_p \Omega\cap Q_r^{r^\alpha}
\end{aligned}\right.
\end{equation*}
with
\begin{equation*}
	 |p| \leq r^\alpha, \quad 
	 p_n=0, \quad
	 \|u\|_{L^\infty(\Omega\cap Q_r^{r^\alpha})} \leq 2 r^{1+\alpha},
\end{equation*}
\begin{equation*}
	 \|f\|_{L^\infty( \Omega\cap Q_r^{r^\alpha})} \leq \delta, \quad
	 \|g\|_{L^\infty(\partial_p\Omega\cap Q_r^{r^\alpha})} \leq \delta r^{1+\alpha}, \quad\text{and}\quad
	 \osc_{Q_r^{r^\alpha}} \partial_p\Omega \leq \delta r^{1+\alpha},
\end{equation*}
then there exists a constant $a\in \mathbb{R}$ such that
\begin{equation*}
	\|u-ax_n\|_{L^\infty(\Omega\cap Q_{r\tau}^{(r\tau)^\alpha})} \leq \eta (r\tau)^{1+\alpha} \quad\text{and}\quad
	|a| \leq C_\star r^\alpha.
\end{equation*}
\end{lemma}
\begin{proof}
Let us consider the function  
\begin{equation*}
 	\tilde{u}(x,t)=\frac{1}{r^{1+\alpha}} u(rx, r^{2-\alpha\gamma}t).
\end{equation*}
Then $\tilde{u}$ is a viscosity solution to 
\begin{equation*} 
\left\{\begin{aligned}
	\tilde{u}_t &= |D\tilde{u} + \tilde{p}|^\gamma \widetilde{F}(D^2 \tilde{u}) + \tilde{f} && \text{in } \widetilde{\Omega}\cap Q_1 \\
	\tilde{u}&=\tilde{g} && \text{on } \partial_p \widetilde{\Omega}\cap Q_1,
\end{aligned}\right.
\end{equation*}
where 
\begin{equation*}
	\tilde{p}=r^{-\alpha}p, \quad
	\tilde{f}(x,t)=r^{1-\alpha\gamma -\alpha}f(rx, r^{2-\alpha\gamma}t), \quad
	\tilde{g}(x,t)=r^{-1-\alpha}g(rx, r^{2-\alpha\gamma}t), 
\end{equation*}
\begin{equation*}
	\widetilde{F}(M)=r^{1-\alpha} F(r^{\alpha-1}M), \quad\text{and}\quad
	\widetilde{\Omega}=\{(x,t)\mid(rx, r^{2-\alpha\gamma}t) \in \Omega \}.
\end{equation*}
Since $\alpha < \frac{1}{1+\gamma}$, we have
\begin{equation*}
	 |\tilde{p}| \leq 1, \quad 
	 \tilde{p}_n=0, \quad
	 \|\tilde{u}\|_{L^\infty(\widetilde{\Omega}\cap Q_1)} \leq 2,
\end{equation*}
\begin{equation*}
	 \|f\|_{L^\infty(\widetilde{\Omega}\cap Q_1)} \leq \delta, \quad
	 \|g\|_{L^\infty(\partial_p\widetilde{\Omega}\cap Q_1)} \leq \delta, \quad\text{and}\quad
	 \osc_{Q_1} \partial_p\widetilde{\Omega} \leq \delta.
\end{equation*}
Therefore, by \Cref{lem:gen_1st}, there exists constant $\tilde{a}\in \mathbb{R}$ such that
\begin{equation*}
	\|\tilde{u}-\tilde{a}x_n\|_{L^\infty(\widetilde{\Omega}\cap Q_\tau^{\tau^\alpha})} \leq \eta \tau^{1+\alpha} \quad\text{and}\quad
	|\tilde{a}| \leq C_\star.
\end{equation*}
Scaling back to $u$, we obtain the desired conclusion.
\end{proof}
The following lemma is obtained by iteratively applying \Cref{lem:gen_next}. The argument naturally divides into two cases depending on whether $p$ is nonzero or not. If $p\neq 0$, then the iteration terminates after finitely many steps, and \Cref{lem:small_bdry} yields the $C^{1,\alpha}$-estimate. On the other hand, if $p=0$, the iteration may continue infinitely, which leads to the $C_\gamma^{1,\alpha}$-regularity.
\begin{lemma}  \label{lem:degenerate}
Assume that $F$ satisfies \textnormal{\ref{F1}}, \textnormal{\ref{F2}}, and $\gamma \ge 0$. For any given $\alpha \in (0,\bar\alpha)$, let $\eta$ be the constant from \Cref{lem:small_bdry}, and let $\delta$ be the constant from \Cref{lem:gen_1st} associated with this $\alpha$ and $\eta$. Let $u$ be a viscosity solution to 
\begin{equation*}  
\left\{\begin{aligned}
	u_t &= |Du + p|^\gamma F(D^2 u)  + f&& \text{in } \Omega\cap Q_1\\
	u&=g && \text{on } \partial_p \Omega\cap Q_1
\end{aligned}\right.
\end{equation*}
with
\begin{equation*}
	 |p| \leq 1, \quad 
	 p_n=0, \quad
	 \|u\|_{L^\infty(\Omega\cap Q_1)} \leq 1, \quad
	\|f\|_{L^\infty( \Omega\cap Q_1)} \leq \delta,
\end{equation*}
\begin{equation} \label{cond:g_bdry}
	 \|g\|_{L^\infty(\partial_p\Omega\cap Q_r^{r^\alpha})} \leq \frac{1}{2}\delta r^{1+\alpha}, \quad\text{and}\quad
	 \osc_{Q_r^{r^\alpha}} \partial_p\Omega \leq \frac{1}{2C_\star}\delta r^{1+\alpha} \quad \text{for all } r \in (0,1].
\end{equation}
Then $u \in C^{1,\alpha}(0,0)$, that is, there exists a constant $a\in \mathbb{R}$ such that
\begin{equation*}
	|u(x,t)-ax_n| \leq C(|x|^{1+\alpha}+|t|^{\frac{1+\alpha}{2}}) \quad \text{for all } (x,t) \in \Omega \cap Q_1
\end{equation*}
and $|a| \leq C$, where $C>0$ is a constant depending only on $n$, $\lambda$, $\Lambda$, $\gamma$, and $\alpha$.
\end{lemma}
\begin{proof}
Let $\tau$ be the constant from  \Cref{lem:gen_1st}. Let us first consider the case when $p\neq 0$. We will show that there exist a integer $K \in \mathbb{N}$ and a sequence $\{a_k\}_{k=0}^{K}$ such that 
\begin{equation}\label{est:u-ax}
	\|u-a_kx_n\|_{L^\infty(\Omega \cap Q_{\tau^k}^{\tau^{k\alpha}})} \leq \tau^{k(1+\alpha)},\quad
	|a_k| \leq \tau^{k\alpha}, \quad \text{and} \quad
	|p| \leq \tau^{k\alpha}
\end{equation}
hold for all $k=0,1,\cdots, K-1$, but only the first estimate in \eqref{est:u-ax} holds for $k=K$.

First, if we set $a_0=0$, it is clear that \eqref{est:u-ax} holds for $k=0$. Suppose that \eqref{est:u-ax} hold for $k=  i-1$ and let $r=\tau^{i-1}$. Since
\begin{equation*}
	\|u \|_{L^\infty(\Omega \cap Q_r^{r^\alpha})} \leq \|u -a_{i-1}x_n\|_{L^\infty(\Omega \cap Q_r^{r^\alpha})} + |a_{i-1}| r \leq 2r^{1+\alpha}\quad \text{and}\quad 
	|p| \leq r^{\alpha},
\end{equation*}
it follows from \Cref{lem:gen_next} that there exists a constant $a_i\in \mathbb{R}$ such that
\begin{equation} \label{est:u-ax_a}
	\|u-a_ix_n\|_{L^\infty(\Omega\cap Q_{\tau^i}^{\tau^{i\alpha}})} \leq \eta \tau^{i(1+\alpha)} \quad\text{and}\quad
	|a_i| \leq C_\star \tau^{(i-1)\alpha}.
\end{equation}
If 
\begin{equation} \label{est:ak-p}
	|a_i| \leq \tau^{i\alpha} \quad \text{and} \quad
	|p| \leq \tau^{i\alpha},
\end{equation}
then this process can be continued. Since $p\neq 0$, we can find the smallest integer $K\in \mathbb{N}$ such that \eqref{est:u-ax_a} hold for $i=K$, but ​\eqref{est:ak-p} does not hold for $i=K$. 
  
Now, let us consider the function 
\begin{equation*}
 	\tilde{u}(x,t)=\frac{u(\tau^{K}x, \rho^{-\gamma} \tau^{2K}t)-a_{K}
	\tau^K x_n}{\tau^{K(1+\alpha)}} ,
\end{equation*}
which is a viscosity solution to
\begin{equation*} 
\left\{\begin{aligned}
	\tilde{u}_t &= |\tilde{\nu} D\tilde{u} + \tilde{p}|^\gamma \widetilde{F}(D^2 \tilde{u}) + \tilde{f} && \text{in } \widetilde{\Omega}\cap Q_1 \\
	\tilde{u}&=\tilde{g} && \text{on } \partial_p \widetilde{\Omega}\cap Q_1,
\end{aligned}\right.
\end{equation*}
where 
\begin{equation*}
	\rho=|p+a_{K} e_n| ,\quad
	\tilde{\nu}=\rho^{-1} \tau^{K\alpha}, \quad
	\tilde{p}=\rho^{-1}(p+a_{K} e_n), 
\end{equation*}
\begin{equation*}
	\tilde{f}(x,t)=\rho^{-\gamma}\tau^{K(1-\alpha)}f(\tau^{K}x, \rho^{-\gamma} \tau^{2K}t), \quad
	\tilde{g}(x,t)=\tau^{-K(1+\alpha)}(g(\tau^{K}x, \rho^{-\gamma} \tau^{2K}t)-a_K\tau^K x_n), 
\end{equation*}
\begin{equation*}
	\widetilde{F}(M)=\tau^{K(1-\alpha)} F(\tau^{K(\alpha-1)}M), \quad\text{and}\quad
	\widetilde{\Omega}=\{(x,t)\mid(\tau^{K}x, \rho^{-\gamma} \tau^{2K}t) \in \Omega \}.
\end{equation*}

Since $|p| \leq \tau^{(K-1)\alpha}$, $p_n=0$, and \eqref{est:u-ax_a} hold for $i=K$, but \eqref{est:ak-p} does not hold for $i=K$, we have
\begin{equation*}
	\tau^{K\alpha} \leq  \rho \leq  (1+C_\star) \tau^{(K-1)\alpha},
\end{equation*}
which implies that
\begin{equation*}
	\frac{\tau^\alpha}{1+C_\star} \leq \nu \leq 1 \quad \text{and}  \quad Q_{\tau^K}^\rho \subset Q_{\tau^K}^{\tau^{K\alpha}}.
\end{equation*}
Combining \eqref{cond:g_bdry} and \eqref{est:u-ax_a} for $i=k$ gives 
\begin{equation*}
	\|\tilde{g}\|_{L^\infty(\partial_p \widetilde\Omega \cap Q_1)} \leq \frac{1}{\tau^{K(1+\alpha)}} \left(  \|g\|_{L^\infty(\partial_p\Omega\cap Q_{\tau^K}^{\rho})} + |a_K| \osc_{Q_{\tau^K}^{\rho}} \partial_p\Omega\right) \leq  \delta \leq \eta.
\end{equation*}
From \eqref{cond:g_bdry}, we have $g\in C_\gamma^{1,\alpha}(0,0)$ with the estimate
\begin{equation*}
	|g(x,t)| \leq \frac{1}{2} \delta (|x|^{1+\alpha} + |t|^{\frac{1+\alpha}{2-\alpha\gamma}})\quad \text{for all } (x,t) \in  \partial_p \Omega \cap Q_1
\end{equation*}
and $\partial_p \Omega \in C_\gamma^{1,\alpha}(0,0)$ with the estimate
\begin{equation*}
	|x_n| \leq \frac{1}{2C_\star} \delta (|x|^{1+\alpha} + |t|^{\frac{1+\alpha}{2-\alpha\gamma}})\quad \text{for all } (x,t) \in  \partial_p \Omega \cap Q_1.
\end{equation*}
Hence, we obtain
\begin{equation*}
	|\tilde{g}(x,t)| \leq \delta (|x|^{1+\alpha} + |t|^{\frac{1+\alpha}{2-\alpha\gamma}})\quad \text{for all } (x,t) \in  \partial_p \widetilde{\Omega} \cap Q_1,
\end{equation*}
which implies that
\begin{equation*}
	\tilde{g}(0,0) = 0 = |D\tilde{g}(0,0)|.
\end{equation*}

Moreover, since 
\begin{equation*}
	|\tilde{p}|=1,\quad
	\|\tilde{u}\|_{L^\infty(\widetilde{\Omega}\cap Q_1)} \leq \eta, \quad \text{and} \quad
	\|\tilde{f}\|_{L^\infty(\widetilde{\Omega}\cap Q_1)} \leq 1,
\end{equation*}
it follows from \Cref{lem:small_bdry} that there exists a contant $\tilde{a}\in\mathbb{R}$ such that 
\begin{equation*}
	|\tilde{u}(x,t) - \tilde{a}x_n| \leq C( |x|^{1+\alpha} + |t|^{\frac{1+\alpha}{2}}) \quad \text{for all } (x,t) \in \Omega\cap Q_1
\end{equation*}
and  $|\tilde{a}| \leq C\eta$. Scaling back to $u$, we obtain the desired conclusion.

Next, let us consider the case when $p=0$. In this case, the sequence $\{a_k\}$ satisfying \eqref{est:u-ax} may be either finite or infinite. If $\{a_k\}$ is finite, then $|a_K| > \tau^{K\alpha}$, and the proof proceeds exactly as in the case  $p\neq 0$. Therefore, it suffices to prove the result in the case where $\{a_k\}$ is an infinite sequence. Then, for any $(x,t) \in \Omega \cap Q_1$, there exists $k\in \mathbb{N}$ such that $(x,t) \in \Omega \cap (Q_{\tau^{k-1}}^{\tau^{(k-1)
\alpha}} \setminus Q_{\tau^k}^{\tau^{k\alpha}})$ and hence 
\begin{align*}
	|u(x,t)| & \leq |u(x,t)-a_k x_n| + |a_k x_n|  \leq \tau^{(k-1)(1+\alpha)} + \tau^{(k-1)\alpha} \tau^{(k-1)} \\
	& \leq \frac{2}{\tau^{1+\alpha}} (|x|^{1+\alpha} + |t|^{\frac{1+\alpha}{2-\alpha\gamma}} ) \leq C (|x|^{1+\alpha} + |t|^{\frac{1+\alpha}{2}} ) ,
\end{align*}
which completes the proof.
\end{proof}
We are now in a position to prove \Cref{thm:main_deg}. The proof will be completed by showing, through a suitable reduction, that the assumptions of \Cref{lem:degenerate} are satisfied.
\begin{proof} [Proof of \Cref{thm:main_deg}]
Without loss of generality, We may assume $(x_0,t_0) =(0,0)$, and set the constant $r=1$ in \Cref{def:function} and \Cref{def:domain}.

Since $g \in C_\gamma^{1,\alpha}(0,0)$ and $\partial_p \Omega \in C_{\gamma}^{1,\alpha}(0,0)$, there exist constants $a,b \in \mathbb{R}$ and a vector $p \in \mathbb{R}^n$ such that $p_n=0$,
\begin{equation} \label{est:c1a_g}
	|g(x,t) - a -bx_n - p\cdot x| \leq \|g\|_{C_\gamma^{1,\alpha}(0,0)} (|x|^{1+\alpha} + |t|^{\frac{1+\alpha}{2-\alpha\gamma}}) \quad \text{for all }(x,t)\in \partial_p \Omega \cap Q_1
\end{equation}
and 
\begin{equation} \label{est:c1a_Om}
	|x_n| \leq [\partial_p\Omega]_{C_\gamma^{1,\alpha}(0,0)}  (|x|^{1+\alpha} + |t|^{\frac{1+\alpha}{2-\alpha\gamma}}) \quad \text{for all } (x,t)\in\partial_p \Omega \cap Q_1.
\end{equation}
Since $|b| \leq  \|g\|_{C_\gamma^{1,\alpha}(0,0)}$, combining \eqref{est:c1a_g} and \eqref{est:c1a_Om} yields
\begin{align*}
	|\tilde{g}(x,t)| &\leq |g(x,t) - a -bx_n - p\cdot x| + |bx_n| \\
	& \leq (1+ [\partial_p\Omega]_{C_\gamma^{1,\alpha}(0,0)} ) \|g\|_{C_\gamma^{1,\alpha}(0,0)}  (|x|^{1+\alpha} + |t|^{\frac{1+\alpha}{2-\alpha\gamma}}) \quad \text{for all } (x,t)\in\partial_p \Omega \cap Q_1.
\end{align*}

Let us consider the function 
\begin{equation*}
 	\tilde{u}(x,t)= \frac{ u(r^{-1}x,r^{-2(1+\gamma)}t)- a -  p\cdot (r^{-1}x)}{r} ,
\end{equation*}
which is a viscosity solution to
\begin{equation*}  
\left\{\begin{aligned}
	\tilde{u}_t &= |D\tilde{u} + \tilde{p}|^\gamma \widetilde{F}(D^2 \tilde{u})  + \tilde{f} && \text{in } \widetilde{\Omega}\cap Q_1\\
	\tilde{u}&=\tilde{g}  && \text{on } \partial_p \Omega\cap Q_1,
\end{aligned}\right.
\end{equation*}
where 
\begin{equation*}
	r= \frac{1}{\varepsilon}\left(1+\|u\|_{L^\infty(\Omega)} + \|f\|_{L^\infty(\Omega)} + [\partial_p\Omega]_{C_\gamma^{1,\alpha}(0,0)} + (1+ [\partial_p\Omega]_{C_\gamma^{1,\alpha}(0,0)} ) \|g\|_{C_\gamma^{1,\alpha}(0,0)}\right)
\end{equation*}
\begin{equation*}
	\tilde{f}(x,t) = \frac{f(r^{-1}x,r^{-2(1+\gamma)}t)}{r^{3+2\gamma}},\quad 
	\tilde{g}(x,t)= \frac{g(r^{-1}x,r^{-2(1+\gamma)}t)-a -p\cdot (r^{-1}x)}{r}.
\end{equation*}
\begin{equation*}
	 \widetilde{F}(M)=r^{-3} F(r^3 M), \quad
	\tilde{p}= r^{-2} p, \quad\text{and}\quad 
	\widetilde{\Omega} =\{(x,t) \mid (r^{-1}x,r^{-2(1+\gamma)}t) \in \Omega\},
\end{equation*}
Let $\delta$ be the constant from \Cref{lem:degenerate}. By choosing $\varepsilon$ sufficiently small so that
\begin{align*}
	\|\tilde{f}\|_{L^\infty(\Omega \cap Q_1)}& \leq \frac{ \|f\|_{L^\infty(\Omega)} }{r^{3+2\gamma}} \leq \varepsilon^{3+2\gamma} \leq \delta, \\
	\|\tilde{g}\|_{C_\gamma^{1,\alpha}(0,0)} &\leq \frac{(1+ [\partial_p\Omega]_{C_\gamma^{1,\alpha}(0,0)} )  \|g\|_{C_\gamma^{1,\alpha}(0,0)}}{r} \leq \varepsilon \leq \frac{1}{2} \delta, \\
	 [\partial_p\widetilde{\Omega} ]_{C_\gamma^{1,\alpha}(0,0)} &\leq \frac{1}{r^\alpha} [\partial_p\Omega ]_{C_\gamma^{1,\alpha}(0,0)} \leq \varepsilon^\alpha \leq \frac{1}{2C_\star} \delta,
\end{align*}
all the assumptions of \Cref{lem:degenerate} are satisfied. Hence, by  \Cref{lem:degenerate}, there exists a constant $\tilde{a}\in \mathbb{R}$ such that
\begin{equation*}
	|\tilde{u}(x,t)-\tilde{a}x_n| \leq C(|x|^{1+\alpha}+|t|^{\frac{1+\alpha}{2}}) \quad \text{for all } (x,t) \in \widetilde{\Omega} \cap Q_1
\end{equation*}
and $|\tilde{a}| \leq C$. Scaling back to $u$, we obtain the desired conclusion.
\end{proof}
The following is the singular version of \Cref{lem:degenerate} , and its proof is exactly the same as that of \Cref{lem:degenerate}. Likewise, the proof of \Cref{thm:main_sing} is exactly the same as that of \Cref{thm:main_deg}. For these reasons, we omit the proofs.
\begin{lemma}  \label{lem:singular}
Assume that $F$ satisfies \textnormal{\ref{F1}}, \textnormal{\ref{F2}}, and $\gamma < 0$. For any given $\alpha \in (0,\bar\alpha)$, let $\eta$ be the constant from \Cref{lem:small_bdry}, and let $\delta$ be the constant from \Cref{lem:gen_1st} associated with this $\alpha$ and $\eta$. Let $u$ be a viscosity solution to 
\begin{equation*}  
\left\{\begin{aligned}
	u_t &= |Du + p|^\gamma F(D^2 u)  + f&& \text{in } \Omega\cap Q_1\\
	u&=g && \text{on } \partial_p \Omega\cap Q_1
\end{aligned}\right.
\end{equation*}
with
\begin{equation*}
	 |p| \leq 1, \quad 
	 p_n=0, \quad
	 \|u\|_{L^\infty(\Omega\cap Q_1)} \leq 1, \quad
	\|f\|_{L^\infty( \Omega\cap Q_1)} \leq \delta,
\end{equation*}
\begin{equation*}  
	 \|g\|_{L^\infty(\partial_p\Omega\cap Q_r)} \leq \frac{1}{2}\delta r^{1+\alpha}, \quad\text{and}\quad
	 \osc_{Q_r} \partial_p\Omega \leq \frac{1}{2C_\star}\delta r^{1+\alpha} \quad \text{for all } r \in (0,1].
\end{equation*}
Then $u \in C_\gamma^{1,\alpha}(0,0)$, that is, there exists a constant $a\in \mathbb{R}$ such that
\begin{equation*}
	|u(x,t)-ax_n| \leq C(|x|^{1+\alpha}+|t|^{\frac{1+\alpha}{2-\alpha\gamma}}) \quad \text{for all } (x,t) \in \Omega \cap Q_1
\end{equation*}
and $|a| \leq C$, where $C>0$ is a constant depending only on $n$, $\lambda$, $\Lambda$, $\gamma$, and $\alpha$.
\end{lemma}


%
%


\begin{thebibliography}{10}

\bibitem{APPT22}
P.~Andrade, D.~Pellegrino, E.~Pimentel, and E.~Teixeira.
\newblock {$C^1$}-regularity for degenerate diffusion equations.
\newblock {\em Adv. Math.}, 409(part B):Paper No. 108667, 34, 2022.

\bibitem{ART15}
D.~J. Ara\'{u}jo, G.~Ricarte, and E.~Teixeira.
\newblock Geometric gradient estimates for solutions to degenerate elliptic
  equations.
\newblock {\em Calc. Var. Partial Differential Equations}, 53(3-4):605--625,
  2015.

\bibitem{AS23}
D.~J. Ara\'{u}jo and B.~Sirakov.
\newblock Sharp boundary and global regularity for degenerate fully nonlinear
  elliptic equations.
\newblock {\em J. Math. Pures Appl. (9)}, 169:138--154, 2023.

\bibitem{ACP11}
R.~Argiolas, F.~Charro, and I.~Peral.
\newblock On the {A}leksandrov-{B}akel'man-{P}ucci estimate for some elliptic
  and parabolic nonlinear operators.
\newblock {\em Arch. Ration. Mech. Anal.}, 202(3):875--917, 2011.

\bibitem{AJ25}
V.~Arya and V.~Julin.
\newblock Harnack inequality for degenerate fully nonlinear parabolic
  equations.
\newblock {\em arXiv preprint arXiv:2506.10608}, 2025.

\bibitem{Att20}
A.~Attouchi.
\newblock Local regularity for quasi-linear parabolic equations in
  non-divergence form.
\newblock {\em Nonlinear Anal.}, 199:112051, 28, 2020.

\bibitem{AP18}
A.~Attouchi and M.~Parviainen.
\newblock H\"older regularity for the gradient of the inhomogeneous parabolic
  normalized {$p$}-{L}aplacian.
\newblock {\em Commun. Contemp. Math.}, 20(4):1750035, 27, 2018.

\bibitem{AR20}
A.~Attouchi and E.~Ruosteenoja.
\newblock Gradient regularity for a singular parabolic equation in
  non-divergence form.
\newblock {\em Discrete Contin. Dyn. Syst.}, 40(10):5955--5972, 2020.

\bibitem{BBLL23}
S.~Baasandorj, S.-S. Byun, K.-A. Lee, and S.-C. Lee.
\newblock {$C^{1,\alpha}$}-regularity for a class of degenerate/singular fully
  non-linear elliptic equations.
\newblock {\em Interfaces and Free Boundaries}, 26(2):189--215, 2024.

\bibitem{BBLL22b}
S.~Baasandorj, S.-S. Byun, K.-A. Lee, and S.-C. Lee.
\newblock Global regularity results for a class of singular/degenerate fully
  nonlinear elliptic equations.
\newblock {\em Mathematische Zeitschrift}, 306(1):1, 2024.

\bibitem{BJdSRR23}
E.~C. Bezerra~J\'{u}nior, J.~V. da~Silva, G.~C. Rampasso, and G.~C. Ricarte.
\newblock Global regularity for a class of fully nonlinear {PDE}s with
  unbalanced variable degeneracy.
\newblock {\em J. Lond. Math. Soc. (2)}, 108(2):622--665, 2023.

\bibitem{BJdSR23}
E.~C. Bezerra~J\'{u}nior, J.~V. da~Silva, and G.~C. Ricarte.
\newblock Fully nonlinear singularly perturbed models with non-homogeneous
  degeneracy.
\newblock {\em Rev. Mat. Iberoam.}, 39(1):123--164, 2023.

\bibitem{BD14}
I.~Birindelli and F.~Demengel.
\newblock {${C}^{1,\beta}$} regularity for {D}irichlet problems associated to
  fully nonlinear degenerate elliptic equations.
\newblock {\em ESAIM Control Optim. Calc. Var.}, 20(4):1009--1024, 2014.

\bibitem{BPRT20}
A.~C. Bronzi, E.~A. Pimentel, G.~C. Rampasso, and E.~V. Teixeira.
\newblock Regularity of solutions to a class of variable-exponent fully
  nonlinear elliptic equations.
\newblock {\em J. Funct. Anal.}, 279(12):108781, 31, 2020.

\bibitem{BKO25b}
S.-S. Byun, H.~Kim, and J.~Oh.
\newblock {$C^{1,\alpha }$} regularity for degenerate fully nonlinear elliptic
  equations with oblique boundary conditions on {$C^1$} domains.
\newblock {\em Calc. Var. Partial Differential Equations}, 64(5):Paper No. 174,
  20, 2025.

\bibitem{BKO25}
S.-S. Byun, H.~Kim, and J.~Oh.
\newblock Interior {$W^{2,\delta}$} type estimates for degenerate fully
  nonlinear elliptic equations with {$L^{n}$} data.
\newblock {\em J. Funct. Anal.}, 289(6):Paper No. 111007, 37, 2025.

\bibitem{CIL92}
M.~G. Crandall, H.~Ishii, and P.-L. Lions.
\newblock User's guide to viscosity solutions of second order partial
  differential equations.
\newblock {\em Bull. Amer. Math. Soc. (N.S.)}, 27(1):1--67, 1992.

\bibitem{dSRRV23}
J.~V. da~Silva, G.~C. Rampasso, G.~C. Ricarte, and H.~A. Vivas.
\newblock Free boundary regularity for a class of one-phase problems with
  non-homogeneous degeneracy.
\newblock {\em Israel J. Math.}, 254(1):155--200, 2023.

\bibitem{dSR20}
J.~V. da~Silva and G.~C. Ricarte.
\newblock Geometric gradient estimates for fully nonlinear models with
  non-homogeneous degeneracy and applications.
\newblock {\em Calc. Var. Partial Differential Equations}, 59(5):Paper No. 161,
  33, 2020.

\bibitem{dSV21a}
J.~V. da~Silva and H.~Vivas.
\newblock The obstacle problem for a class of degenerate fully nonlinear
  operators.
\newblock {\em Rev. Mat. Iberoam.}, 37(5):1991--2020, 2021.

\bibitem{dSV21b}
J.~V. da~Silva and H.~Vivas.
\newblock Sharp regularity for degenerate obstacle type problems: a geometric
  approach.
\newblock {\em Discrete Contin. Dyn. Syst.}, 41(3):1359--1385, 2021.

\bibitem{Dem11}
F.~Demengel.
\newblock Existence's results for parabolic problems related to fully non
  linear operators degenerate or singular.
\newblock {\em Potential Anal.}, 35(1):1--38, 2011.

\bibitem{Del11}
F.~Demengel.
\newblock Existence's results for parabolic problems related to fully non
  linear operators degenerate or singular.
\newblock {\em Potential Anal.}, 35(1):1--38, 2011.

\bibitem{FRZ21}
Y.~Fang, V.~D. R\u{a}dulescu, and C.~Zhang.
\newblock Regularity of solutions to degenerate fully nonlinear elliptic
  equations with variable exponent.
\newblock {\em Bull. Lond. Math. Soc.}, 53(6):1863--1878, 2021.

\bibitem{FZ21}
Y.~Fang and C.~Zhang.
\newblock Gradient {H}\"{o}lder regularity for parabolic normalized {$p( x,
  t)$}-{L}aplace equation.
\newblock {\em J. Differential Equations}, 295:211--232, 2021.

\bibitem{FZ23}
Y.~Fang and C.~Zhang.
\newblock Regularity for quasi-linear parabolic equations with nonhomogeneous
  degeneracy or singularity.
\newblock {\em Calc. Var. Partial Differential Equations}, 62(1):Paper No. 2,
  46, 2023.

\bibitem{FPS25}
Y.~Feng, M.~Parviainen, and S.~Sarsa.
\newblock Second order {S}obolev regularity results for the generalized
  {$p$}-parabolic equation.
\newblock {\em J. Funct. Anal.}, 288(5):Paper No. 110799, 27, 2025.

\bibitem{IJS19}
C.~Imbert, T.~Jin, and L.~Silvestre.
\newblock H\"{o}lder gradient estimates for a class of singular or degenerate
  parabolic equations.
\newblock {\em Adv. Nonlinear Anal.}, 8(1):845--867, 2019.

\bibitem{IS13}
C.~Imbert and L.~Silvestre.
\newblock {$C^{1,\alpha}$} regularity of solutions of some degenerate fully
  non-linear elliptic equations.
\newblock {\em Adv. Math.}, 233:196--206, 2013.

\bibitem{IS13b}
C.~Imbert and L.~Silvestre.
\newblock An introduction to fully nonlinear parabolic equations.
\newblock In {\em An introduction to the {K}\"{a}hler-{R}icci flow}, volume
  2086 of {\em Lecture Notes in Math.}, pages 7--88. Springer, Cham, 2013.

\bibitem{JS17}
T.~Jin and L.~Silvestre.
\newblock H\"{o}lder gradient estimates for parabolic homogeneous
  {$p$}-{L}aplacian equations.
\newblock {\em J. Math. Pures Appl. (9)}, 108(1):63--87, 2017.

\bibitem{KLY23}
T.~Kim, K.-A. Lee, and H.~Yun.
\newblock Higher regularity up to boundary for degenerate parabolic equations.
\newblock {\em J. Differential Equations}, 348:223--277, 2023.

\bibitem{KLY24}
T.~Kim, K.-A. Lee, and H.~Yun.
\newblock Generalized {S}chauder theory and its application to
  degenerate/singular parabolic equations.
\newblock {\em Math. Ann.}, 390(4):6049--6109, 2024.

\bibitem{LLY24}
K.-A. Lee, S.-C. Lee, and H.~Yun.
\newblock {$C^{1,\alpha}$}-regularity for solutions of degenerate/singular
  fully nonlinear parabolic equations.
\newblock {\em J. Math. Pures Appl. (9)}, 181:152--189, 2024.

\bibitem{LY25}
K.-A. Lee and H.~Yun.
\newblock Boundary regularity for viscosity solutions of fully nonlinear
  degenerate/singular parabolic equations.
\newblock {\em Calc. Var. Partial Differential Equations}, 64(1):Paper No. 25,
  32, 2025.

\bibitem{LLYZ25b}
S.-C. Lee, Y.~Lian, H.~Yun, and K.~Zhang.
\newblock Boundary {H}{\"o}lder gradient estimates for parabolic {$p$}-laplace
  type equations, 2025.

\bibitem{LLYZ25a}
S.-C. Lee, Y.~Lian, H.~Yun, and K.~Zhang.
\newblock Time derivative estimates for parabolic $p$-laplace equations and
  applications to optimal regularity, 2025.

\bibitem{LZ20}
Y.~Lian and K.~Zhang.
\newblock Boundary pointwise {$C^{1,\alpha}$} and {$C^{2,\alpha}$} regularity
  for fully nonlinear elliptic equations.
\newblock {\em J. Differential Equations}, 269(2):1172--1191, 2020.

\bibitem{LZ22}
Y.~Lian and K.~Zhang.
\newblock Boundary pointwise regularity for fully nonlinear parabolic equations
  and an application to regularity of free boundaries.
\newblock {\em arXiv preprint arXiv:2208.01194}, 2022.

\bibitem{MK97}
M.~Ohnuma and K.~Sato.
\newblock Singular degenerate parabolic equations with applications to the
  {$p$}-{L}aplace diffusion equation.
\newblock {\em Comm. Partial Differential Equations}, 22(3-4):381--411, 1997.

\bibitem{PV20}
M.~Parviainen and J.~L. V\'azquez.
\newblock Equivalence between radial solutions of different parabolic
  gradient-diffusion equations and applications.
\newblock {\em Ann. Sc. Norm. Super. Pisa Cl. Sci. (5)}, 21:303--359, 2020.

\bibitem{Wan92}
L.~Wang.
\newblock On the regularity theory of fully nonlinear parabolic equations. {I}.
\newblock {\em Comm. Pure Appl. Math.}, 45(1):27--76, 1992.

\end{thebibliography}

\end{document}